\documentclass[11pt, reqno]{amsart}
\usepackage[colorlinks=true,pagebackref,hyperindex]{hyperref}
\usepackage{mathrsfs} 
\usepackage[all]{xy}
\usepackage{color}
\usepackage{amsfonts}
\usepackage{amscd}
\usepackage{amssymb,latexsym,amsmath,amscd, graphicx} %Mircea's
\usepackage{enumerate}
\usepackage{bbm}

\definecolor{darkgreen}{rgb}{0.0, 0.4, 0.0}

\definecolor{cyan}{cmyk}{1,0,0,0}

\newcommand{\cb}{\color{blue}}
\newcommand{\cm}{\color{magenta}}

\newcommand{\bdg}{\begin{dg}}
\newcommand{\edg}{\end{dg}}

\newcommand{\cre}{\color{red}}

\newtheorem{tm}{Theorem}[subsection]
\newtheorem{lm}[tm]{Lemma}
\newtheorem{pr}[tm]{Proposition}
\newtheorem{rmk}[tm]{Remark}
\newtheorem{cor}[tm]{Corollary}
\newtheorem{ex}[tm]{Example}

\newtheorem{fact}[tm]{Fact}
\newtheorem{??}[tm]{Question}
\newtheorem{defi}[tm]{Definition}

\newtheorem{setup}[tm]{Set-up}

\newtheorem*{thmA}{Theorem A}
\newtheorem*{thmB}{Theorem B}
\newtheorem*{thmBp}{Theorem B$'$}

\newcommand{\ben}{\begin{enumerate}}
\newcommand{\een}{\end{enumerate}}
\newcommand{\bit}{\begin{itemize}}
\newcommand{\eit}{\end{itemize}}
\newcommand{\beq}{\begin{equation}}
\newcommand{\eeq}{\end{equation}}
\newcommand{\la}{\label}

\newcommand\ci{\cite}

\font\tenmsb=msbm10
\font\sevenmsb=msbm7
\font\fivemsb=msbm5
\newfam\msbfam
\textfont\msbfam=\tenmsb
\scriptfont\msbfam=\sevenmsb
\scriptscriptfont\msbfam=\fivemsb

\font\teneufm=eufm10
\font\seveneufm=eufm7
\font\fiveeufm=eufm5
\newfam\eufmfam
\textfont\eufmfam=\teneufm
\scriptfont\eufmfam=\seveneufm
\scriptscriptfont\eufmfam=\fiveeufm
\def\frak#1{{\fam\eufmfam\relax#1}}

\newcommand{\im}{ \hbox{\rm Im} }
\newcommand{\ke}{ \hbox{\rm Ker} }

\newcommand\rat{{\mathbb Q}}

\newcommand\oql{\overline{\mathbb Q}_\ell}
\newcommand\comp{{\mathbb C}}

\newcommand\zed{{\mathbb Z}}
\newcommand\nat{{\mathbb N}}

\newcommand\pn[1]{{\mathbb P}^{#1}}

\newcommand\s{\sigma}

\newcommand\e{\epsilon}

\newcommand{\w}[1]{\widetilde{#1}}
\newcommand{\wh}[1]{\widehat{#1}}
\newcommand{\ov}[1]{\overline{#1}}

\newcommand{\m}[1]{\mathcal{#1}}
\newcommand{\ms}[1]{\mathscr{#1}}

\newcommand{\om}{\Omega}

\newcommand{\ph}{^p\!H}

\newcommand{\ogm}{\widetilde{M}}
\newcommand{\ogmm}{M}
\newcommand{\ogn}{N}
\newcommand{\ogb}{B}
\newcommand{\pogm}{\w{m}}
\newcommand{\pogn}{n}
\newcommand{\bogm}{b}
\newcommand{\pogmm}{m}
\newcommand{\pogmmp}{m'}
\newcommand{\ogmp}{M'}
\newcommand{\ogbp}{B'}
\newcommand{\pogmp}{p}
\newcommand{\p}{P}
\newcommand{\ppr}{P'}
\newcommand{\pogmpr}{p'}
\newcommand{\ks}{S}
\newcommand{\ogcred}{{C'}}
\newcommand{\ogc}{C}

\newcommand{\mc}{\mathcal}
\DeclareMathOperator{\Hom}{Hom}
\DeclareMathOperator{\Ext}{Ext}
\DeclareMathOperator{\Sym}{Sym}

\newcommand{\Shext}{\mc{E}xt}
\newcommand{\Shhom}{\mc{H}om}

\newcommand{\C}{\mathbb C}

\title{The Hodge numbers of O'Grady $10$ via Ng\^o strings}
\author{
Mark Andrea A.  de Cataldo, 
Antonio Rapagnetta, Giulia Sacc\`a, }
\address{Mark Andrea A. de Cataldo, Stony Brook University}
\email{mark.decataldo@stonybrook.edu}

\address{Antonio Rapagnetta, Universit\`a di Roma 2}
\email{rapagnet@axp.mat.uniroma2.it}

\address{Giulia Sacc\`a, Columbia University}
\email{gs3032@columbia.edu}

\begin{document}

\begin{abstract} 
We determine the Hodge numbers of the hyper-K\"ahler manifold known as O'Grady $10$ by studying some related modular Lagrangian fibrations by means of N\^go strings, which we introduce via a refinement of the Ng\^o Support Theorem. 
%We determine the Hodge numbers of the hyper-K\"ahler manifold known as O'Grady $10$
%by studying some related modular Lagrangian fibrations by means of a refinement  of the
%Ng\^o Support Theorem. 
\end{abstract}

\maketitle

\tableofcontents

\section{Introduction}\la{intro}

Irreducible holomorphic symplectic manifolds are simply connected compact K\"ahler manifolds with a unique up to scalar holomorphic two-form, which is furthermore symplectic. They are hyper-K\"ahler manifolds and are one of the three building blocks for compact K\"ahler manifolds with trivial first Chern class \cite{Beauville83, Bogomolov}.  In every even complex dimension, there are two series of known examples 
of irreducible holomorphic symplectic manifolds: those that are deformation equivalent to the Hilbert schemes of $n$-points on a K3 surface, called of K3$^{[n]}$-type, and those that are deformation equivalent to generalized Kummer varieties, called of generalized Kummer type \cite{Beauville83, Fujiki}.
In addition to these, there are two sporadic deformation classes occurring in dimension $6$ and $10$ discovered by  O'Grady \cite{OGrady1,OGrady2} as symplectic resolutions of a singular moduli spaces of sheaves on a K3 surface (in the case of the $10$--dimensional example) and  on an abelian surface (in the case of the $6$--dimensional example). 
%They are usually referred to as the exceptional examples of irreducible holomorphic symplectic manifolds.
The class of irreducible holomorphic symplectic manifolds deformation equivalent to the $10$ dimensional example is denoted by $OG10$ and manifolds in this deformation class are said to be of $OG10$--type.  Similarly, for the $6$--dimensional example.

The Betti and Hodge numbers of irreducible holomorphic symplectic manifolds of  K3$^{[n]}$- type and of generalized Kummer type are known thanks to G\"ottsche's Formula \cite{Gottsche, Gottsche Sorgel} (see also 
\ci{dec mig douady, dec mig chow}). The Betti and Hodge numbers of the deformation class $OG6$ were calculated in \cite{MRS}.
The purpose of this paper is to compute the Betti and Hodge numbers of the remaining known deformation class, namely $OG10$. Previously, the only known invariants of this class manifolds were the second Betti number  $b_2=24$ \cite{ra08}  and the Euler number $e=176,904$  \cite{Mozgovoy, HLS}. Our first main result is the following theorem:

\begin{thmA}\la{bno10}
The odd Betti numbers of the ten-dimensional irreducible symplectic manifolds of 
 $OG10$-type are zero and the even ones are:
\[
\begin{tabular}{|c|c|c|c|c|c|}
\hline
$b_0=1$ & $b_2= 24$ & $b_4= 300$ & $b_6= 2899$ & $b_8= 22150$ & $b_{10}= 126156$.\\
\hline
\end{tabular}
\]
The relevant part of the Hodge diamond  is  as follows:

\iffalse
\[\begin{array}{ccccccccccc}
 & &   &    &  &  1 & &  & & &\\
  & &   & &  1  &    22 & 1 & & & &\\
 &   & &  1 &  22  &  254 & 22 & 1 & & & \\
& & 1 & 23 &   276 & 2299 & 276 & 23 & 1 & & \\
&    1 & 22 & 276 &  2531 & 16490 &  2531 & 276  &  22 & 1 & \\
 1 & 22 & 254 & 2299 & 16409 & 88024 & 16409 & 2299  & 254 &22 & 1. \\
\end{array}
\]
\fi

\[
\begin{array}{cccccc}
 & &   &    &  &  1 \\
  & &   & &  1  &    22\\
 &   & &  1 &  22  &  254  \\
& & 1 & 23 &   276 & 2299  \\
&    1 & 22 & 276 &  2531 & 16490 \\
 1 & 22 & 254 & 2299 & 16490 & 88024. \\
\end{array}
\]

\end{thmA}

We prove the theorem by comparing the cohomology of an irreducible symplectic variety of $OG10$-type with that of an irreducible symplectic variety of K3$^{[5]}$-type. This is done by considering two 
Lagrangian fibered irreducible holomorphic symplectic manifolds, denoted $\ogm \to \pn{5}$ and $N \to \mathbb P^5$, where the first is of $OG10$-type and the second is of K3$^{[5]}$-type. These are chosen so that over a dense open subset of the base the two fibrations are torsors over the same group scheme. In this setting, the crucial ingredient is Ng\^o's Support Theorem, which allows us to exploit the Lagrangian fibration structure and to express the difference between the cohomology of $\ogm$ and that of $\ogn$ in terms of the cohomology of lower dimensional irreducible symplectic manifolds.

%Before describing how we apply Ng\^o's Support Theorem, 
Let us describe the two Lagrangian fibrations.
Let $(\ks, H)$ be a general polarized  K3 surface  of genus  $2$ and let $\mc C \to \pn{5}$ be the universal family of the genus $5$ linear system $|2H|\simeq \pn{5}$.  The relative compactified Jacobian $M \to \mathbb P^5$ of a given even degree of the linear system $|2H|$ is a singular projective symplectic variety. We let $\w M \to M$ be the blowup of its singular locus (with its reduced induced structure). This is an irreducible holomorphic symplectic manifold of $OG10$-type and $\w M \to M$ is a symplectic resolution.  If instead we consider a relative compactified Jacobian of a given odd degree, we get a smooth projective irreducible holomorphic symplectic manifold  $N \to \mathbb P^5$, deformation equivalent to the Hilbert scheme of $5$ points on $\ks$. The degree $0$ relative Picard variety $\p \to \mathbb P^5$ of the linear system is a group scheme over $\mathbb P^5$, and it acts naturally and fiberwise on both spaces. It is important to note that there are dense open subsets of $\mathbb P^5$ over which the fibrations are torsors over (the restriction of) the group scheme $\p/\mathbb P^5$.
Note that  the Lagrangian fibration $\w M \to \mathbb P^5$ is the same used by Mozgovoy in his determination \cite{Mozgovoy}  of the Euler number of $OG10$. It is a degeneration of the Intermediate Jacobian fibration \cite{LSV, KLSV}, which was used in \cite{HLS}.

%over a dense open subset of the base, there are dense open subsets of both $\ogm$ and $\ogn$ that are torsors under (the base change of)

The starting point of our analysis is that  each of the two Lagrangian fibrations, together with the natural action of $\p$, gives rise to what Ng\^o calls a $\delta$-regular weak Abelian fibration (see \S\ref{sonowaf}). Before showing this, we prove in Section \S\ref{subsection symplectic} some general results on weak Abelian fibrations arising from certain Lagrangian fibrations. 
%These results are a general duality statement (Lemma \ref{da1}) and the $\delta$-regularity property (Proposition \ref{deltaregularity})  and may be of independent interest.

% Section \S\ref{subsection symplectic}  contains Lemma \ref{da1} (a general Duality result for the infinitesimal action associated with certain Lagrangian fibrations) and Proposition \ref{deltaregularity} ($\delta$-regularity of these Lagrangian fibrations), which may be of general interest. 
% 
% 
% Part of the data constituting a weak Abelian fibration is a group scheme over $B$ acting 
%fiberwise on the fibration.   The group scheme is the same for both the Lagrangian fibrations $\ogm, \ogn \to \ogb$, and it is
%an open subvariety $\p$ of the relative compactified Jacobian of degree $0$ for the family of curves in the linear system
%$|2H|$ on the K3 surface $S$.
% %moduli space $M_{(0,2,-4})$.   
%An important observation is that both $\ogm$ and $\ogn$ 
%contain  open  and dense subvarieties which are $\p$-torsors 
%over a large open subset of $\ogb$, which  coincides with $\ogb$ in the case of $\ogm$, but is strictly contained in $\ogb$
%in the case of $\ogn.$

 Ng\^o's Support Theorem is  a topological result concerning the Decomposition Theorem for  $\delta$-regular weak Abelian fibrations. 
 Roughly speaking,  it states that if in the Decomposition Theorem  for the fibration a subvariety of the target is the support of a direct summand of the derived direct image of the constant sheaf, then such a subvariety  (called a support for the fibration)  is the support of a non trivial 
  direct summand of the \emph{top degree} direct image sheaf.  What makes this theorem powerful is the possibility of restricting the set of potential supports by  only using the top degree direct image sheaf.

To carry out our cohomological computation, we first need to refine the statement of  Ng\^o Support Theorem. Given a subvariety of the base that is a support for the Decomposition Theorem of a $\delta$-regular weak Abelian fibrations, the refinement consists in identifying  the direct summands supported on that subvariety. We show that they can be expressed in terms of contributions coming from lower dimensional $\delta$-regular weak abelian fibrations that naturally appear in the picture. We call these direct summands Ng\^o Strings (see Theorem \ref{nstabis} and Definition \ref{defngostrbis} ). We believe this is a result of independent interest. In addition to this, we need to make sure that a given  Ng\^o String   contributes the same pure Hodge structure, regardless
of which weak Abelian fibration it appears in  (cf. Remark \ref{nondipende}).  Since we are not aware of a reference, we offer  some context and a sketch of proof, which involves M. Saito's mixed Hodge modules. This is done in Appendix \ref{vnstbis}.

%As a side remark, we mention that  Ng\^o's support theorem is used in the forthcoming \ci{dhm} in a different fashion from how it is used in this paper:
%in this paper it  is used  as a tool  in determining the supports, and hence the resulting Ng\^o strings;
%in \ci{dhm} it is used to determine the Ng\^o strings  for the $GL_n$-Hitchin fibration  over the locus of reduced spectral curves, {\it after} the supports
%have been determined by completely different means.

%{\color{red} Questa frase \`e oscura. La togliamo? } We mention that the Ng\^o support theorem is used in a different fashion from how it is used in this paper,  in the forthcoming \ci{dhm} to determine the shape of the decomposition theorem for the $GL_n$-Hitchin fibration  over the locus of reduced spectral curves, once the supports
%have been determined by other means.

In order to identify which subvarieties of the base appear as supports, we need to study the top-degree higher direct images  for the Lagrangian fibrations $\ogm/ \mathbb P^5$ and $\ogn/ \mathbb P^5$. The determination of these top degree sheaves is the subject of  the key Proposition \ref{rtop}, which is based on the detailed analysis of the moduli spaces of semi-stable sheaves over the locus parametrizing non-reduced curves in the linear system $|2H|$ (see \S\ref{sect irr comp}-\ref{secnonred}).

In \S\ref{nonon} we identify the Ng\^o Strings of the two fibrations, i.e., we determine the direct summands appearing in the Decomposition Theorem for 
$\ogm/ \mathbb P^5$ and $\ogn/ \mathbb P^5$. This is the content of Proposition \ref{dtwe}.  
Actually, this is done up to some indeterminacy which is
present  in both fibrations. Remarkably, these two  indeterminacies cancel out when comparing the Hodge numbers of the two fibrations, and we are thus able to deduce our main Theorem A in  \S\ref{pmt}.

The following observations are key for successfully carrying out this strategy. First of all, the full support Ng\^o Strings (i.e., the direct summands whose support is the whole $ \mathbb P^5$) are nothing but natural extensions to $\mathbb P^5$ of the local systems associated to the restriction of the fibration to the smooth locus. By our choice of  $\ogm/ \mathbb P^5$ and  $\ogn/ \mathbb P^5$, these local systems are the same for the two fibrations and thus the contribution  to the cohomology of $\ogm$ and of $\ogn$  of the full support direct summands are the same. Secondly, we show that the Ng\^o Strings of the two fibrations that are supported on proper subvarieties of $ \mathbb P^5$  appear also in the Decomposition Theorem for certain lower-dimensional Lagrangian fibrations associated with other moduli spaces of sheaves on the same $K3$ surface $\ks$. The geometry of these is considerably simpler and their Hodge numbers are known. 
It follows that we may express the difference between the Hodge numbers of $\ogm$ and of $\ogn$ in terms of the (known) Hodge numbers of lower-dimensional irreducible holomorphic symplectic manifolds.

Upgrading  these cohomological comparisons to the appropriate Grothendieck group, we express the ``difference'' between the pure Hodge structure of  $\ogm$    and the pure Hodge structure of $N$ in terms of the pure Hodge structures of lower dimensional moduli spaces of sheaves on the degree $2$ K3 surface $S$. Thanks to Goettsche's formula, these can in turn be expressed as functions of the Hodge structure of $S$.  This gives Theorem B below, which determines the pure Hodge structure of the manifolds $\ogm$ (which define a codimension $3$ locally closed subset in the moduli space of varieties of $OG10$--type)  in terms of the pure Hodge structure of $\ks$.

\begin{thmB}\la{ghost}
Let $(\ks,H)$ be a  general polarized  $K3$ surface of genus $2$ and let  $\ogm$ and $\ogn$ be the 
 irreducible  symplectic manifolds introduced above. 
Let $\mathbb S_{(-)}$ be the Schur functors  \cite[Ch 6]{Fulton-Harris}. Let $\langle \bullet \rangle := [-2 \bullet]
(-\bullet)$. By abuse of notation, for a projective manifold $X$, 
we denote by  $X$ also the graded rational polarizable pure Hodge structure $H^*(X,\rat)$. Then we have isomorphisms:
\beq\la{motive}
\begin{aligned}
\ogm = \,\, & \ks^{(5)} \oplus \left[\ks^{(4)} \langle -1 \rangle \right]^{\oplus 2}
\oplus \mathbb S_{(2,2)} (\ks) \langle -1 \rangle \oplus \left[ \ks^{(3)} \langle -2 \rangle \right]^{\oplus 2} \oplus \\
&\oplus  \left[ \mathbb S_{(2,1)} (\ks) \langle -2 \rangle \right]^{\oplus 2}
\oplus \left[\ks \otimes \ks \right]  \langle -3 \rangle \oplus \left[ \ks^{(2)} \langle -3 \rangle \right]^{\oplus 3} \oplus \left[ \ks \langle -4 \rangle\right]^{\oplus 2};
\end{aligned}
\eeq
\beq\la{motiven}
\xymatrix{
N=  S^{(5)} \oplus \left[S^{(3)}\otimes S \right] \langle -1 \rangle 
 \oplus \left[S \otimes S^{(2)}\right]^{\oplus 2} \langle -2 \rangle
\oplus \left[S^2\right]^{\oplus 2} \langle -3 \rangle \oplus S\langle -4 \rangle.
}
\eeq

\end{thmB}

%Notice that the Hodge structures of the  manifolds $\ogm$ (which depend on the integer $k$)   are independent of $k$; similarly, for $\ogn$ and  $k'$ (cf. Lemma \ref{lemma birazionali} and  Remark \ref{rovo}). 
Note that (2) above is well-known and it is only included for completeness.  Theorem B  is here formulated only for a general genus $2$ K3 surface. In Section \ref{pmtB} we give the  slightly more general statement (Theorem B$'$ on page \pageref{thmbprimo}), as well as its proof.

%%STRUCTURE OF THE PAPER

%%Let us outline the structure of the paper.
%%Section \S\ref{things} discusses $\delta$-regular weak Abelian fibrations and shows that, under some natural conditions, Lagrangian fibrations are examples of such. In Section \S\ref{mo10}  we introduce the two Lagrangian fibrations $\ogm, \ogn \to  \mathbb P^5$ and show that they are examples of  $\delta$-regular weak Abelian fibrations. Before showing this, we need to recall some well-known facts on moduli spaces of pure dimension one sheaves on K3 surfaces.
%%Section \S\ref{sec r10}  is devoted to studying the irreducible components of the fibers of the Lagrangian fibrations, with the aim of determining their top degree direct image sheaves (Proposition  \ref{rtop}).  This section is rather long and technical because it requires studying moduli spaces of sheaves on non-reduced (possibly singular) curves in the linear system $|2H|= \mathbb P^5$.
%%In Section \S\ref{nonon} we determine  the Decomposition Theorem for $\ogm/ \mathbb P^5$ and $\ogn/ \mathbb P^5$ as an application of the following results: the  refined Support Theorem
%%in \S\ref{vnstbis};  the determination of the top degree direct image sheaves carried out
%%in \S\ref{sec r10};  the determination of the Decomposition Theorem for certain lower dimensional Lagrangian fibrations. Finally,  the proofs of Theorems A, B and B$'$ are given in Section \ref{pfs ab}.
%%Section \S\ref{vnstbis} is the Appendix, which is devoted to generalizing and completing  Ng\^o's support Theorem so that it fits the needs of this paper.

After our paper appeared,  the paper \cite{gklr} computed the Hodge structure of an irreducible holomorphic symplectic variety of $OG10$-type (in fact, the cohomology is even determined as a module over the so-called Loojenga-Lunts-Verbitsky algebra  \cite{lo-lu,Verbitsky96}), conditionally to assuming that the odd Betti numbers are zero. Then the paper \cite{ffz} appeared, where it is shown that odd Betti numbers of $OG10$ are trivial.  Our methods and those of  \cite{gklr} are completely different and we view  the two papers as complementary.

%By building on Theorem A (or on  the main result of \ci{ffz}), \ci{gklr} determine the Hodge structure of every member of $OG10$ in terms of the
%Hodge structure of its second cohomology group.  In fact, they even determine the cohomology as a module for the so-called LLV Lie algebra. Theorems B and B$'$ are concerned with the  members of $OG10$ that arise starting  from a degree two $K3$ surface;
%on the other hand, they express the Hodge structure more explicitly, namely  as a function of the Hodge structure of the $K3$ surface.

%DA TOGLIERE? In addition to being subject to the usual constraints and symmetries common to all compact K\"ahler manifolds,
%the Betti and Hodge numbers of compact hyper-K\"ahler manifolds satisfy some additional properties.
%For example, the Betti numbers of an irreducible holomorphic symplectic manifold of dimension $2n$ satisfy Salamon's formula \cite{Salamon}:
%$2\sum_{i=1}^{2n} (-1)^i(3i^2-n)b_{2n-i}=nb_{2n}.
%$
%One also has M. Verbitsky's theorem \cite[Theorem 1.4]{Verbitsky96} that the  structure Lie algebra in the sense of Loojenga-Lunts of a compact  hyper-K\"ahler manifold
%is $\frak so (4,b_2 -2)$,  where $b_2$ is the second Betti number.

\noindent {\bf Acknowledgments.} 
%\subsection*{Acknowledgments.} 
We would like to thank E. Arbarello, K. Hulek, J. de Jong, R. Laza, E. Markman,   L. Maxim, M. McLean, 
A. Toth, D. Varolin for conversations related to the topic of this paper, as well as I. Grosse-Brauckmann for pointing out several typos.
M. A. de Cataldo, who has been  partially supported by  NSF grants DMS 1600515 and  1901975, would like to thank the Max Planck Institute for Mathematics in Bonn and  the Freiburg Research Institute for Advanced Studies  for the perfect working conditions; the research leading to these results has received funding from the People Programme (Marie Curie Actions) of the European Union's Seventh Framework Programme (FP7/2007-2013) under REA grant agreement n. [609305].  A. Rapagnetta  acknowledges the MIUR Excellence Department Project awarded to the Department of Mathematics, University of Rome Tor Vergata, CUP E83C18000100006.
G. Sacc\`a acknowledges partial support from NSF grant DMS-1801818.
We would also like to thank the anonymous referee for many suggestions that have greatly improved the exposition of the paper.

\section{Notation and preliminaries}\la{things}$\;$

\subsection{General notation and facts}\la{gennot}$\;$

We work over the field of complex numbers.  A variety is a separated scheme of finite type over $\comp.$
Unless otherwise stated, by point we mean a closed point.

We work  with the following two categories attached to a variety $S$: the constructible bounded derived category $D^b(S,\rat)$, whose objects are bounded
complexes of  sheaves on $S$ of rational vector spaces whose cohomology complexes are constructible
with algebraic strata (see \ci{bams} for the basics and references);   M. Saito's bounded derived category  $D^b MHM_{alg}(S)$ of 
algebraic mixed Hodge modules with rational structure, which is endowed with the formalism of weights; see
  \ci{Saito 1990} for the foundations, and  \ci{Schnell on Saito} especially \S22 and \S23, which contains the basics and references. We call the objects of $D^b MHM_{alg}(S)$  simply ``objects."

If in Theorem A one is interested only in  Betti numbers, then it is enough to work with $D^b(-,\rat)$. If one is interested in the Hodge numbers, then it seems necessary to work with $D^b MHM_{alg}(-)$.

There is the natural functor ${\rm rat}:  D^b MHM_{alg}(S) \to D^b(S, \rat)$. The functor ${\rm rat}$
is compatible with the usual operations (push-forwards, pull-backs, tensor, hom, duality, vanishing/nearby cycles).
Via the  functor ${\rm rat}$, the standard $t$-structure  on the domain corresponds to the middle perversity
$t$-structure  on the target (see \ci{bams} for the basics on middle perversity and references); in particular, if $K \in MHM_{alg}(S)$, then
${\rm rat} (K)$ is a perverse sheaf on $S$.  
There is a second $t$-structure on $D^b MHM_{alg}(S)$
which corresponds to the standard $t$-structure in $D^b(S,\rat)$ via the functor ${\rm rat}$; \ci[Remark 4.6.2]{Saito 1990}. In particular,  given a morphism $f:X \to Y$ of varieties, one promotes direct image cohomology shaves $R^if_* K, R^i f_!K$ of a complex $K$ in the essential image of ${\rm rat}$ to objects in $D^b MHM_{alg}(S);$ by restricting such objects to suitably Zariski dense
open subvarieties of the regular locus of their support, these objects are admissible polarizable variations
of  mixed Hodge structures.

The functor  ${\rm rat}$ is neither essentially surjective, nor fully faithful;
what is important in this paper is the evident fact that a splitting of an object $K$ induces a splitting of ${\rm rat} (K)$. By abuse of notation and for simplicity, we do not distinguish notationally between an object $K$ and its counterpart  ${\rm rat} (K)$.  

If $Z\subseteq S$ is a closed irreducible subvariety and  $L$ is a polarizable variation of  rational pure Hodge structures  of some weight  $w(L)$
on some Zariski-dense open subset $V$ contained in $Z^{\rm reg}$,
then the intersection cohomology object $IC_Z(L) \in MHM_{alg}(S)$  yields, via $\rm rat$,  the usual intersection cohomology complex of   $Z$  with coefficients the local system underlying $L$; such a local system must be self-dual,   i.e. $L \simeq L^{\vee} (-w(L))$, and semi-simple, i.e. $L$ splits into a direct sum of simple objects.
One can shrink $V$, if desired; this is useful when  one desires $V$ to be subject to certain conditions.

The objects of $MHM_{alg} (S)$ which are pure of weight
$w$ are exactly the ones of the aforementioned form   $IC_{Z} (L)$, with $L$ pure  of weight $w(L)=w - \dim Z$;
they are semisimple, and every simple object has this form with, $L$ simple as polarizable
variation of pure Hodge structures. Note that the corresponding object $\ms{IC}_Z (L):= IC_Z(L)[-\dim Z]$ in 
$MHM_{alg} (S) [-\dim Z]$
in $D^b MHM_{alg}(S)$ is pure of weight $w-\dim Z$; the  cohomology sheaves of ${\rm rat}
(\ms{IC}_Z (L))$ are in non-negative
degrees, starting at zero. For the geometric statements, e.g. (\ref{dtb'z}), we use $\ms{IC}$, and at times, in order to exploit the simplified bookkeeping stemming from Poincar\'e-Verdier and Relative Hard Lefschetz dualities,
we use $IC$, e.g. (\ref{dtb'}).

Let  $Z$ be irreducible and complete and let $L$ (on  some $V\subseteq Z^{\rm  reg} \subseteq Z$) be as above, pure of weight $w(L)$.  Then the cohomology groups
$I\!H^\bullet (Z,L): =H^\bullet (Z, \ms{IC}_Z(L))$ are polarizable pure Hodge structures
of weight $w(L)+\bullet$.
Polarizations are typically not unique, but the 
pure Hodge structures on $I\!H^\bullet (Z,L)$ do not depend on a choice of polarization.
The category of rational polarizable pure Hodge structures  is abelian semisimple: every subobject 
admits a direct sum complement; the induced splitting does not need to preserve given polarizations.

An object $L$ as above of weight $2v \in \zed^{\rm ev}$ is said to be of pure Hodge-Tate type
$(v,v)$ 
if the pure Hodge structure on the stalks is of type $(v,v)$.
If the polarizable variation of pure Hodge structures $L$ is of pure Hodge-Tate type  and of weight $2v$, then the underlying local system
determines the structure of  polarizable variation of pure Hodge structure of weight $2v$ on $L$ uniquely, i.e.
if $L, L'$ are two polarizable variations of pure Hodge structures, which are both pure of Hodge-Tate type
and of the same weight $2v$, and  such that the underlying local systems are isomorphic, then 
$L \simeq L'$ as objects in $MHM_{alg}(S)$; this is simply because the associated 
Hodge filtration is trivial.

Let $f: X\to Y$ be a proper morphism of varieties and let $K \in MHM_{alg} (X)$ be pure of weight
$w$. E.g. $K= IC_X$,  or even  $K = \rat_X [\dim X]$  when $X$ is irreducible nonsingular;  in either case, $w = \dim X$. In this context, the Decomposition  and the Relative Hard Lefschetz Theorems (see \ci{bams} for references) give :
$Rf_* K \in D^b MHM_{alg} (Y)$  is pure of weight $w$,  it splits
as the direct sum $\oplus_b \,{\ph^b} (Rf_* K)[-b]$ of its shifted perverse cohomology sheaves, and each 
$ \,{\ph^b} (Rf_* K)[-b]$ is pure semisimple of weight $w-b$; if in addition, $f$ is projective and $\frak l$ is the first Chern class of an ample line bundle on $X$, then cupping with the powers of $\frak l$ induces the
Relative Hard Lefschetz isomorphism ${\frak l}^k: \,{\ph^{-k}} (Rf_* K)
\simeq \,{\ph^k} (Rf_* K)$, $\forall k \geq 0$. Of course, similar statements hold for a shifted $K[a]$,
so that,  after the evident bookkeeping, they apply, for example, to $\ms{IC}_X$ (this is defined for not necessarily irreducible nor pure dimensional  varieties and  is pure of weight zero as the direct sum
of the $\ms{IC}$'s of the irreducible components (cf. \ci[\S5]{dejag}). 

Similar consideration hold for $Rf_* \ms{IC}_X$
in $D^b(X,\rat)$. A statement in $D^b MHM_{alg}(-)$ may contain Tate shifts; when considering the resulting statement in 
$D^b(-,\rat)$ it is understood that one omits the meaningless Tate shifts.

If $f:X\to Y$ and $K$ are as above, then the closed subvarieties that are the supports of a simple summand
of the direct image $Rf_* \ms{IC}_X$ are called the supports of  $f$.  The notion of support is of paramount importance in this paper. The determination of the supports of a morphism is a subtle problem.
Of course, one can define
the supports of $Rf_* K$ for any object, by taking the supports of the constituents of $Rf_*K$,
which is the unique finite collection of simple perverse sheaves appearing in a Jordan-H\"older
decomposition of the perverse sheaves $\,{\ph^k} (Rf_* K)$  for $k \in \zed$.

Let $f: X \to Y$ be a morphism, let $Z\subseteq Y$ be a subvariety. Then we set $X_Z:= f^{-1}(Z)$.
In particular, if $y \in Y$ is a point, then $X_y=f^{-1}(y)$ denotes the fiber.

\subsection{The notion of $\delta$-regular weak abelian fibrations}\la{drwaf} In \ci{ngofl,ngoaf}
B.C. Ng\^o  introduced the notion of  $\delta$-regular weak abelian fibration, which is a notion of paramount importance in our approach to the proof of Theorem A.
For more details on what follows, we refer to the papers of  Ng\^o quoted above as well as to
our Appendix    \S\ref{vnstbis}, which contains the refinement of Ng\^o Support Theorem needed in this paper.

Let $g:P\to S$ be a smooth commutative group $S$-scheme, and let $g^o: P^o\to S$ be its identity component.
Given any Zariski point $s$ in $S$, there is the canonical Chevalley devissage of the fiber $P^o_s$ of $P$ at $s$, 
i.e. a canonical short exact sequence of connected commutative group schemes $0\to R_s \to P^o_s \to A_s \to 0$ over the perfect field $k(s)$, where $R_s$ is affine  and connected,  and maximal with these properties, and where  $A_s$ is an abelian variety
(cf. e.g.:
  \ci[Thm 10.25,  Prop 10.24, Prop 10.5 (and its proof), Prop 10.3]{milneAGS}.
  
We introduce the quantity $\delta(s):= \dim_{k(s)} R_s$, which plays an important role in this paper. For example, if $P_s$ is the Jacobian of a reduced projective curve $C$, then $\delta$ is the dimension of $
\ker[\nu^* :Pic^0(C) \to Pic^0(\hat{C})]=  (\nu_* \m{O}^\times_{\hat{C}}/\m{O}^\times_C)/\im H^0(\hat{C},  \m{O}^\times_{\hat{C}})
$
where $\nu: \hat{C} \to C$ is the normalization morphism (cf. \cite[21.8.5]{EGAIV(4)}). If $C$ is integral with one node or cusp and no other singular point, then $\delta=1$; if $C$ is a curve with two nodes and no other singularities, then $\delta=2$ when $C$ is irreducible, and $\delta=1$ when $C$ is the union of two  irreducible components. 
%an integral projective  curve $C$ with one node or one cusp $x$, then 
%$\delta=1$ is the usual Serre's invariant ($\dim_\comp \nu_* \m{O}_{\hat{C}}/\m{O}_C$, where $\nu$ is the normalization morphism); if $P$ is the Jacobian of  a banana curve $C$ with two nonsingular  components meeting transversally at two distinct points, then 
%$\delta=1$, whereas the cokernel above has dimension $2$. 
The function $\delta$ is upper-semicontinuous on $S$.  If $Z\subseteq S$ is an irreducible subvariety, then one sets $\delta (Z)$ to be the value $\delta (\eta_Z)$
at the generic point $\eta_Z$ of $Z$. In this case, we have that  $\delta (Z)$ coincides with
the value $\delta (z)$
 at a general  point $z \in Z$, 
and also with the minimum of $\delta (-)$ on the points of $Z$. 
A crucial role is played by the $\delta$-loci, which  are the locally closed subvarieties $S_i\subseteq S$ defined by setting:
$
S_i:= \{s \in S\, |\; \delta (s) =i \}
$.
 Let $d:=\dim P/S$ be constant (it is so on the connected components of $S$). The Tate module $P/S$ is the 
 object 
$T(P)= T(P/S):=R^{2d-1}g^o_! \rat_{P^o} (d)$ in  $D^b MHM(S)_{alg}$ (cf. \ci[\S4.12]{ngofl}; the Tate shift 
$(d)$ is in the sense of M. Saito's theory of mixed Hodge modules (cf. \S\ref{gennot}). Given a point $s$ of $S$, we have the short exact sequence
$0 \to T(R_s) \to  T(P^o_s) \to T(A_s) \to 0$ of  rational mixed Hodge structures. Note that $\dim T(R_s) \leq \delta(s)$ and the strict inequality is possible, e.g. a projective rational curve with a cusp.

For a fixed prime $\ell$, let $T_{{\rm et}, \oql} (P/S)$ be the $\oql$-adic counterpart to $T(P/S)$ (defined by the analogous formula, using the \'etale topology/cohomology formalism). We say that $T_{\rm et, \oql}(P/S)$ is polarizable
if  \'etale-locally there is a pairing $T_{{\rm et}, \oql}(P/S)\otimes T_{{\rm et}, \oql}(P/S)\to \oql (1)$ such that, for every point $s$ of $S$, the  null space
of the pairing at $s$ is precisely $T_{{\rm et}, \oql}(R_s)$.

\begin{defi}\la{waf}
{\rm ({\bf $\delta$-regular weak abelian fibration})}
A weak abelian fibration is a pair of morphisms: 
$(f: M \to S, g: P \to S)$ 
such that:
\begin{itemize}
\item[(a)] $f$ is proper;
\item[(b)] $P$ is a smooth commutative group scheme;
\item[(c)] There is an action $a: P\times_S M \to M$ of $P$ on $M$ over $S$; 
\item[(d)]
$f$ and $g$ have the same pure relative dimension, denoted by $d$;
\item[(e)] the action  has affine stabilizers at every point $m\in M$;
\item[(f)]
the Tate module $T_{{\rm et}, \oql}(P/S)$  of $P/S$ is polarizable.
In context, we abbreviate $(f: M \to S, g: P \to S)$ to $(M,S,P)$.
\end{itemize}

If, in addition, 
we  also have the following property of $P/S$:
\beq\la{dereg}
{\rm codim} \,S_i  \geq i, \quad \forall i \in \zed^{\geq 0},
\eeq
then we say that $(M,S,P)$ is a  $\delta$-regular weak abelian fibration.

\end{defi}

\begin{rmk}\la{stezz}
Note that the $\delta$-regularity inequality  (\ref{dereg}) is equivalent to:
\beq\la{deregz}
{\rm codim}_S \,Z  \geq \delta(Z), \quad \mbox{for every locally closed integral subvariety $Z \subseteq S$}.
\eeq
Moreover,  (\ref{dereg}) implies that if $(M,S,P)$ is a $\delta$-regular weak abelian fibration, then the general fiber $P^o_s$ is complete, i.e. it is
an Abelian variety.
\end{rmk}

If $P \to S$ is the relative degree--$0$ Picard scheme of a family of integral curves parametrized by $S$, then the fibration is $\delta$-regular if and only if for all $\delta \ge 0$, the locus of curves of cogenus $\delta$ has codimension $\ge \delta$. Examples of $\delta$-regular weak abelian fibrations will be given in Example \ref{BMsystemdeltareg} and in \S \ref{picrel}.

%In \S \ref{} we will show that certain Lagrangian fibrations are examples of delta regular WAF and in \S \ref{} we will N AND M.
  %Other examples include the Hitchin system BLA. An example where $\delta$ regularity fails is BLA

%\subsection{The abelian part of $P$ over subvarieties}\la{pab}$\;$

Over a perfect field  the Chevalley devissage exists and  is unique. 
Since we are working in characteristic zero, 
by  working with Zariski points on $S$, one sees that (cf. \ci[\S7.4.8]{ngofl}), given any integral  locally closed subvariety
$Z\subseteq S$, there exists an open and dense  subvariety $V\subseteq Z$  and a short exact sequence of 
smooth commutative group schemes over $V$ with  connected fibers:
\begin{equation}\label{ghj}
0 \to R_{V} \to P^o_{|V} \to A_{V} \to 0,
\end{equation}
that realizes the Chevalley devissage  point-by-point on $V$. One can shrink $V$, if needed.

\subsection{Weak abelian fibrations and holomorphic symplectic manifolds}  \label{subsection symplectic}$\;$

In this section we consider weak abelian fibrations arising from Lagrangian fibrations and we show that under some assumptions they are $\delta$-regular (Proposition \ref{deltaregularity}). Examples of such Lagrangian fibrations are given in Example \ref{BMsystemdeltareg} and  \S \ref{mo10}.  

In this section, we fix the following data:
let $(M,\s_M)$ be a quasi-projective   holomorphic symplectic manifold  of  dimension $2d$ with holomorphic symplectic form $\sigma_M$;  let $S$ be a smooth variety  of dimension $d$; let  $f: M \to S$ be a  proper Lagrangian fibration, i.e. $f$ is a  proper surjective morphism  with 
connected fibers whose
general fiber is a Lagrangian subvariety;  let $P/S$ be a  smooth commutative group scheme with connected fibers, acting on  $M/S$.

By \cite[Thm 1]{Matsushita}, every
irreducible component of every closed  fiber  of $f$ is Lagrangian. This means that for every  $s \in S$, the pullback of 
the symplectic form to a resolution of the singularities of any component of the fiber $M_s$, endowed with the reduced structure,  is trivial.
In particular,  $f$ is equidimensional. By \cite[Prop 1]{Beauville},  the morphism $f$   is an algebraic completely integrable system. It follows that the general fiber of $f$ is a compact complex torus.

As we show in Lemma \ref{da1}, the fact that $M \to S$ is a Lagrangian fibration has consequences on the action of $P/S$ on $M$. This Lemma is then used to prove the $\delta$-regularity property. 
The action of the group scheme $P/S$ on the Lagrangian fibration $M/S$ defines a commutative diagram with Cartesian squares:
\beq\la{em}
\xymatrix{
M \ar[r]^-{l_M} \ar[rd]^-{e_M}  \ar[d]^-f & P\times_S M \ar[d]^-{p_2} 
\ar@<-.5ex>[r]_-{p_2}   \ar@<.5ex>[r]^-a
   \ar[rd]^-r  &  M \ar[d]^-f 
\\
S \ar[r]^-\zeta    
%\ar@{.>}/^1pc/[u]     
%\ar@/^1pc/[u]^\xi 
\ar@{.>}@/^1pc/[u]^-\xi    
& P \ar[r]^-g & S,
}
\eeq
where $\zeta: s \mapsto e_s \in P_s$ (the identity section), $l_M: M_s \ni m_s \mapsto (e_s, m_s)$, $a$ is the action,
$p_i$ the projections. Note that $e_M (m_s)=e_s$.
The dotted arrow $\xi$ is not part of the initial data, but is produced, after a suitable base change, under the hypotheses
of Lemma \ref{da1}.(2a).

\begin{lm}\la{da1} 
{\rm ({\bf Duality for the  action})}
$\;$
Let $M/S, \s_M, P/S$ and $a_M$ be as above.
\ben
\item

There is a commutative diagram:
\beq\la{dactm}
\xymatrix{
& 
%f^*\om^1_S  \ar@/_3pc/[dd]_(.25){v={\rm Id}}  \ar[r]^-u_-\simeq 
& e^*_M T_{P|S} 
% = f^* \zeta^* T_{P|S} 
%\ar[ld]_-{\theta_M}^(.30){\simeq}  
\ar[ld]_-{\theta_M}  
\ar[rd] ^-{d{\rm act}_M} &&\\
0 \ar[r] & (\om^1_{M|S})^\vee  \ar[rr]^-{\beta^\vee}  \ar[d]^\simeq_-{\s'_M}&  &  T_M \ar[r]^-{df} \ar[d]^\simeq_-{\s_M} &  \im\, df
(\subseteq f^* T_S)  \ar[d]^-\simeq _-{\s''_M} \ar[r] & 0\\
0 \ar[r] & f^* \om^1_S \ar[rr]^-{df^\vee} &  & \om^1_M \ar[r]^-\beta & \om^1_{M|S} \ar[r] & 0,
}
\eeq
where: $e_M^* T_{P|S}= f^* \zeta^* T_{P|S}$ is the pull-back via $f$ of the vector bundle on $S$
with fibers the Lie algebras ${\rm Lie}(P_s);$
 $\om^1_{M|S}$ is torsion free; $\ke \, df = (\om^1_{M|S})^\vee = \m{H}om_{\m{O}_M} (\om^1_{M|S}, \m{O}_M)$ is locally free;  the vertical arrows are isomorphisms induced by the symplectic form $\s_M$; rows two and three are short exact sequences.

  The formation of the diagram is compatible with \'etale base change
$S' \to S$ and with restriction to non-empty, not-necessarily $P$-invariant,  open subsets of $M.$
In particular, $\theta_M$ is an isomorphism if an only if it is an isomorphism after base change 
$S' \to S$ given by an \'etale covering.
The roof of the diagram depends on all the initial data, 
while  the rest of the diagram depends only on $M/S$ and $\s_M$.

\item
Assume  in addition that:

\ben
\item either there is  an open subset $M^\circ \subseteq M$ such that $M^\circ/S$ is a $P/S$-torsor;
\item
or there are: an open subset $S^\natural \subseteq S,$ with ${\rm codim} (S\setminus S^\natural) \geq 2,$
and an open subset   $M^{\natural \circ} \subseteq M^\natural$ such that $M^{\natural \circ}/S^\natural$ is a $P^\natural/S^\natural$-torsor.
\een

Then $\theta_M$ is an isomorphism. In particular, $d {\rm act}_M$ and $df$ are dual to each other via the symplectic form $\s_M,$ i.e., for every $m\in M$ with $s:=f(m)$,
we have the commutative diagram at the level of fibers  at $m$: 
\beq\la{cddu}
\xymatrix{
{\rm Lie} (P_s)=T_{P_s, e} \ar[r]^-{\s'_{M,m} \circ \theta_{M,m}}_-\simeq \ar[d]_-{d{\rm act}_{M,m}} &
\Omega^1_{S,s} \ar[d]^-{df^\vee_m} 
\\
T_{M,m} \ar[r]^-{\s_{M,m}}_-\simeq & 
\Omega^1_{M,m}.
}
\eeq

\een

\end{lm}

\begin{proof}
We start by explaining the isomorphism of the two horizontal short exact sequences in (\ref{dactm}). Consider the  short exact  sequence 
$ 0 \to f^* \om^1_S \stackrel{df^\vee}{\to} \om^1_{M}  \stackrel{\beta}{\to} \om^1_{M|S}\to 0$. The first morphism is injective because it is generically injective ($f$ is a dominant generically smooth morphism) and $f^* \om^1_S$ is locally free. Dualizing, we get a short exact sequence $0 \to  (\om^1_{M|S})^\vee  \stackrel{\beta^\vee}{\to} T_M  \stackrel{df}{\to} \im df \to 0 $, where $\im [df: T_M \to f^* T_S] \subset f^* T_S$. Since $f$ is Lagrangian, the composition $\beta \sigma_M \beta^\vee$ is zero, because
it is generically zero and hence zero by  torsion-freeness of the source. This gives, uniquely,  an injective $\s'_M$ and a surjective $\s''_M$ making the diagram commutative. Since $\s''_M$ is a surjective morphism of  coherent sheaves of the same  rank, it is also generically injective. By the torsion-freeness of its source $\s''_M$ is injective and thus an isomorphism. At this juncture,  $\s'_M$ must be an isomorphism as well. 

Now we construct the roof of (\ref{dactm}). The action $a$ induces a morphism $a^* \om^1_{M|S} \stackrel{da^\vee}{\to} \om^1_{P\times_S M|S}=p_1^* \om^1_{P|S}\oplus p_2^* \om^1_{M|S}$. Precomposing with $a^* \beta: a^* \om^1_{M} \to a^* \om^1_{M|S}$ and postcomposing with the projection $\rm {proj}: \om^1_{P\times_S M|S} \to p_1^* \om^1_{P|S}$ yields a morphism of vector bundles $a^* \om^1_{M} \to p_1^* \om^1_{P|S}$, which we  dualize: this gives a morphism $p_1^* T_{P|S} \to a^* T_M$, which we pullback via $l_M^*$. We obtain a morphism of vector bundles
\[
d\,{\rm act}_M: l_M ^*p_1^* T_{P|S}=e_M^* T_{P|S} \to l_M^*a^* T_M=T_M
\]
Notice that $e_M^* T_{P|S}$ is the pull-back  via $f$ of the vector bundle on $S$  with fibers the Lie algebras  of the fibers of  $P/S$. Notice also that the morphism $d \, {\rm act}_M$ deserves its name: by construction,
for  every  $s \in S$ and  every $m \in M$ with $m$ over $s$, 
the fiber $(e^*_M T_{P|S})_m={\rm Lie} \, P_s$, and $d \, {\rm act}_m (v)= d \, a_m (v)$, where $a_m: P_s \to M_s \subset
M$ sends $p \in  P_s$ to $p \cdot m$ (where $\cdot$ denotes the action of $P/S$ on $M/S$). 

Note that $a^*$ and duality commute by the smoothness of $a$. By construction, $d \, {\rm act}_M$ factors as $e_M^* T_{P|S} \stackrel{}{\longrightarrow}   l_M^*a^* (\om^1_{M|S})^\vee=(\om^1_{M|S})^\vee \stackrel{\beta^\vee}{\longrightarrow}  T_M$, where the first morphism, which we denote by $\theta_M$,  is  $l_M^* ( {\rm proj} \circ da^\vee)^\vee: e_M^* T_{P|S} \stackrel{}{\longrightarrow}  (\om^1_{M|S})^\vee$.

Let us also note the following. Assume that the general  stabilizer of the action of $P/S$ on $M/S$ is trivial.  Then 
the morphism $d\, {\rm act}_M$, and thus $\theta_M$,  is  injective. 
If $U\subseteq M$ is a $P$-invariant Zariski-dense open subset  with the property that the stabilizers of its points are trivial, then the restriction  $\theta_{U}$ is an  isomorphism, for  the action of $P$ is free on $U.$

Let us prove part (2a). There is an \'etale covering of finite type $S'\to S$ such that (we denote, temporarily, objects after base change with a prime) the torsor $(M^\circ)'/S'$ admits a section $\xi: S'\to (M^\circ)'$. 
It is enough to show that $\theta_{M'}$  is an isomorphism. We may thus assume that the original $M^\circ$ is a
trivializable $P/S$-torsor, and we drop the primes. Since our aim is to verify that $\theta_M$ is an isomorphism,
we may assume that   $M$ is irreducible.
Let $\tau: P \stackrel{\sim}\to M^\circ$ be the corresponding trivialization: $ P_s \ni p \mapsto p \cdot \xi (s).$
The structural $S$-morphism of $P/S$ is Lagrangian for the symplectic form  $\sigma_{P,\tau}:= \tau^* \sigma_{M^\circ}$, so that we have the standard natural isomorphism 
$u_{\sigma_{P,\tau}}:\om^1_S \stackrel{\sim}\to \zeta_P^* T_{P|S}$, with pull-back 
$g^*u_{\sigma_{P,\tau}}:  g^*\om^1_S \stackrel{\sim}\to  g^*\zeta_P^* T_{P|S}=e_P^* T_{P|S}$
given by $\theta_P^{-1} \circ (\sigma_{P,\tau}')^{-1}$.
Pulling-back via $f,$ gives the isomorphism:
$f^*u_{\sigma_{P,\tau}}:  f^*\om^1_S \stackrel{\sim}\to  f^*\zeta_P^* T_{P|S}= e_M^* T_{P|S}.$

CLAIM:  $u_{M,\tau}:= \s'_M \circ  \theta_M \circ  f^*u_{\sigma_{P,\tau}}\in {\rm End} (f^*\om^1_S)$ is an isomorphism, so that, since the first and third factors are isomorphisms,  $\theta_M$ is an isomorphism, and (2a) follows.

Proof of the CLAIM (and end of proof (2a)).
Note that the section $\xi$ of $M^\circ/S$ also defines a section, denoted by the same symbol, of  $M/S$, i.e. of $f$.
We now have diagram (\ref{em}) in its entirety.
 Using the adjoint pair $(f^*,f_*)$,  that $f \circ  \xi = {\rm Id}_S$,  the projection formula, and the fact that $f_*
 \m{O}_M=\m{O}_S$ ($f$ is its own Stein factorization),
 we see that the adjunction isomorphism
${\rm Hom} (f^*\om^1_S, f^* \om^1_S) \stackrel{\sim}\to   {\rm Hom} (\om^1_S, \om^1_S)$
coincides with $\xi^*$, and has inverse $f^*$: a morphism $\varphi$ in the former group is the pull-back via $f$ of a unique morphism in the latter, namely $\xi^*\varphi$, and, moreover,  $\varphi$ is an isomorphism if and only if $\xi^*\varphi$ is an isomorphism. 
It remains to show that $\xi^* u_{M,\tau}$ is an isomorphism. Since the first and last factors are isomorphisms,
it remains to show that $\xi^* \theta_M$ is an isomorphism. This is automatic, since $\xi (S)$ is inside
the $P$-torsor $M^\circ$ and, as it has been observed above, $\theta_{M^\circ}$ is an isomorphism.

Let us prove part (2b). Note that (2b) implies (2a), however, for clarity, we have chosen to prove (2a) first.
We apply (2a) to the situation of $S^\natural.$ We obtain that the morphism $\theta_M$ of vector bundles
on $M$ is an isomorphism in codimension  two, hence an isomorphism.
\end{proof}

In the next proposition, by abuse of notation, when we write something like $T_{P_s, e}$, etc.,  we mean the fiber of the corresponding vector bundle, not the stalk of the coherent sheaf.

\begin{pr} \label{deltaregularity}
{\rm  ({\bf  $\delta$-regularity})}
 Let $(M, S, P)$ be a weak abelian fibration  such that: $M$ and $S$ are nonsingular; $M$ is quasiprojective;
 $M/S$ is a proper Lagrangian fibration;
 $P/S$ has connected fibers; there is an open subset $M^\circ \subseteq M$ such that $M^\circ/S$ is a $P/S$-torsor.
  Then $(M, S, P)$ is $\delta$-regular
 (cf. (\ref{deregz})).
\end{pr}

\begin{proof} 
Let $Z \subset S$ be a locally closed integral subvariety. By shrinking $Z$ if necessary,  we have the Chevalley devissage (\ref{ghj})   $0 \to R_Z \to P^o_Z \to A_Z \to 0$ of the identity component  $P^o_Z$ in affine and Abelian parts. For a general point $z \in Z$, $\dim A_z = \dim S - \delta (Z).$

 It is enough to show that for  a general point $z \in Z$, we  have a surjection: 
\beq \label{claim da dim}
T_{A_z, e} \twoheadrightarrow \Omega^1_{Z,z},
\eeq
since this   implies that $ d- \delta  (Z) \ge  \dim Z$, which is our contention (\ref{deregz}).

It is enough to obtain a surjection (\ref{claim da dim}) after an  \'etale base change $Z'\to Z$ with $Z'$ integral.

Let  $M_{Z} \to Z$  be the restriction of $M\to S$ over $Z$. 
Let $\eta_Z \in Z$ be the generic point. Since $f$ is proper, $M_{\eta_Z}/\eta_Z$ is complete.

Let $L/k(\eta_Z)$ be a finite field extension 
so that $M_L/L$ has an $L$-rational point and such that the affine solvable $R_L$ is $L$-split.
By the Borel fixed-point theorem \ci[\S V, Proposition 15.2]{Borel Alg Grp}, $M_L$ admits an $L$-rational point that is fixed by $R_L$.
By taking the integral closure $Z'$  of $Z$ in $L$, and after shrinking $Z'$ if necessary, we obtain a morphism $Z' \to Z$ such that the resulting $M_{Z'} \to Z'$ admits a section -corresponding to the fixed $L$-rational point found above- fixed by $R_{Z'}$,  and such that  $Z'\to Z$ is \'etale. Let $Q_Z'$ be the image of this section.

In view of our objective  (\ref{claim da dim}), we may now assume without loss of generality that $Z=Z'$.

 The action of $R_Z$ is trivial on $Q_Z$. If $m \in Q_Z$,  and $z \in Z$ is the point over which  $m$ lies, 
 then the  infinitesimal action of $P$ on $M$ at $m$ factors as follows:  ${\rm Lie} (P_z) \twoheadrightarrow {\rm Lie}(A_z) \to T_{M,m}$. 
 The compositions $Q_Z \to Z \to S$ and $Q_Z \to M \to S$ coincide.  
 By combining this coincidence of compositions, with the factorization of the infinitesimal actions,  and with   (\ref{cddu}), we obtain the  commutative diagram:
 \beq\la{few}
 \xymatrix{
 {\rm Lie}(A_z) \ar[rd] & {\rm Lie}(P_z)  \ar[r]^-{{\s'_{M,m} \circ \, \theta_{M,m}}}_-\simeq \ar[d]^-{d {\rm act}_m}   \ar@{->>}[l]  &
 \Omega^1_{S,z} \ar[r]\ar@{->>}[r]  \ar[d]^-{df^\vee} &
 \Omega^1_{Z,z} \ar[d]^-\simeq  \\
 & T_{M,m} \ar[r]^-{\s_{M,m}}_-\simeq &
 \Omega^1_{M,m} \ar[r] &
 \Omega^1_{Q_Z,z},
 }
 \eeq
It follows that the image
 of $T_{A_z,e}$ in $\Omega^1_{Q_Z,z}$ coincides with the injective image of 
 $\Omega^1_{Z,z}$. Hence the desired surjection 
 (\ref{claim da dim}) at a general point of $Z.$
\end{proof}

\begin{ex} \label{BMsystemdeltareg} Let $(S,H)$ be a general polarized K3 surface. Let $\chi$ be an integer, set $v=(0,H, \chi)$, and let $\pi: M_v(S) \to |H|$ be the moduli space of pure dimension $1$ sheaves on $S$ with Mukai vector $v$ (see \S  \ref{pure one dim}). This is a smooth projective irreducible holomorphic symplectic manifold. Let $P \to |H|$ be the relative degree--$0$ Picard scheme of the family of curves in $|H|$, which exists as a scheme since the curves in $|H|$ are reduced an irreducible (\cite[Thm 1 \S 8.2]{Bo-Lu-Ra}). Then $P/|H|$ acts on $M_v(S)/|H|$ with affine stabilizers (see the forthcoming Lemma \ref{actionPonM} and \ref{affstab}), the Tate modules is polarizable  (cf. Lemma \ref{plariz}) so the pair $(M_v(S), P)$ is a weak abelian fibration. 
Over the open subset $ \mathcal U \subset M_v(S)$ parametrizing line bundles on the curves of $|H|$, the action is free. Since the fibers of $\pi$ are compactifed Jacobians of locally planar integral curves, the assumption of Proposition \ref{deltaregularity} are satisfied (with $M^\circ=\mc U$) and the fibration is $\delta$-regular.
\end{ex}

For  more details on the background and context for this example and for other  weak abelian fibrations see \S \ref{picrel}.

\section{The manifolds $\ogm$ and $\ogn$ as $\delta$-regular weak abelian fibrations}\la{mo10}$\;$
{ 
In this section we introduce the Lagrangian fibered irreducible holomorphic symplectic manifolds $\ogm$ and $\ogn$, as well as other auxiliary fibrations. We start by assembling some known facts about $OG10$-type varieties (\S \ref{manog10}) and about moduli spaces of pure dimension one sheaves (\S \ref{pure one dim})  that are needed in the paper. Then we show that the Lagrangian fibrations we have introduced are $\delta$-regular weak abelian fibrations (\S \ref{picrel} and \ref{sonowaf}).
}

\subsection{ $OG10$-type manifolds}\la{manog10}$\;$
Let $(\ks, H)$ be a polarized K3 surface with N\'eron-Severi group $NS(\ks)= \zed H$.
%\ci[p.390]{Huybrechts lectures on k3 surfaces} 
We identify vectors $v:= (v_0, v_2, v_4) \in \zed^3$ with elements in $H^{\rm even}_{alg}(\ks, \zed)$,  via the obvious identification, and we consider the even quadratic form
$
v^2:=v_2^2 H^2-2v_0 v_4
$
 on $H^{\rm even}_{alg}(S, \zed)$
induced by the  Mukai pairing.
A vector $v \in H^{\rm even}_{alg}(S, \zed)$ is called positive (cf. \cite[Def. 0.1]{Yoshioka-Moduli-Ab} and \cite[Def. 5.1]{Bayer-Macri-Proj}) if $v^2 \ge -2$ and if either: $a)$ $v_0 >0$; or $b)$ $v_0=0, v_2>0$ 
%and $v_4 \neq 0$
; or $c)$ $v_0=v_2=0$ and $v_4 >0$. We say that a Mukai vector $v$ 
is primitive if $v$ is not of the form   $k v'$, for $k \neq \pm1 $.  For a coherent sheaf $\m F$ on $S$ we denote by $v(\m F):= {\rm ch} (\m F) \cup \sqrt{\rm td (S)}$ the Mukai vector of $\m F$.

\begin{rmk} \la{caso0m0}  Conditions $a)$, $b)$ and $c)$ are necessary  on $v$  for  the  existence  of a coherent sheaf $\m F$ with $v(\m F)=v$.
In 
\cite[Def. 0.1]{Yoshioka-Moduli-Ab} and \cite[Def. 5.1]{Bayer-Macri-Proj}, in case $b)$ it is assumed  that  $v_4 \neq 0$ in order to ensure that for primitive $v$ and generic polarization, semistability  implies stability
for $\m F$.  Since this is automatic if $NS(\ks)\simeq \zed$, we drop this assumption (cf. \cite[Rmk 2.6]{Pe-Ra-sing}).
However, in the few cases where we deal with K3 surfaces 
with higher Picard rank (Remark \ref{non general} and Theorem B$'$), we refer to the notion of positivity for Mukai vectors given in \cite[Def. 0.1]{Yoshioka-Moduli-Ab} and \cite[Def. 5.1]{Bayer-Macri-Proj} \end{rmk} 

Let  $v'$  be a  primitive   and positive Mukai vector, 
let $m \ge 1$,  and set  $v:=m v'$. 
What follows is classical work of several authors
 \cite{Mukai, Yoshioka-Moduli-Ab, Kaledin-Lehn-Sorger, OGrady1}.  
The moduli space $M_v(\ks)$  of Gieseker-semistable  pure  sheaves on $\ks$ with Mukai vector equal to $v$ is a normal irreducible projective variety of dimension $v^2+2$    \cite[Theorem 0.3]{Gieseker} and \cite[Theorem 4.4]{Kaledin-Lehn-Sorger}.  
%The definition of Gieseker stability for pure sheaves of dimension one on a projective surface, which is what we use in this paper, is recalled below in Definition \ref{defstab}. 
The points of $M_v(S)$ are the isomorphism classes of polystable sheaves or, equivalently, the S-equivalence classes of semistable sheaves \cite[Theorem 4.3.3]{HL}.
%Recall that: a polystable sheaf is a semistable sheaf that is the direct sum of stable sheaves;  every semistable sheaf admits a filtration, the  Jordan-H\"older filtration, whose associated graded object is a polystable sheaf; two semistable sheaves are called S-equivalent if their associated graded object with respect to the Jordan-H\"older filtration are isomorphic as sheaves;  each S-equivalence class contains  a polystable sheaf, which is unique up to isomorphism.
By \cite{Mukai}, the smooth locus of   $M_v(\ks)$ is precisely the locus parametrizing stable sheaves and it admits  a holomorphic symplectic form.  It follows that the moduli space is smooth if and only if the locus parametrizing strictly semistable  sheaves is empty. When $NS(\ks)\simeq \zed$, this  happens if and only if $m=1$,  i.e. if and only if $v$ is primitive. If this is the case, the moduli space is an irreducible holomorphic symplectic manifold deformation equivalent to the Hilbert scheme of $\frac{1}{2}v^2 +1 $ points on $\ks$ \cite[Theorem 8.1]{Yoshioka-Moduli-Ab}. For $m \ge 2$,  the singular moduli space $M_v(\ks)$ admits a symplectic resolution if and only if $m=2$ and $v^2=2$ \cite{OGrady1, Kaledin-Lehn-Sorger, Lehn-Sorger}. In this case, the following result holds: 
%as the following theorem states, the symplectic resolution of the $10$-dimensional moduli space  $M_v(\ks)$  is of $OG10$ type.
%%
%%
%%In this case, as the following example
%%
%%In the following theorem we collect the fundamental results on
%%
%%In this case, the moduli space $M_v(\ks)$  is $10$-dimensional and the symplectic resolution
%%$\w{M}_v(\ks)$ is an irreducible holomorphic symplectic manifold deformation equivalent to O'Grady's $10$-dimensional exceptional example  \cite[Theorem 1.6]{Perego-Rapagnetta}. 

\begin{tm} \label{thmpoly} 
{\rm  ({\bf  $OG10$-type manifolds})}
%Let $v=mv'$,  with $v'$ primitive   and $m\geq 1$. Under the assumption that $NS(\ks)\simeq \zed$ the polystable sheaves with Mukai vector $v$ are of the form $\oplus F_i$, where $v(F_i)=m_i v'$  and $\{m_i\}$ is a partition of $m$. 
Let $v=2v'$, with $v'$ positive and $v'^2=2$. 

(1) The singular locus of $M_{2v'}(\ks)$  is naturally identified with the $8$--dimensional ${\rm Sym}^2 M_{v'}(\ks)$. 

(2) Locally in the classical topology, the singularities of $ M_{2v'}(\ks)$ do not depend on the choice of $v'$ with $v'^2=2$. Moreover, the blow up $\widetilde M_{2v'}(\ks) \to M_{2v'}(\ks)$ of $ M_{2v'}(\ks)$ along the singular locus, with its reduced induced structure, is a  symplectic resolution. 

%%%along $\Delta_{ M_{v'}(\ks) is isomorphic to the closure of a nilpotent orbit of type $\mathfrak o(2,2)$ and thus does not depend on the choice of $v'$ with $v'^2=2$. Moreover,  
%%%%along the smooth locus of ${\rm Sym}^2 M_{v'}(\ks)$, the singularities of $ M_{2v'}(\ks)$  are of type $A_1$;
%%%the blow up $\widetilde M_{2v'}(\ks) \to M_{2v'}(\ks)$ of $ M_{2v'}(\ks)$ along the singular locus, with its reduced induced structure, is a  symplectic resolution. 

(3) The symplectic resolution $\w{M}_v(\ks)$ is an irreducible holomorphic symplectic manifold deformation equivalent to O'Grady's $10$-dimensional exceptional example.

\end{tm}
\begin{proof}
(1) Follows from the fact that since  $NS(\ks)= \mathbb Z$ the polystable sheaves with Mukai vector $v$ are of  form $F_1 \oplus F_2$, with $v(F_i)=v'$. (2) This is  \cite[Thms 1.1 and 4.5]{Lehn-Sorger} (for the rank 2 case, see also  \cite[Prop 2.2]{Kaledin-Lehn}). (3) This is \cite[Thm 1.6]{Perego-Rapagnetta}.
\end{proof}

We recall that O'Grady's original example is $\widetilde{M}_{(2,0,-2)}(S).$

%\begin{rmk} \label{rmkpoly} 
%{\rm  ({\bf  The singular locus of $M_{2v'}(S)$ and the symplectic resolution})}
%%Let $v=mv'$,  with $v'$ primitive   and $m\geq 1$. Under the assumption that $NS(\ks)\simeq \zed$ the polystable sheaves with Mukai vector $v$ are of the form $\oplus F_i$, where $v(F_i)=m_i v'$  and $\{m_i\}$ is a partition of $m$. 
%Let $v=2v'$ with $v'^2=2$. The singular locus of $M_{2v'}(\ks)$  parametrizes sheaves of the form $F_1 \oplus F_2$, with $v(F_i)=v'$ and is thus naturally isomorphic to the $8$--dimensional ${\rm Sym}^2 M_{v'}(\ks)$. 
%
%By \cite[Thm 4.5]{Lehn-Sorger} (for the rank 2 case, see also  \cite[Prop 2.2]{Kaledin-Lehn}), the local complex analytic structure  of the singularities of $ M_{2v'}(\ks)$ does not depend on the choice of $v'$ with $v'^2=2$ and 
%%along the smooth locus of ${\rm Sym}^2 M_{v'}(\ks)$, the singularities of $ M_{2v'}(\ks)$  are of type $A_1$;
%by \cite[Thm 1.1]{Lehn-Sorger}, the blow up $\widetilde M_{2v'}(\ks) \to M_{2v'}(\ks)$ of $ M_{2v'}(\ks)$ along the singular locus, with its reduced induced structure, is the aforementioned  symplectic resolution. 
%
%\end{rmk}

\subsection{Pure dimension one sheaves and the manifolds $\ogm$ and $\ogn$}\la{pure one dim}$\;$
By Theorem \ref{thmpoly} (3), in order to study the Hodge numbers of $OG10$--manifolds we are free to choose any Mukai vector $v=2v' \in H^{\rm even}_{alg}(S, \zed)$ with $v'^2=2$. To make sure  the corresponding projective model
$\widetilde{M}_{2v'}(S)$ admits  a Lagrangian fibration, we chose for $v'$ the Mukai vector of a pure dimension $1$ sheaves on $S$ \cite{ra08}. 

%In order to study the Hodge numbers of the deformation class $OG10$, in  the forthcoming \S\ref{fissnot} we  choose 
% a particular  Mukai vector $v' \in H^{\rm even}_{alg}(S, \zed)$ with $v'^2=2$. This is done to ensure that the corresponding projective model
%$\widetilde{M}_{2v'}(S)$ admits the further property of  having a Lagrangian fibration \cite{ra08},  and one of a special kind to boot.  In order to study the resulting situation, we first need to recall a few facts about moduli spaces of pure dimension one sheaves on K3 surfaces.

Recall that a coherent sheaf $\mc F$ on a scheme is said to be of pure dimension $d$ if the support of $\mc F$, as well as the support of all non-trivial subsheaves of $\mc F$, has dimension $d$ \cite[\S 1.1]{HL}.
Let $(\ks, H)$ be a polarized K3 surface with $NS(\ks)=\mathbb Z H$  and consider the positive  Mukai vector:
\[
v=(0, k, \chi).
%,  \quad \chi \neq 0.
\]
Pure sheaves  on $\ks$ with Mukai vector  $v$ are of pure dimension $1$, have Euler characteristic  $\chi$  and
first Chern class  $kH$.
For example, if $\Gamma \in |kH|$ is a smooth curve and $i: \Gamma \to \ks$ is the closed embedding, then sheaves with Mukai vector equal to  $v$ and  support $\Gamma$  are of the form
$
i_* L,
$
where $L$ is a line bundle on $\Gamma$ with $\chi(L)=\chi$. 

%%%Let $\mc F$ be a pure dimension one sheaf on $S$. Then $\mc F$ admits a depth one and hence admits a length one resolution 
%%%\[
%%%0 \to A \stackrel{j}{\to} B \to  {\m F} 
%%%\] where $A$ and $B$ are vector bundles on of the same rank $r$.  
%%%The Fitting support $Fitt({\m F})$ is defined as the vanishing subscheme 
%%%of the induced morphism $\det j:\det A\rightarrow  \det B$. The scheme thus defined does not depend on the resolution, it always contains the set theoretic support of $F$, and it represents $c_{1}({\m F})$.  
%%%For example, if $C'\in \ks$ is a curve in the linear system $|H|$,  ${\m F}_1$ is the push forward to $\ks$ of a rank $k$ vector bundle on $C'$
%%%and ${\m F}_2$ is the push forward to $\ks$ of a line bundle on $C:=kC'\in |kH|$, then the Fitting support of  ${\m F}_1$ and ${\m F}_2$ 
%%%is the precisely the curve $C$. 

The Fitting support of a pure dimension one sheaf on $\ks$ is a pure dimension $1$ subscheme on $S$ which represents the first Chern class of the sheaf and is defined as follows:

%Following Le Potier \cite[\S 2.2]{LePotier}, every pure dimension one sheaf ${\m F}$ on $\ks$ has an associated Fitting  support, which is a pure dimension one subscheme of $S$ and which represents the first Chern class of $\mc F$.

\begin{rmk}( \cite[\S 2.2]{LePotier}, \cite[\S 1.1]{HL})
\la{fitting} Let $\mc F$ be a pure dimension one sheaf on a smooth projective surface. Then
${\mc F}$ has depth one and hence admits a length one resolution 
\[
0 \to A \stackrel{j}{\to} B \to  {\m F} 
\] where $A$ and $B$ are vector bundles on of the same rank $r$.  
The Fitting support $Fitt({\m F})$ is defined as the vanishing subscheme 
of the induced morphism $\det j:\det A\rightarrow  \det B$. The scheme thus defined does not depend on the resolution, it always contains the set theoretic support of $F$, and it represents $c_{1}({\m F})$.  
For example, if $C'\in \ks$ is a curve in the linear system $|H|$,  ${\m F}_1$ is the push forward to $\ks$ of a rank $k$ vector bundle on $C'$
and ${\m F}_2$ is the push forward to $\ks$ of a line bundle on $C:=kC'\in |kH|$, then the Fitting support of  ${\m F}_1$ and ${\m F}_2$ 
is the precisely the curve $C$. 
\end{rmk}

Using the reduced Hilbert polynomial, we get the following definition of stability for pure dimension one sheaves.

\begin{defi} \label{defstab}
{\rm
({\bf Gieseker stability for pure dimension one sheaves} \cite[Def. 1.2.3, Prop 1.2.6]{HL})
}
A  pure dimension $1$ sheaf  $\mc F$ on a polarized surface  $(X,H)$ is called Gieseker-(semi)stable with respect to $H$ if and only if for all proper  pure dimension one quotients $\mc F \to \mc E$ the following inequality holds:
\[
\frac{\chi(\mc F)}{c_1(\mc F) \cdot H} \underset{(=)}{<} \frac{\chi(\mc E)}{c_1(\mc E) \cdot H}.
\]
\end{defi}

For example, if the Fitting support of $\mc F$ is integral, then $\mc F$  is the push forward of a rank $1$ torsion free sheaf on an integral curve, so $\mc F$ has no proper quotient and is automatically stable. If $\mc F$ is the pushforward of line bundle on a curve $C$, then the only pure dimension $1$ proper quotients of $\mc F$ are the restrictions to the pure dimension one subschemes of $C$. It follows that $\mc F$ is (semi)-stable if and only if  
\begin{equation}\la{stablb}
\frac{\chi(\mc F)}{C \cdot H} \underset{(=)}{<} \frac{\chi(\mc F_{| D})}{D \cdot H}
\end{equation}
for every proper subcurve $D \subset C$. 
In Lemmas \ref{fibratilineari} and \ref{possiblesurjections} and Proposition \ref{restriction of F to C} we will consider the stability of pure dimension one sheaves on certain non reduced curves.

\begin{rmk}\label{non general} {\rm ({\bf  Non generic polarizations for pure dimension one sheaves})}
A priori, the notion of stability depends on the choice of polarization.  
Since in this paper we are mostly concerned with the case of a Picard rank one surface,  the polarization is unique up to scalars and we omit it from the notation for the moduli spaces and simply write $M_v(\ks)$.
However, since in Theorem B$'$ we consider  moduli spaces of sheaves on K3 surfaces of higher Picard rank, we now recall some facts about this more general setting. Consider a positive $v$ (cf. Remark \ref{caso0m0} ).
If $v$ is primitive,
%and positive in the sense of \cite[Def. 0.1]{Yoshioka-Moduli-Ab} or \cite[Def. 5.1]{Bayer-Macri-Proj} (see Remark \ref{caso0m0}), 
then there is always a polarization for which there are no strictly semistable sheaves, i.e., all semistable sheaves are in fact stable; if $v=kv'$ is not primitive, then there is always a polarization for which the summands of the polystable sheaves with Mukai vector $v$ have Mukai vector a multiple of $v'$ \cite[\S 4.C]{HL}, \cite{Yoshioka-Moduli-Ab}. For a given $v$, a polarization with these properties is called $v$-generic. If $v$ is primitive  and $L$ is a $v$-generic polarization, then the moduli space $M_{v, L}(\ks)$ of $L$--stable sheaves on $\ks$ with Mukai vector $v$ is an irreducible holomorphic symplectic manifold of dimension $v^2+2$, deformation equivalent to the Hilbert scheme of $v^2/2+1$ points on $\ks$. The birational class of $M_{v, L}(\ks)$ is independent of $L$ \cite{Bayer-Macri-MMP}. If $v=2v'$ with $(v')^2=2$ and $H$ is any polarization, then $M_{2v', H}(\ks)$ has a symplectic resolution $\widetilde M_{2v', H}(\ks)$, which is of $OG10$-type and whose birational class does not depend on $H$ (\cite{Lehn-Sorger},  \cite[Thm 1.6]{Perego-Rapagnetta}, \ci[Prop 2.5]{Ar-Sa},\cite{MeachanZhang}).
% if $H$ is $v$--generic the existence of a symplectic resolution follows from \cite{Lehn-Sorger} while if $H$ is not $v$--generic one uses \ci[Prop 2.5]{Ar-Sa}  
%(which  assumes $\chi \neq 0$, but the same proof holds for $\chi=0$  as well) 
%to reduce to the case of generic polarization. Finally, recall that by \cite[Thm 1.6]{Perego-Rapagnetta} this symplectic resolution is deformation equivalent to O'Grady's $10$--dimensional example.
\end{rmk}

Let $g$ be the genus of the general curve in $|kH|$, so that $|kH| \cong \mathbb P^g$.
The Le Potier morphism \cite[\S 2.2]{LePotier},\cite[\S 2.3]{LePotierA}:
\beq\la{lepotiermo}
\pi: M_v(S) \to \mathbb P^g,
\eeq
associates with each sheaf its Fitting support. Since the fiber of $\pi$ over an integral curve is precisely the degree $d$ compactified Jacobian of the curve itself, by abuse of notation we refer to $M_v(S)$ as the relative compactified  Jacobian of degree $d:=\chi+g-1$ of the universal curve  over $|kH|$. Note, however, that if curve is non--reduced, the locus in a fiber parametrizing line bundles might be empty (cf. Corollary \ref{odd not loc free}).
%%%
%%%By abuse of notation, we call $M_v(S)$ the relative compactified  Jacobian of degree $d:=\chi+g-1$ of the universal curve  over $|kH|$, Since the fiber of $\pi$ over an integral curve is precisely the degree $d$ compactified Jacobian of the curve itself. More generally, the fiber of $\pi$ over a curve $\Gamma \in |kH|$ is identified with the Simpson moduli space of pure rank $1$ degree $d$ sheaves on $\Gamma$ that are semistable with respect to the polarization $H$; if the curve is integral then there is a dense open set of the fiber parametrizing line bundles on the curve \cite{re80}, so the notation is justified, but if the curve is non--reduced, then the locus parametrizing line bundles might be empty (cf. Corollary \ref{odd not loc free}).
By \cite{Mukai, Beauville-int}, the  morphism (\ref{lepotiermo}) has Lagrangian fibers when restricted to the smooth locus $M_v(S)^{\rm reg}$. 

%As mentioned above, the  integral variety $M_v(S)$  is singular precisely when the positive  Mukai vector $v$ is not primitive.
%As explained in Remark \ref{rmkpoly}, O'Grady's $10$-dimensional examples are symplectic resolutions of moduli spaces of the form $ M_{2v'}(\ks)$, with $(v')^2=2$. 

In order for $v=2v'$ with $(v')^2=2$  to be the Mukai vector of pure dimension one sheaves, one needs  $v_0=0$ and  $H^2=2$. 
%We want a projective model in $OG10$ endowed with a Lagrangian fibration For this,  we first need to  choose for $v'$ a Mukai vector corresponding  to pure dimension one sheaves. This forces $v_0=0,$ and thus $H^2=2$. 
To achieve this consider a polarized K3 surface of degree $2$, i.e., a surface $\ks$ realized as a $2:1$ cover:
\[
r: \ks \to \mathbb P^2,
\]
ramified along a smooth plane sextic curve.  The linear system $H:=r^* \mc O_{\mathbb P^2} (1)$ is of genus $2$, i.e.,  its general member is a nonsingular connected curve of genus $2$ and $|H| \cong  \mathbb P^2$. 
If the smooth sextic curve is general then  $NS(S)\simeq \zed \langle H\rangle$ \cite{Buium}. From now on, unless otherwise stated, we will assume that $\ks$ is a general in this sense. The Mukai vector  
\beq\la{vprimo}
v':=(0,1, \chi')
\eeq
is positive and satisfies $v'^2=2$. 
With these choices, the moduli space $M_{v'}(S)$ is a smooth projective irreducible holomorphic symplectic fourfold, with a Lagrangian fibration
\beq \la{mprimo}
 M_{v'}(S) \to |H| = \mathbb P^2,
\eeq
realizing it as the relative compactified Jacobian of degree $d':=\chi'+1$ of the universal curve $\m{C}'/|H|$.

\begin{rmk} \label{nodeandcusps}
For general $\ks$ as above, the curves in $|H|$ are reduced and irreducible and have only nodes and (simple) cusps as singularities. Moreover, the rational curves in the linear systems have only nodes \cite{Chen}. By \cite{re80}, the fibers of (\ref{mprimo}) are irreducible of dimension $2$.
\end{rmk}

The linear system $|2H|$ is base point free and has dimension $5$,
and its general member is a connected nonsingular hyperelliptic curve of genus $5$ \cite{saint-donat}, which is a ramified double cover of a plane conic. Consider the Mukai vector:
\beq\la{2vprimo}
2v'=(0,2, 2 \chi').
\eeq
The support morphism $M_{2v'}(S) \to |2H| $ realizes this $10$-dimensional singular moduli space as the relative compactified Jacobian of degree $d=2\chi'+4$ of the universal curve $\m{C}/|2H|$. By  \cite[Thm 1]{Matsushita}, composing this morphism with the symplectic resolution gives the structure of Lagrangian fibration (cf. \S\ref{subsection symplectic})
\[
\widetilde M_{2v'}(S) \to M_{2v'}(S) \to |2H| \cong \mathbb P^5,
\]
on this projective manifold $OG10$-type.
We also consider the moduli space   associated with the primitive Mukai vector:
\beq\la{w}
w=(0,2, \chi), \qquad  \mbox{with $\chi$ odd.}
\eeq
With this choice of Mukai vector, the moduli space $M_w(S)$ is a smooth irreducible holomorphic symplectic $10$-fold of K3$^{[5]}$-type, with Lagrangian fibration $M_w(S) \to |2H|$ realizing it as odd-degree $\chi+4$ relative compactified Jacobian of the universal curve $\m{C}/|2H|$.

\begin{rmk} \label{SigmaDelta}
{\rm ({\bf  The subvarieties $ \Delta \subseteq \Sigma \subseteq |2H| $})}
The non integral curves in $|2H| \simeq \pn{5}$ are the preimages under $r$ of the non integral conics. They are parametrized by an irreducible divisor   $\Sigma \subset |2H|$, which is identified with $ \Sym^2 |H|
\simeq \Sym^2 \pn{2}$. The non-reduced curves are the preimages of the non-reduced conics (double lines) and they are parametrized by the irreducible $2$-dimensional subvariety
$\Delta_{|H|} \subset  \Sym^2 |H| \subset |2H|$. \end{rmk}

The analysis of the irreducible components of the fibers of $\widetilde M_{2v'}(S) \to  |2H| $ and $M_w(S) \to  |2H| $ over  $\Sigma$ and $\Delta$ plays a crucial role in this paper and will be carried out in Propositions  \ref{compmtildedelta}, \ref{componentisigma},  and \ref{propriass}. For the moment, we review the following general properties.

\begin{rmk} \label{piattezza}
{\rm ({\bf Flatness of the Lagrangian fibrations $\widetilde M_{2v'}(S), M_{w}(S), M_{v'}(S)$, and $M_{2v'}(S)$}.)
}
The varieties $\widetilde M_{2v'}(S)$, $M_{w}(S)$ and $M_{v'}(S)$ are irreducible holomorphic symplectic manifolds.  By  \cite{Matsushita}, the Lagrangian fibrations 
$\widetilde M_{2v'}(S) \to \mathbb P^5$, $ M_{w}(S) \to \mathbb P^5$, and $ M_{v'}(S) \to \mathbb P^2$ are equidimensional and therefore flat. 
The following three facts concerning the morphism $M_{2v'}(S) \to \mathbb |2H| \simeq \mathbb P^5$ hold true:
(i)  since $M_{2v'}(S)$ has a symplectic resolution, it has canonical singularities, 
and is thus Cohen-Macaulay (see Theorem 5.10 and Corollary 5.24 of \cite{KoMo});
(ii)  since the fibers of 
$M_{2v'}(S) \to \mathbb |2H| \simeq  \mathbb P^5$ are dominated by the corresponding fibers of $\widetilde M_{2v'}(S) \to \mathbb P^5$, they  all have the same dimension  five;   the base of the fibration
$M_{2v'}(S) \to \mathbb |2H|$  is nonsingular. By \cite[Thm 23.1]{Matsumura},  the morphism   $M_{2v'}(S) \to \mathbb |2H|$  is flat. As a consequence, the fibers of $M_{2v'}(S) \to \mathbb |2H|$, $\widetilde M_{2v'}(S) \to \mathbb P^5$, $ M_{w}(S) \to \mathbb P^5$, and $ M_{v'}(S) \to \mathbb P^2$ are  Cohen-Macaulay.
\end{rmk}

 For later use, we collect some known results:

 \begin{pr} \label{lemma birazionali}$\;$  
 Let $v'=(0,1,\chi')$ and let $v=(0,2,\chi)$. Denote by $\ks^{[n]}$ the Hilbert scheme of $n$ points of
the K3 surface $\ks$.
\ben
\item
The birational class of the irreducible holomorphic symplectic manifold $M_{v'}(\ks)$ depends on the parity of $\chi'$ alone. If $\chi$ is odd, then $M_{v'}(\ks)$ is an irreducible holomorphic symplectic manifold birational to  $\ks^{[2]}$.
\item The birational class of $M_{v}(\ks)$ depends on the parity of $\chi$ alone.
If $\chi$ is odd, then $M_{v}(\ks)$ is an irreducible holomorphic symplectic manifold birational to  $\ks^{[5]}$. If $\chi$ is even, then $M_{v}(\ks)$ is singular; blowing up its singular locus (with its reduced induced structure) produces  an irreducible holomorphic symplectic manifold $\w{M}_{v}(\ks)$ in the deformation class $OG10$; 
\item For odd $\chi'$ and $\chi$, there are isomorphisms of integral Hodge structures $H^*(M_{v'}(\ks), \mathbb Z) \cong H^*(\ks^{[2]}, \mathbb Z)$ and $H^*(M_{v}(\ks), \mathbb Z) \cong H^*(\ks^{[5]}, \mathbb Z)$. 

\item The isomorphism class of the integral Hodge structure of  $\w{M}_{v}(\ks)$  is the same for every even $\chi$. 
\een

\end{pr}

\begin{proof}
%The following facts are well known (\cite[Thm 8.1]{Yoshioka-Moduli-Ab}, \cite[Thm 1.6]{Perego-Rapagnetta}): $M_{v'}(S)$ is irreducible holomorphic symplectic; the same is true for $M_{v}(S)$ with $\chi$ odd; the same is true for $\w{M}_{v}(S)$ with $\chi$ even. The first two varieties are deformation of the Hilbert scheme of two, respectively five, points on $\ks$, while the second is of $OG10$--type.

The general curves in the linear systems $|H|$  and $|2H|$ are smooth hyperelliptic, so tensoring a line bundle supported on a smooth curve with the unique $g^1_2$ on its support defines birational maps
  $M_{(0,1,\chi')}(\ks) \dashrightarrow M_{(0,1,\chi'+2)}(\ks)$ and 
$M_{(0,2,\chi)}(\ks)\dashrightarrow M_{(0,1,\chi+2)}(\ks)$.
%It follows that the birational classes  of $M_{v'}(\ks)$ of $M_{v}(\ks)$, and of $\w{M}_v(S)$ with $\chi$ even, depend only on the parity of the last entry of the Mukai vector. 
The statement that  $M_{v'}(S)$ and $M_{v}(S)$, with $\chi'=\chi=1$, are birational to the relevant Hilbert scheme is well-known (cf. \cite{Beauville}).
%Finally, the statements concerning the birational class being the one of Hilbert schemesin the cases of $M_{v'}(S)$ and $M_{v}(S)$ with $\chi'$ and $\chi$ odd, is well-known and can be seen as follows. Let $|D|$ be a  ample  linear system  on a K3 surface $X$ containing a nonsingular  curve of genus $g$. The choice of   $g$ general points on $X$ determine \cite[Thm 3.1]{saint-donat} a unique nonsingular connected curve in $|D|$ together with a line bundle of degree $g$ on  it. This defines a rational map $X^{[g]} \dashrightarrow M_{(0,D,1)}$, from the Hilbert scheme of $g$ points on $X$, to the degree $d$ relative compactified Jacobian of the universal curve over $|D|$. Since a general degree $g$ line bundle on a smooth genus $g$ curve has a unique global section, the morphism above has a rational inverse and hence is birational.
Finally, the statements  about the isomorphisms of  Hodge structures  follow from \ci[Thm 2.5 Cor. 2.6]{Huybrechtsbir}.

\end{proof}

\subsection{The group schemes }\la{picrel}$\;$
%\begin{rmk}\la{picrel}   {\rm ({\bf  Relative Picard})}
 Recall  the definitions (\ref{vprimo}), (\ref{2vprimo}) and (\ref{w}) of the Mukai vectors $v',2v',$ and $w$.
As we shall see in \S\ref{sonowaf},  each of the moduli space  $M_{v'}(\ks)$, $M_{w}(\ks)$ and  $\widetilde{M}_{2v'}(\ks)$
is part of the data forming a $\delta$-regular  weak abelian fibration. 
We now introduce the  relevant group schemes as open subsets:   
\beq\la{il primo}
P' \subseteq M_{(0,1,-1)}(S), \qquad \mbox{for $M_{v'}(\ks)$,  and}
\eeq
\beq\la{il secondo}
P\subseteq M_{(0,2,-4)}(S), \qquad 
\mbox{for $M_w(\ks)$ 
and  $\w{M}_{2v'}(\ks)$.}
\eeq

\begin{lm} \label{loc free naka}
For every Mukai vector $v$ as in (\ref{vprimo}), (\ref{2vprimo}), and (\ref{w}), the locus $P_{v}(S)\subset M_{v}(S)$ 
parametrizing stable sheaves that are push forwards of line bundles on their schematic supports is a non empty Zariski open subset.
\end{lm}
\begin{proof}
An application of Nakayama's Lemma shows that a pure dimension $1$ sheaf $\m F$ on $S$  is a line bundle 
on its schematic support if and only if every non zero fiber of $\m F$ has rank $1$. The predicated openness then follows from the upper semicontinuity of the rank of the fibers of a coherent sheaf, while the non-emptiness is clear  for the general fiber, where the curves are integral.
%A coherent sheaf $\m F$ on $S$  is a line bundle 
%on its schematic support if and only if every non zero fiber of $\m F$ has rank $1$. Indeed, by Nakayama's Lemma, if the fiber of $\mc F$ at a point $x \in S$ is non zero and $\Gamma$ denotes the scheme theoretic support of $\mc F$, then, locally around $x$ there is a surjective morphism $\alpha: \mc O_\Gamma \to \mc F$.  The kernel of $\alpha$ is pure of dimension $1$, so it is either trivial, or it is supported on a subcurve of $D \subset \Gamma$. If this is the case,  then the ideal sheaf of $D$ in $\mc O_\Gamma$ acts trivially on $\mc F$, which is a contradiction to $\Gamma$ being the schematic support of $\mc F$. The predicated openness now follows from the upper semicontinuity 
%of the rank of the fibers of a coherent sheaf, while the non-emptiness is clear  for the general fiber, where the curves are integral.
\end{proof}

We remark that while this open set $P_v (S)$ intersects non trivially the fibers corresponding to an integral curve \cite{re80}, it may not intersect some  fibers of the Le Potier morphism over the non-reduced locus (cf. Corollary \ref{odd not loc free} ).

The following lemma describes the sheaves corresponding to closed point of $P_{v}(S)$  in the cases we are interested in, i.e. for $v= (0,2,\chi)$. More general statements along these lines  appear in Section \ref{secnonred} (cf. Proposition \ref{comparingstab}, Lemma \ref{possiblesurjections}, and Proposition \ref{restriction of F to C}).

\begin{lm}\la{fibratilineari} 
Let $\mc F$ be a  coherent  sheaf on $\ks$ with $v(\mc F)=(0,2,\chi)$.
Assume that $\mc F$ is the push forward of a line bundle on a curve $C\in|2H|$.
\begin{enumerate}\item{If $C=2C'$ for a curve $C'\in |H|$, then the Euler characteristic $\chi$ is even, $\mc F$ is stable, and the degree of the restriction of $\mc F$ to $C'$ is ${\rm deg}(\mc F_{|C'})=(\chi/2)+2$.}
\item{
\begin{enumerate}\item{If $C=C_1+C_2$, for $C_{1}\ne C_{2}\in |H|$  and $\chi$ 
is even, then the sheaf $\mc F$ is stable 
if and only if 
${\rm deg}(\mc F_{|C_1})={\rm deg}(\mc F_{|C_2})=(\chi/2)+2$. }
\item{If $C=C_1+C_2$, for $C_{1}\ne C_{2}\in |H|$,  and $\chi$ 
is odd, the sheaf $\mc F$ is stable 
if and only if 
${\rm deg}(\mc F_{|C_1})=2+(\chi+1)/2$ and ${\rm deg}(\mc F_{|C_2})=2+(\chi-1)/2$ or viceversa.}
\end{enumerate}}
\end{enumerate}
\end{lm}
\begin{proof}
Let $L$ be the line bundle on $C$ whose push forward on $\ks$ is $\mc F$. In case 1) let $i:C'\rightarrow C$ be the inclusion 
and let $\mathcal{I}$ be the sheaf of ideals of $C'$ in $C$.
Since $L$ is a line bundle on $C=2C'$, tensoring by $L$ the exact sequence of $\mathcal O_C$-modules
$0 \to \mathcal{I}  \to  \mathcal{O}_{C} \to i_{*}\mathcal{O}_{C'} \to 0$
gives the exact sequence
$$0 \to \mathcal{I}\otimes i_{*} L_{|C'} \to  L \to i_{*} L_{|C'} \to 0.$$
Note that since $\mc I^2=0$, $\mc I$ is an $\mc O_{C'}$--module which is isomorphic to $\omega_{C'}^\vee \cong \mc O_{C'}(-C')$ and has degree $-2$ on $C'$ (cf. (\ref{fascio I})). It follows that $\chi=\chi(\mc F)=\chi(L)=2\chi(L_{|C'})-2=2\chi(\mc F_{|C'})-2$, hence  $\chi$ is even and $\chi(\mc F_{|C'})=(\chi/2)+1$. Formula (\ref{stablb}) becomes $\chi(\mc F)<2\chi(\mc F _{|C'})$, so the sheaf $\mc F$ is stable and, since the arithmetic genus of $C'$ is $2$, we conclude that  ${\rm deg}(\mc F_{|C'})=
\chi(\mc F_{|C'})+1= (\chi /2)+2$.

In case 2), by formula (\ref{stablb}), the sheaf $\mc F$ is stable if and only if $ \chi=\chi(\mc F)<2\chi(\mc F_{|C_{i}})$ for $i=1,2$.
Using that the arithmetic genus of $C$ is $5$ 
and the arithmetic genus of $C_i$ is $2$, the last inequality  becomes
${\rm deg}(\mc F)-4 <2({\rm deg}(\mc F_{|C_i})-1)$ or equivalently 
${\rm deg}(\mc F_{|C_i})>({\rm deg}(\mc F)/2)-1$.
Since ${\rm deg}(\mc F_{|C_1})+{\rm deg}(\mc F_{|C_2})=
{\rm deg}(\mc F)=\chi-4$ item (2) follows. 
\end{proof}

\begin{cor} \label{Pic e P}\,

For $v=(0,1,-1)$ (respectively, $v=(0,2,-4)$), the open set
\[
P_{v}(S)\subset M_{v}(S)
\]
can be identified with the relative degree-$0$ Picard scheme
${\rm Pic}^0_{\m{C}'/|H|}$ of the family   
$\m{C}'/|H|$  of  curves of the linear system $|H|$ (respectively, the relative degree-$0$ Picard scheme 
${\rm Pic}^0_{\m{C}/|2H|}$ of the family   
$\m{C}/|2H|$  of  curves of the linear system $|2H|$).

\end{cor}
\begin{proof}
We prove the statement for $v=(0,2,-4)$, since the proof for $v=(0,1,-1)$ is analogous.
Since $|2H|$ is an ample linear system  on a K3 surface, for every $C\in |2H|$
the space $H^{0}(C,\mathcal{O}_{C})$ has dimension $1$.
This implies that the family  $\m{C}/|2H|$ is cohomologically flat, i.e., in the language of  \cite{AK}, the structure sheaf
$\mathcal{O}_{\mathcal{C}}$ of $\mathcal{C}$ is $|2H|$-simple.   
 
By \cite[Cor. 7.6]{AK} the \'etale sheafification of the
relative Picard functor of the family  $\m{C}/|2H|$ 
is represented by an algebraic space and the same holds for 
the \'etale sheafification of the  relative degree-$0$ Picard 
functor by \cite[Cor. 15.6.5]{EGAIV(3)} or \cite[\S 8.4, Thm 4]{Bo-Lu-Ra}. 
The representing object of the latter functor is, by definition, the 
nonsingular algebraic space ${\rm Pic}^0_{\m{C}/|2H|}\rightarrow |2H|$.
By \cite[\S 9.3, Cor. 13]{Bo-Lu-Ra} closed points of 
${\rm Pic}^0_{\m{C}/|2H|}$ correspond to line bundles on curves of $|2C|$
having degree zero on every integral subcurve.
By Lemma \ref{fibratilineari} the nonsingular variety $P:=P_{(0,2,-4)}(S)$
parametrizes the same set of line bundles seen as sheaves on $\ks$ and, 
since $P$ is an open subset 
of $M_{(0,2,-4)}(S)$, it corepresents the corresponding moduli functor by \cite[Thm 4.3.4]{HL}.
It follows that there exists a bijective morphism $\varphi:{\rm Pic}^0_{\m{C}/|2H|}\rightarrow P_{(0,2,-4)}(S)$; since 
the algebraic spaces
${\rm Pic}^0_{\m{C}/|2H|}$ and $P_{(0,2,-4)}(S)$ are smooth, the morphism  $\varphi$ is an isomorphism and the algebraic space ${\rm Pic}^0_{\m{C}/|2H|}$ is identified with the scheme $P_{(0,2,-4)}(S)$.
\end{proof}

%\subsection{The varieties $\ogm, \ogmm, \ogn, \p, \ogb, $ and $\ogmp,  \ppr, \ogbp$}\la{fissnot}$\;$

To fix ideas, in the remainder of the paper we work with the   the following  varieties and morphisms, which have been introduced in the previous sections:

\beq \la{notm}
\begin{aligned}
&\pogmmp: \ogmp:= M_{(0,1,1)}(S) \longrightarrow \ogbp:=|H| \\
&\pogm: \ogm:= \w{M}_{(0,2,2)} \stackrel{\bogm}{ \longrightarrow} \ogmm:=M_{(0,2,2)}  \stackrel{\pogmm}{ \longrightarrow} \ogb :=|2H|\\
&\pogn: \ogn:= M_{(0,2,1)}   \longrightarrow   \ogb\\
& \pogmp: \p:= {\rm Pic}^0_{\m{C}/|2H|}  \longrightarrow \ogb \\
&\pogmpr: \ppr := {\rm Pic}^0_{\m{C}'/|H|}  \longrightarrow |H|.
\end{aligned}
\eeq

%%%\beq
%%%\xymatrix{
%%%\pogmmp: \ogmp \ar[r]  & \ogbp & &   := M_{(0,1,1)}(S)  \ar[r] & |H|, \\
%%%\pogm: \ogm \ar[r]^-{\bogm} & \ogmm \ar[r]^-\pogmm & \ogb &:= \w{M}_{(0,2,2)} \ar[r]  &M_{(0,2,2)} \ar[r] &
%%%|2H|,\\
%%%\pogn: \ogn  \ar[r] &  \ogb &&:= M_{(0,2,1)}  \ar[r] & |2H|,\\
%%%\pogmp: \p  \ar[r] &  \ogb &&:= {\rm Pic}^0_{\m{C}/|2H|}  \ar[r] & |2H|,\\
%%%\pogmpr: \ppr  \ar[r] &  \ogbp &&:= {\rm Pic}^0_{\m{C}'/|H|}  \ar[r] & |H|.\\
%%%} 
%%%\eeq

\begin{rmk}\la{tuttook}
Everything we say in the remainder of the paper about $\ogm$, $\ogn$, and $\ogmp$ can be formulated also for any other moduli space with Mukai vector equal to $2v'$, $w$, and $v'$ 
as in (\ref{vprimo}, \ref{2vprimo}, \ref{w}), respectively.
\end{rmk}

\subsection{The manifolds $\ogm, \ogn, \ogmp$ and ${\ogmp}^{(2)}$ as  $\delta$-regular weak abelian fibrations}\la{sonowaf}$\;$
The goal of this section is to prove the forthcoming Proposition \ref{tadaw}, to the effect that the  triples
listed there
are $\delta$-regular weak abelian fibrations as in Definition \ref{waf}. 
Recall the notation (\ref{notm}).
We start by lifting the action of $\p$ on $\ogmm$ to an action of $\p$ on $\ogm$.

\begin{lm} \label{actionPonM}
{\rm  ({\bf The action of the Picard group})
}
Tensoring a sheaf by a line bundle with the same 
Fitting support induces:
\begin{enumerate}\item{ a group scheme structure on $P$  over $B$ with actions  $a_{\ogmm}:P \times_B \ogmm \to  \ogmm $ and $a_{\ogn}:P \times_B \ogn \to  \ogn $, on $M$ and $N$, over $B$,}
\item{
a group scheme structure on $\ppr$  over $\ogbp$ with an action  $a_{\ogmp}:\ppr \times_{\ogbp} \ogmp \to  \ogmp $ on $\ogmp$ over $\ogbp$.
}
\end{enumerate}

\end{lm}

\begin{proof} %This seems standard. We include a proof for lack of a reference.
This is a general statement.
Tensoring a pure dimension one sheaf whose schematic support is contained in a curve $C$  by an element of ${\rm Pic}^0(C)$, does not change the Euler characteristic of the sheaf itself, nor of its quotient sheaves.  By Definition \ref{defstab}, (semi)stability is preserved by this operation. The fact that this set theoretic action is algebraic follows from the fact that the fiber product of the Picard scheme and the moduli space corepresents the appropriate product functor. 

\end{proof}

The following corollary follows immediately, since $b: \ogm \to \ogmm$ is the blow up of a locus which is left invariant  by the action of the group scheme.

\begin{cor}\la{liftact}
The action $\p\times_\ogb \ogmm \to \ogmm$ lifts to a, necessarily unique, action
$\p\times_\ogb \ogm \to \ogm.$ 
\end{cor}

At this stage it is  clear that the requirements
(a)--(d) in the Definition \ref{waf} of weak abelian fibration  are met.  The following two lemmata ensure that
the remaining requirements (e) and (f) are met as well.

\begin{lm}\la{taat}  Let the notation be as in Lemma \ref{loc free naka}.
The open subsets 
\[
P_{(0,2,2)} \subset \ogm, \quad \text{respectively} \quad P_{(0,1,1)} \subset \ogmp,  
\]
surject onto $B$, respectively $B'$, and are torsors under the group scheme $P$, respectively $P'$.  The image of the open subset $P_{(0,2,1)} \subset \ogn$ in $B$ is $\Sigma \setminus \Delta$ and $P_{(0,2,1)}$ is a torsor under the restriction of $P$ to  $\Sigma  \setminus \Delta$.
Moreover, there is no open set of $N$ which is a $P$-torsor.
\end{lm}

\begin{proof}
 The proof uses results proved in the forthcoming \S\ref{sect irr comp}-\ref{secnonred}. Since the curves in $|H|$ are integral (cf. Remark \ref{nodeandcusps}), the surjectivity of $P_{(0,1,1)}  \to B'$ follows from Rego's  result on the irreducibility of the moduli space of rank $1$ torsion free sheaves on an integral locally planar curve \cite{re80}.
 The surjectivity of $P_{(0,2,2)} \to B $ follows from Proposition \ref{componentisigma} and Proposition \ref{even case}. It is clear that the action of $P$ (and of $P'$) on the moduli spaces is free on the open invariant sets parametrizing sheaves that are line bundles on their scheme theoretic support. 
 The assertion about the image of $P_{(0,2,1)} \to B$ follows from Proposition \ref{componentisigma} and Corollary \ref{odd not loc free}. The last assertion follows from Remark \ref{NDelta non reduced} and Corollary \ref{stabrk2}.
 \end{proof}

\begin{lm}\la{affstab} The actions of $P$ on $\ogm$, $\ogmm$ and $\ogn$, and of $P'$ on $\ogmp$ and on $\ogmp^{(2)}$ have affine stabilizers. 
\end{lm}
\begin{proof}  Recall that the action of the group scheme on the moduli spaces
of sheaves is given by tensorization. We start by considering the following case. 
Let $F$ be a rank $r$ vector bundle on a smooth projective curve $C$ and suppose $F \cong F \otimes L$, for $L \in {\rm Pic}^0(C)$. Taking determinants, we see that $L^r\cong \mc O_C$ so $L$ belongs to a finite subgroup of ${\rm Pic}^0(C)$. Now suppose that $F$ is a pure (see Definition \ref{defstab}) sheaf on a curve $C=\sum C_i$, where $C_i$ are the irreducible, but possibly non-reduced, components of $C$. If $L \in {\rm Pic}^0(C)$ is such that $F \cong F \otimes L$, then we have that
\[
F_{|\widetilde C_i^{red}}\slash Tor \cong F_{|\widetilde C_i^{red}}\slash Tor \otimes L_{|\widetilde C_i^{red}},
\]
where $j_i: \widetilde C_i^{red} \to C_i^{red} \to C_i$ is the normalization of the reduced underlying curve and where $F_{|\widetilde C_i^{red}}\slash Tor$ is the pullback of $F$ to $\widetilde C_i^{red}$, modulo its torsion. Since $F_{|\widetilde C_i^{red}}\slash Tor$ is a vector bundle on a smooth curve, we can apply the previous observation to conclude that there is an integer $r_i$ such that $L_{|\widetilde C_i^{red}}^{r_i}\cong \mc O_{\widetilde C_i^{red}}$. Hence we conclude that, up to a finite quotient group, $L$ lies in the kernel of the natural morphism:
\[
{\rm Pic}^0(C) \stackrel{\prod j_I^*}{\longrightarrow} \prod {\rm Pic}^0(\widetilde C_i^{red}),
\]
which is precisely the affine part of ${\rm Pic}^0(C)$. From these observations it follows that the points of $\ogmm$,  $\ogn$,  or $\ogmp$ corresponding to stable sheaves all have affine stabilizers. Hence, the points of $ \ogmp^{(2)}$ also also have affine stabilizers. 
 The same reasoning, applied to a polystable sheaf $F=F_1 \oplus F_2$, shows that the points corresponding to the singular locus of $\ogmm$ also have affine stabilizers. Indeed, 
 %by the uniqueness of the isomorphism class of a polystable sheaf in each S-equivalence class, 
 the stabilizer of a point corresponding to a polystable sheaf is contained in the automorphism group of the polystable sheaf, which is a product of affine linear groups.
By Corollary \ref{liftact} the morphism $b: \ogm \to \ogmm$ is equivariant, so the stabilizer of a point $ m \in \ogm$ is contained in the stabilizer of $b(m) \in \ogmm$. Thus if the points of $\ogmm$ have affine stabilizers, so do the points of $\ogm$.
\end{proof}

\begin{lm}\la{plariz}
{\rm
({\bf Polarizability of Tate module})
}
 The $\oql$-adic counterparts $T_{{\rm et}, \oql} (-)$ of the Tate module   $T(-)$
 associated with $\p \to \ogb$ --the relative ${\rm Pic}^0$ for $\m{C} \to \ogb$--, 
and with
$\ppr \to \ogbp$ --the relative ${\rm Pic}^0$ for $\m{C'} \to \ogbp$-- are polarizable.
The same is true if we restrict the families of curves to any subfamily of curves over a locally closed subvariety
of $\ogb$, or  of $\ogbp$.
\end{lm}
\begin{proof}
The polarizability result \ci[Theorem 3.3.1]{de17}, which  is stated for  the family of spectral curves for the $GL_n$-Hitchin system, is in fact proved  for any  family of curves obtained via base change from a linear system
of curves on a nonsingular surface.
%for a  proper and flat family of  locally planar  curves (i.e. \'etale locally given by one equation on a nonsingular surface).
%that is proper flat, with geometrically connected fibers, nonsingular total space, and with nonsingular general fiber. 
Therefore, the  polarizability result holds for  the families of curves  $\m{C}'\to \ogbp$ and  $\m{C} \to \ogb$.  \end{proof}

By putting together the preceding contents of this section, we get the following
\begin{pr}\la{tadaw}
{\rm
({\bf $\delta$-regularity})
}
The triples:  
\[(\ogm, \ogb, P), \; (\ogn, \ogb, \p),  \; (\ogmp, \ogb, \ppr), \; \left(\ogmp^2, \ogbp^2, \ppr^2\right), \;
 \left({\rm Sym}^2 \ogmp, {\rm Sym}^2 \ogbp, {\rm Sym}^2 \ppr\right)
\] 
are $\delta$-regular weak abelian fibrations satisfying the assumptions of Theorem \ref{nstabis} (Direct image is sum of Ng\^o strings). The group schemes in the triples  appear in, or are related to,   (\ref{epigs}). 
\end{pr}

\begin{proof}  In view of Lemma \ref{actionPonM}, Corollary \ref{liftact}, and Lemma \ref{affstab}, all the triples above are weak abelian fibrations and we only need to prove $\delta$-regularity. 
In the cases  $(\ogm, \ogb, P)$ and   $(\ogmp, \ogb, \ppr)$, $\delta$-regularity follows  from
Proposition \ref{deltaregularity}
 and Lemma \ref{taat}. In the cases 
$\left(\ogmp^2, \ogbp^2, \ppr^2\right)$ and 
$\left({\rm Sym}^2 \ogmp, {\rm Sym}^2 \ogbp, {\rm Sym}^2 \ppr\right)$,
$\delta$-regularity  follow formally from the $\delta$-regularity of $(\ogmp, \ogb, \ppr)$.
In the case  $(\ogn, \ogb, \p)$ it is not true that $N$ contains a $P$-torsor over $\ogb$,
 so that this case is not covered directly by Proposition \ref{deltaregularity}. On the other hand, 
 $\delta$-regularity is a property of the group scheme $P$, and this has already  been  established 
 when dealing with $\left(\ogm, \ogb, \p\right)$.
\end{proof}

\section{The top direct image sheaf \texorpdfstring{$R^{10}$}{R10}  for  the Lagrangian fibrations 
\texorpdfstring{$\ogm, \ogn \to \ogb$}{M,N to B}}\la{sec r10}

The comparison of the derived direct images of the constant sheaves of the two fibrations $\ogm, \ogn \to \ogb$ involves understanding the direct summands of these derived direct images, as well as their supports. Thanks to Ng\^o's Support Theorem, it is often enough to understand the top degree (in this case, degree $10$) direct image sheaves. 
The aim of this section is to prove Proposition \ref{rtop}, which describes, over certain relevant subvarieties of $\ogb$,  the top  direct image sheaves of the constant sheaf for the two fibrations. This result will be used in \S \ref{nonon} to determine the Decomposition Theorem for the two fibrations. 
% The coupling of this  proposition with Ng\^o's Support Theorem, is one key to our proof of the main Theorem A of this paper.
Note that as a simple consequence  of the structure of the symplectic resolution $\ogm \to \ogmm$ (Fact \ref{blowup}),  Corollary \ref{dtbr} immediately singles out three subvarieties  of $\ogb$ as supports for the Lagrangian fibration $\ogm \to \ogb$. However, to show that these subvarieties are the {only} supports of the fibrations and to determine the summands on each support actually requires Proposition \ref{rtop}.
A first step towards this proposition  is  to identify the irreducible components (as well as their monodromy) of the fibers  of  the equidimensional fibrations $\ogm ,\ogn \to \ogb$.  This identification is done in \S \ref{sect irr comp} and uses the results developed in \S \ref{secnonred}.

\subsection{The Decomposition Theorem for the blow up $\bogm: \ogm \to  \ogmm$}\la{tbu}$\;$

%%{\cdg Adding to the results already recorded in Theorem \ref{thmpoly}, we recall some more results 
%%on the symplectic resolution $\ogm \to \ogmm$. We use these to show in Corollary \ref{dtbr} that $\Sigma$ and $\Delta$ are supports for the Decomposition Theorem of $\ogm \to \ogb$.}

By Theorem \ref{thmpoly} and Remark \ref{SigmaDelta}, we have the following commutative diagram
\beq\la{bu1}
\xymatrix{
& \ogmp=\Delta_{\ogmp} \ar[d]  \ar[r]  & \Delta:=\Delta_{\ogbp}=\ogbp  \ar[d] \\
 & \hskip-13mm  {\rm Sing}(\ogmm)={\rm Sym}^2 \ogmp  \ar[r] \ar[d] & \Sigma := {\rm Sym}^2 \ogbp \ar[d]  \\
\pogm: \ogm \ar[r]^-\bogm &  \ogmm \ar[r]^-\pogmm & \ogb=|2H|,
}
\eeq
where the morphisms $\pogm$ and $\pogmm$  are as in (\ref{notm}) and the morphism $\bogm$ is the blow up of $\ogmm$ along its singular locus
${\rm Sing}(\ogmm) \simeq  {\rm Sym}^2 \ogmp$ (with the reduced induced structure). 
%%%%Recall that: 
%%%%\begin{enumerate}
%%%%\item{by Remark \ref{SigmaDelta}, the locus $\Sigma \subset \ogb=|2H|$ parametrizing non-integral curves in $\ogb=|2H|$  
%%%%is naturally identified with ${\rm Sym}^2 |H|= {\rm Sym}^2 \ogbp \subseteq \ogb$, and the non-reduced curves are parameterized by the diagonal $\Delta\simeq |H|=\ogbp$;}
%%%%\item{by Theorem \ref{thmpoly} (1), the singular locus ${\rm Sing} (\ogmm)$ parametrizes polystable sheaves of the form $F\oplus G$,  with $F,G\in \ogmp$ and is naturally identified with ${\rm Sym}^2 \ogmp$. The singular locus  of   
%%%%${\rm Sing} (\ogmm)$ is the diagonal $\Delta_{\ogmp}=\ogmp$ of  ${\rm Sym}^2 \ogmp$;}
%%%%\item{using these identifications the restriction of $\pogmm: \ogmm \to \ogb$ to the singular locus
%%%%${\rm Sing} (\ogmm)$ is the map  ${\rm Sym}^2 \pogmmp: {\rm Sym}^2 \ogmp\rightarrow {\rm Sym}^2 \ogbp$ induced by $\pogmmp$ on the second symmetric products and the restriction of $m$ to 
%%%%$\Delta_{\ogmp}$ is the Lagrangian fibration $\pogmmp: \ogmp\rightarrow \ogbp=\Delta_\ogbp$.  }
%%%%\end{enumerate} 
%The structure of the symplectic resolution of $OG10$-type irreducible holomorphic symplectic manifolds has been studied in \cite{Lehn-Sorger}.
The study of the local analytic structure of the singularities of $\ogmm$ from \cite{Lehn-Sorger} gives the following:

\begin{fact}\la{blowup} {\rm ({\bf The symplectic resolution  $\bogm: \ogm \to  \ogmm$})}\cite[\S 4]{Lehn-Sorger}.
%%We have the following commutative diagram of morphisms:
%%\beq\la{bu1}
%%\xymatrix{
%%& \ogmp=\Delta_{\ogmp} \ar[d]  \ar[r]  & \Delta:=\Delta_{\ogbp}=\ogbp  \ar[d] \\
%% & \hskip-13mm  {\rm Sing}(\ogmm)={\rm Sym}^2 \ogmp  \ar[r] \ar[d] & \Sigma := {\rm Sym}^2 \ogbp \ar[d]  \\
%%\pogm: \ogm \ar[r]^-\bogm &  \ogmm \ar[r]^-\pogmm & \ogb=|2H|,
%%}
%%\eeq
%%where the morphisms $\pogm$ and $\pogmm$  are as in (\ref{notm}) and the morphism $\bogm$ is the blow up of $\ogmm$ along its singular locus
%%${\rm Sing}(\ogmm) \simeq  {\rm Sym}^2 \ogmp$ (with the reduced induced structure). 
The morphism $\bogm$ is semismall with irreducible exceptional divisor $E \subset \ogm$. The fibers of $E \to {\rm Sing}(\ogmm) $ are as follows: 

%%\ben
%%%\item the morphism $\bogm$ is the blowing up of $\ogmm$ along ${\rm Sing} \ogmm= {\rm Sym}^2 \ogmp$;
%% 
%%\item the morphisms $\pogm$ and $\pogmm$  are as in (\ref{notm});
%%
%%\item the vertical arrows are the natural closed embeddings;
%%\item
%%the morphism $\bogm$ is the blow up of $\ogmm$ along its singular locus
%%${\rm Sing}(\ogmm) \simeq  {\rm Sym}^2 \ogmp$ (with the reduced induced structure) and it is a semismall morphism with exceptional divisor $E \subset \ogm$. The fibers of the morphism $E \to {\rm Sing}(\ogmm) $ are as follows: 

\ben
%\item over the points of $\ogmm \setminus {\rm Sym}^2 \ogmp$, the fiber is a point; 
\item [(a)]
the morphism $E_{{\rm Sym}^2 (\ogmp) \setminus \Delta_{\ogmp}}  \to  {\rm Sym}^2 (\ogmp) \setminus \Delta_{\ogmp}$
is an analytic fiber bundle with fiber $\pn{1}$; 

\item [(b)]
the morphism $E_{\Delta_{\ogmp}}  \to  \Delta_{\ogmp}$
is an analytic fiber bundle with fiber a three dimensional smooth quadric.
\een
%\een
\end{fact}
%For (a) and (b), see Section 2 of \cite{Lehn-Sorger}.
%%We will need two consequences of this fact. The first is concerned with the irreducible components of the fibers of  the morphisms $\ogmm \to \ogb$, ${\rm Sing}(\ogmm) \to {\rm Sym}^2 (\ogb)$, and $E \to {\rm Sing}(\ogmm) \to  {\rm Sym}^2 (\ogb)$.
%%This is recorded in the following remark and of Proposition \ref{}, which is used in Proposition \ref{rtop}. The second consequence concerns the Decomposition Theorem for $\ogmm \to \ogb$ and is the content of Corollary \ref{}.
For later use, we highlight the following consequence of diagram (\ref{bu1}).
 \begin{rmk}\la{exCompecc} 
A point $b\in\Sigma\setminus \Delta$ corresponds to two distinct point $b_1, b_2 \in \ogbp$; the strictly semistable locus 
${\rm Sing} (\ogmm)_{b}= ({\rm Sym}^2 \ogmp)_b$ of the fiber $\ogmm_{b}:=\pogmm^{-1}(b)$ 
is isomorphic to the product $\ogmp_{b_1}\times \ogmp_{b_2}$ of the corresponding fibers of $m': M' \to \ogbp$.
In particular,  ${\rm Sing} (\ogmm)_{b}$ is  irreducible of dimension $4$.

Analogously, for $b\in\Delta$ the strictly semistable locus 
${\rm Sing} (\ogmm)_{b}$ of the fiber $\ogmm_{b}$ is isomorphic to
${\rm Sym}^2(\ogmp)_b$ (recall the identification
$\Delta\simeq \ogbp$). It follows that  ${\rm Sing} (\ogmm)_{b}$ is irreducible of dimension $4$.
% so that its complement in $\ogmm^I_b$ is open and dense, hence irreducible.
The strictly semistable locus of $\ogmm_{\Delta}:=\pogmm^{-1}(\Delta)$ is  
%${\rm Sym}_{\Delta}^2(\ogmp)\subset {\rm Sym}^2(\ogmp)$, i.e. 
the quotient of the fiber product 
$\ogmp \times_{\Delta} \ogmp$ by the involution exchanging the factors.  
For $b \in \Delta$, the intersection $\ogmm_{b}\cap\Delta_{\ogmp}$ has a unique irreducible component isomorphic to $\ogmp_{b},$ which is irreducible of dimension 2.
\end{rmk}

%%%\begin{rmk}\la{Compecc}
%%%By Fact \ref{blowup}.(4) and Remark \ref{exCompecc}, for every point $b \in \Sigma \setminus \Delta$, 
%%%the fiber $E_b$ is  a $\pn{1}$-bundle over the irreducible $({\rm Sym}^2 \ogmp)_b$ and hence is irreducible.
%%%
%%%By  Fact \ref{blowup}.(4)(a,b), the variety 
%%%$E_{\Delta}$
%%%has two irreducible components 
%%%denoted by  $E^I_\Delta:=E_{\Delta_{\ogmp}}$
%%%and $E^{I\!I}_{\Delta}: =  \ov{\bogm^{-1}({\rm Sym}^2 \ogmp)_\Delta \setminus \Delta_{\ogmp}}$.
%%%
%%%The projective morphism $E^I_\Delta \to \Delta$ has irreducible fibers since it   factors as the composition   $E^I_\Delta \to \Delta_{\ogmp} \to \Delta$ of two morphisms with irreducible fibers. 
%%%
%%%The second component $E^{I\!I}_{\Delta}$ has an open dense subset  
%%%$E^{I\!I}_{\Delta}\setminus E^{I}_{\Delta}$ that  is a $\pn{1}$-bundle over 
%%%${\rm Sym}_{\Delta}^2(\ogmp)\setminus\Delta_{\ogmp}$ and that maps flatly over $\Delta$:
%%%in fact,  the Lagrangian fibration
%%%$\ogmp\rightarrow \Delta$ is flat, so that the relative symmetric product  ${\rm Sym}_{\Delta}^2(\ogmp)\setminus\Delta_{\ogmp}$ is flat over $\Delta$.
%%%For every $b \in \Delta$,
%%%${\rm Sym}^2 \ogmp_b \setminus \Delta_{\ogmp_b}$  is irreducible, hence, the  $\pn{1}$-bundle
%%%$E^{I\!I}_b$ over it is irreducible.
%%%\end{rmk}

%%%%%%%%%%%%%%%%%%%%%%%%%%%%%%%%%%%%%

We can now determine the Decomposition Theorem for $\ogm \to \ogmm$.
\begin{lm}\la{dtb} {\rm ({\bf Decomposition Theorem for the blow up $\bogm: \ogm \to \ogmm$})}
There is a canonical isomorphism in $D^b MHM_{alg}(\ogmm)$ and in  $D^b(M,\rat)$ (turn off the Tate shifts):
\beq\la{dtbeqrat00}
R \bogm_* \rat_{\ogm}  \simeq
\ms{IC}_{\ogmm} \bigoplus  \rat_{{\rm Sym}^2 \ogmp}[-2]  (-1)  \bigoplus \rat_{\Delta_{\ogmp}}[-6](-3),
\eeq
\end{lm}
\begin{proof}
By combining Fact \ref{blowup} with the Decomposition Theorem for semismall morphisms (cf. the survey \ci{bams}, and references therein),
we obtain a canonical isomorphism:
\beq\la{dtbeq}
R \bogm_* \rat_{\ogm} \simeq
\ms{IC}_{\ogmm} \bigoplus \ms{IC}_{{\rm Sym}^2 \ogmp} [-2] (-1) \bigoplus \ms{IC}_{\Delta_{\ogmp}}[-6](-3).
\eeq
Since $\ogmp$ is nonsingular,   the intersection complexes appearing as the second and third summand are actually the constant sheaf with the corresponding dimensional shift.
\end{proof}

\begin{cor}\la{dtbr} There is a canonical isomorphism in $D^b MHM_{alg}(S)$
and in  $D^b(S,\rat)$:
\beq\la{dtbreq}
R \pogm_* \rat_{\ogm} \simeq
R\pogmm_* \ms{IC}_{\ogmm} \bigoplus R\pogmm_* \rat_{{\rm Sym}^2 \ogmp}[-2] (-1)  \bigoplus R\pogmm_* \rat_{\Delta_{\ogmp}}[-6](-3).
\eeq
In particular, the subvarieties 
$\ogb, \Sigma, \Delta$ of $\ogb$ are among the   supports of $R \pogm_* \rat_{\ogm}$ on $\ogb$.
\end{cor}
\begin{proof}
Apply $R\pogmm_*$ to (\ref{dtbeq}) and obtain (\ref{dtbreq}). Each direct summand in (\ref{dtbreq}) is the direct image
of the intersection complex of a variety  that surjects properly onto $\ogb, \Sigma$ and $\Delta$, respectively; the statement
about the supports follows from this.
\end{proof}

\begin{rmk}\la{nomo}
The forthcoming Proposition \ref{dtwe} shows 
that $\ogb, \Sigma$ and $\Delta$ are the only supports  of $R \pogm_* \rat_{\ogm}$ on $\ogb$. 
\end{rmk}

\subsection{Irreducible components of the fibers: main results}\la{sect irr comp}$\;$
In view of proving Proposition \ref{rtop}, which is concerned with the top degree direct image sheaves  $R^{10} \pogm_* \rat_{\ogm}$ and $R^{10} \pogn_* \rat_{\ogn}$, in this subsection we study the irreducible components of the fibers of the fibrations 
 $\ogmm , \ogn\rightarrow \ogb$ over the loci $\Sigma=\Sym^2|H|$ and $\Delta=\Delta_{|H|}$ of $B=|2H|$.
 
The main results are: Proposition \ref{compmtildedelta}, dealing with the components of the fibers of $\ogm$ that arise from the exceptional divisor $E$; Proposition \ref{componentisigma}, on the structure of the fibers of $\ogmm$ and $\ogn$ over the locus  $\ogb \setminus \Delta$ of  reduced  curves; and finally Proposition \ref{propriass}, on the structure of the fibers of $\ogmm$ and $\ogn$ over the locus  $\Delta$ of  non-reduced curves.
The proof of this last proposition rests on Propositions \ref{even case} and \ref{odd case}, which are the main results of  the forthcoming \S \ref{irr comp II}.

%%%
%%%
%%% is Proposition \ref{propriass} which combines and records in one place the following results:
%%%Proposition \ref{componentisigma}, on the structure of the fibers over the locus  $\ogb \setminus \Delta$ of  reduced curves;
%%%equation (\ref{IeII}),  on the decomposition of $\ogmm_\Delta$ and $\ogn_\Delta$ into two parts I and II;
%%%Proposition \ref{rank 2}, on the structure of $\ogmm^I_\Delta$ and of $\ogn^I_\Delta$; 
%%%Proposition \ref{even case},  on the structure of $\ogmm^{I\!I}_\Delta$;
%%%Proposition \ref{odd case}, on the structure  of $\ogn^{I\!I}_\Delta$.
%%%
%%%Finally, Proposition \ref{compmtildedelta}, describes the  structure of $\ogm_\Delta$. {\cdg ATTENZIONE WRONG REFERENCE}

\subsubsection{Irreducible components arising from the exceptional divisor}

Let $\pogm_{E}:E \to Sing (M)=\Sym^2 \ogmp \to B$ be the restriction of $\pogm$ 
to the exceptional divisor of the symplectic resolution $b:\ogm\rightarrow \ogmm$. Set
\begin{equation} \label{defEdelta}
E_{\Delta}:=\pogm_{E}^{-1}(\Delta), \quad E_\Delta^I:=b^{-1}(\Delta_\ogmp), \quad E^{I\!I}_{\Delta}:=E_{\Delta} \setminus E_\Delta^I.
\end{equation}

\begin{pr}[{\bf The structure of the exceptional divisor over $\Delta$}]$\;$\la{compmtildedelta} \la{Compecc}
\begin{enumerate}
\item The variety $E_{\Delta}$ is the union of two projective irreducible varieties of pure dimension $7$, namely $E^{I}_\Delta$
and the closure $\ov{E^{I\!I}_{\Delta}}$ of 
$E^{I\!I}_{\Delta}$ in $E_{\Delta}$;
\item The morphism $ E_\Delta^I \to \Delta$ is projective, with irreducible $5$-dimensional  fibers;
\item The morphism $E^{I\!I}_\Delta \to \Delta$ is flat, surjective, with irreducible $5$-dimensional fibers.
\end{enumerate}

\end{pr}

\begin{proof} Part (1) follows from parts (2) and (3), which we now prove. (2) By Fact \ref{blowup} (b), the morphism $E_\Delta^I \to \Delta$ factors via a $3$--dimensional quadric fibration $E_\Delta^I \to \Delta_{\ogmp}=\ogmp$, followed by the relative compactified Jacobian $m': \ogmp \to \Delta=|H|$ of the genus $2$ linear system $|H|$. As noted in Remark \ref{nodeandcusps}, the curves in $|H|$ are reduced and irreducible and hence so are the fibers of $m'$ \cite{re80}. (3) By definition of $E^{I\!I}_{\Delta} $, there is a factorization
\[
E^{I\!I}_{\Delta} \to \Sym^2_\Delta(M') \setminus \Delta(M') \to \Delta
\]
where $\Sym^2_\Delta(M')=(\Sym^2(M'))_\Delta$ is the quotient of the fiber product $\ogmp \times_\Delta \ogmp$ by the involution interchaging the two factors. By Fact \ref{blowup} (a), $E^{I\!I}_{\Delta} \to \Sym^2_\Delta(M') \setminus \Delta(M')$ is a fibration in $\mathbb P^1$ and, by Remark \ref{piattezza}, $\Sym^2_\Delta(M')\setminus \Delta(M') \to \Delta$ is flat surjective with irreducible $4$--dimensional fibers.

\end{proof}

%%%%%%%%%%%%%%%%%%%%%%%%%
%%%%%%%%%%%%%%%%%%%%%%%%%
%%%%%%%%%%%%%%%%%%%%%%%%%
%%%%%%%%%%%%%%%%%%%%%%%%%
%%%%%%%%%%%%%%%%%%%%%%%%%
%%%%%%%%%%%%%%%%%%%%%%%%%
%%%%%%%%%%%%%%%%%%%%%%%%%
%%%%%%%%%%%%%%%%%%%%%%%%%
%%%%%%%%%%%%%%%%%%%%%%%%%
%%%%%%%%%%%%%%%%%%%%%%%%%

\subsubsection{Irreducible components over reduced curves}\la{irr comp red}$\;$

For $b\in B$ representing a reduced curve, a description of the fibers $\ogmm_{b}:=\pogmm^{-1}(b)$,
$\ogn_{b}:=\pogn^{-1}(b)$
and $\ogm_{b}:=\pogm^{-1}(b)$ can be found in \cite{ra08}. 
In the following proposition and corollary we collect the result that we will use in this paper.

\begin{pr}\la{componentisigma}
Consider a point $b\in \ogb \setminus \Delta$. The fibers $\ogmm_{b}$ and $\ogn_{b}$
are reduced, Cohen-Macauley of dimension $5$ and have dense open subsets parametrizing line bundles. The fiber
$\ogmm_{b}$ is irreducible while the fiber  $\ogn_{b}$ is irreducible  for $b\in
\ogb \setminus \Sigma$ and has exactly $2$ irreducible components for  $b\in \Sigma\setminus\Delta$.
\end{pr}

\begin{proof}
By Remark \ref{piattezza}, the morphisms $\pogmm$ and $\pogn$ are flat and the fibers are  Cohen--Macauly 
of dimension $5$.
%Since $\ogn$ is smooth and $\ogmm$ only has canonical singularities (hence it is Cohen-Macaulay, see Theorem 5.10 and Corollary 5.24 of \cite{KoMo}) and $\ogb$ is smooth,
%the fibers $\ogmm_{b}$ and $\ogn_{b}$
%are  Cohen-Macaulay of dimension $5$.
% 
For $b\in \ogb \setminus \Delta$, the points of the fibers $\ogmm_{b}$ and $\ogn_{b}$ represent  rank $1$ torsion free sheaves on the reduced locally planar curve  $C_{b}$ corresponding to $b$.
By Proposition 2.3 (ii) of \cite{mrv}, the loci
of $\ogmm_{b}$ and $\ogn_{b}$  representing stable line bundles are dense 
in $\ogn_{b}$ and in 
the stable locus of 
$\ogmm_{b}$, respectively. Since the strictly semistable locus of $\ogmm_{b}$ has dimension $4$  (see Remark \ref{exCompecc}), and since $\ogmm_{b}$  is  $5$ dimensional and Cohen-Macauley, density of the line bundle locus holds also in  $\ogmm_{b}$.
%Density also holds in $\ogmm_{b}$, which is  $5$ dimensional and Cohen-Macauley,  since its strictly semistable locus has dimension at most $4$ (see Remark \ref{exCompecc}).
By Proposition 2.8 of \cite{LePotier} the Le Potier 
morphism is smooth at every point
corresponding to  a stable   line bundle: hence $\ogmm_{b}$ and $\ogn_{b}$ have open dense subsets which are reduced.
%there arereduced open dense subsets of $\ogmm_{b}$ and $\ogn_{b}$.
Since  $\ogmm_{b}$ and $\ogn_{b}$
are Cohen-Macauley, they are reduced everywhere.

Finally, using Corollary 13 of Chapter 9 of \cite{Bo-Lu-Ra},   
every connected component $\Gamma$ of the loci of $\ogmm_{b}$ or $\ogn_{b}$ 
parametrizing stable line bundles is determined 
by the set of degrees of the restrictions of any line bundle parametrized by $\Gamma$ to the irreducible components of $C_{b}$. 
%or 
%equivalently an irreducible components of  $\ogmm_{b}$ and $\ogn_{b}$, 
%is determined by 
%equals the number of the allowed multi 
%degrees\footnote{The multi degree of a 
%line bundle $L$ on a reducible curve $C$ is 
%the function assocating to each 
%component $C'$ of $C$ the degree of the 
%restriction of $L$ to $C'$.} 
%for stable line bundles  in $\ogmm_{b}$ and $\ogn_{b}$.
If $b\in\ogb\setminus \Sigma$, then the corresponding  curve $C_{b}$ as well as the fibers $\ogmm_{b}$ and $\ogn_{b}$ are irreducible \cite{re80}. If $b\in \Sigma\setminus \Delta$, the curve $C_{b}$ has $2$ distinct irreducible components so by Lemma \ref{fibratilineari}(2) the fiber $\ogmm_{b}$ is irreducible
while the fiber $\ogn_{b}$ has two irreducible components.
 
%A line bundle $L\in\ogmm_{b} $ has even total 
%degree and,  by Definition \ref{defstab}, 
%it is stable if and only if it has  
%the same degree on each component of 
%$C_b$. Hence there is a unique allowed multi degree and  $\ogmm_{b}$.
%On the other hand, if  $L$ is a line 
%bundle in $\ogn_{b}$, Definition 
%\ref{defstab} implies that the absolute 
%value of  difference between the degrees of the 
%restrictions of $L$ to the two components 
%of $C_b$ has to be $1$. Hence there 
%are $2$  allowed multidegrees 
%and $\ogn_{b}$ has $2$ irreducible components.

\end{proof}
\begin{cor}\la{corcomp}
For $b\in \Sigma\setminus\Delta$, the fiber $\ogm_{b}$ is pure of dimension $5$ and  has exactly $2$ irreducible components.
\end{cor}
\begin{proof}
By Proposition \ref{componentisigma} the strict transform of $\ogmm_b$ in $\ogm_b$ is irreducible of dimension $5$. By Fact \ref{blowup} and Remark \ref{exCompecc} the inverse image of the strictly semistable locus ${\rm Sing} (\ogmm)_{b}$ of $\ogm_{b}$ is a $\pn{1}$-bundle over an irreducible variety of dimension $4$. 
\end{proof}

%%%%%%%%%%%%%%%%%%%%%%%%%
%%%%%%%%%%%%%%%%%%%%%%%%%
%%%%%%%%%%%%%%%%%%%%%%%%%
%%%%%%%%%%%%%%%%%%%%%%%%%
%%%%%%%%%%%%%%%%%%%%%%%%%
%%%%%%%%%%%%%%%%%%%%%%%%%
%%%%%%%%%%%%%%%%%%%%%%%%%
%%%%%%%%%%%%%%%%%%%%%%%%%
%%%%%%%%%%%%%%%%%%%%%%%%%
%%%%%%%%%%%%%%%%%%%%%%%%%
%%%%%%%%%%%%%%%%%%%%%%%%%
%%%%%%%%%%%%%%%%%%%%%%%%%
%%%%%%%%%%%%%%%%%%%%%%%%%
%%%%%%%%%%%%%%%%%%%%%%%%%
%%%%%%%%%%%%%%%%%%%%%%%%%
\subsubsection{Irreducible components over  non-reduced curves: the statement}

Next, we  analyze the irreducible components of the fibers  of $\ogmm, \ogn \to \ogb$  over  the points  $b \in \Delta \subseteq \ogb$ parametrizing non-reduced  curves.
This analysis divides the sheaves parametrized by the fibers over $\Delta$ into two types, which we call type I 
(sheaves defined on the reduction of the curves) and type II (the remaining ones).
Consider the restrictions 
\[
\ogn_\Delta \to \Delta, \quad \ogmm_\Delta \to \Delta,
\]
of the  morphisms $\ogn, \ogmm \to \ogb$ to the locus of double curves  $\Delta
\subseteq B$.
The main result of this subsection is Proposition \ref{propriass}, which  describes the irreducible components of the varieties $\ogn_\Delta$
and $\ogmm_\Delta$, as well as of their fibers  $\ogn_{\ogc}$ and $\ogmm_{\ogc}$ over the points $\ogc \in \Delta$. The proof of this proposition uses  results from \S \ref{secnonred} on certain sheaves on non-reduced curves.

%Recall that the locus  of non-reduced (double)  curves  $\Delta \subseteq \ogb = |2H| \simeq \pn{5}$ is naturally identified
%with the diagonal $\Delta_{\ogbp} \subseteq {\rm Sym}^2 \ogbp \subseteq \ogb$, where  $B'=|H|\simeq \pn{2}$.
%
Recall from Remark \ref{SigmaDelta} that every curve $\ogc \in \Delta \subseteq  |2H|$ is of the form $\ogc=2\ogcred$ for a curve $\ogcred \in |H|$.
Let us introduce the sheaves of type $I$ and $I\!I$.
The sheaves parametrized by $\ogn_\Delta$ and 
$\ogmm_\Delta$ are of pure dimension $1$ on $\ks$ and their 
Fitting supports are non-reduced curves $\ogc =2\ogcred \in \Delta$. 
%For a fixed curve $\ogc \in \Delta$, we denote by $\ogn_{\ogc} \subseteq \ogn_\Delta$ and $\ogmm_{\ogc} \subseteq \ogmm_\Delta$ the corresponding fiber over $\ogc$. 
There are two kinds of such sheaves.
The first kind are the sheaves $\mc F$  on $\ks$
for which  the compositum of the natural morphisms  $\mc O_\ks \to  \mc O_{\ogc} \to \mc{E}nd_{\ks}(\mc F)$ factors via the natural surjection $\mc O_\ogc \to \mc O_\ogcred$. We call these the sheaves of type $I$; in this case, 
 $\mc F$  is an $\mc O_\ogcred$-module and  can be viewed as a rank $2$ torsion free sheaf on the underlying reduced curve $\ogcred$;  we denote the corresponding loci by $\ogmm^I_{\ogc}$ and $\ogn^I_{\ogc}$ and, by letting $\ogc$ vary in $\Delta$, $\ogmm_{\Delta}^I$ and
$\ogn^I_\Delta$. We call the remaining sheaves of type $I\!I$,
and we set
\beq \label{IeII}
\ogmm^{I\!I}_{\ogc}:=\ogmm_{\ogc} \setminus \ogmm^{I}_{\ogc}, \quad,  \ogn^{I\!I}_{\ogc}:=\ogn_{\ogc} \setminus \ogn^{I}_{\ogc}, \quad \ogmm^{I\!I}_{\ogc}:=\ogmm_{\Delta} \setminus \ogmm^{I}_{\Delta}, \quad  \ogn^{I\!I}_{\ogc}:=\ogn_{\Delta} \setminus \ogmm^{I}_{\Delta}.
\eeq
In other words, the sheaves of type $I\!I$ on $\ogc=2\ogcred$ are those for which the Fitting support equals the schematic support.
%Recall the decompositions (\ref{IeII}). 
In the following proposition, we consider the varieties and fibers with their reduced induced structure.

\begin{pr} \label{propriass} {\rm ({\bf Irreducible components of $\ogmm_{\Delta},\ogn_{\Delta} \to \Delta$})}

 \begin{enumerate}
\item{
\begin{enumerate}\item{The variety $\ogmm_{\Delta}$ is of pure dimension $7$ and has  two irreducible components, namely $\ogmm^{I}_\Delta$
and  the closure $\ov{\ogmm^{I\!I}_{\Delta}}$ of $\ogmm^{I\!I}_\Delta$.}
%\item{For every $\ogc \in \Delta$, the fiber  $\ogmm_{\ogc}$  is of pure dimension  $5$, it has two irreducible components, namely $\ogmm^{I}_{\ogc}\subset \ogmm^{I}_\Delta$ and the closure $\ov{\ogmm^{I\!I}_{\ogc}}$ of $\ogmm^{I\!I}_{\ogc}$ in $\ogmm_{\ogc}$.}
\item{The morphisms $\ogmm^I_{\Delta}, \ogmm^{I\!I}_\Delta \to \Delta$ have irreducible $5$--dimensional fibers, namely, for $\ogc \in \Delta$, $\ogmm^I_{\ogc}$ and $\ogmm^{I\!I}_{\ogc}$.
%whereas the morphism  $(\ov{\ogmm^{I\!I}_{\Delta}})  \to \Delta$ has irreducible general fiber, and, for every $\ogc \in \Delta$, the   fiber $(\ov{\ogmm^{I\!I}_{\Delta}})_{\ogc}$ is contained in the union $\ogmm^I_{\ogc}\cup \ov{\ogmm^{I\!I}_{\ogc}}$.
}
\end{enumerate}}
\item{
\begin{enumerate}\item{The variety $\ogn_{\Delta}$ is of pure dimension $7$ and has  two irreducible components, 
namely $\ogn^{I}_\Delta$
and the closure 
$\ov{\ogn^{I\!I}_{\Delta}}$ 
of $\ogn^{I\!I}_{\Delta}$.}
%\item{For every $\ogc \in \Delta$, the fiber  $\ogn_{\ogc}$ is of pure dimension  $5$, it  has  two irreducible components, namely $\ogn^{I}_{\ogc}\subset \ogn^{I}_\Delta$ and  the closure $\ov{\ogn^{I\!I}_{\ogc}}$ of $\ogn^{I\!I}_{\ogc}$.}
\item{The morphisms $\ogn^I_{\Delta}, \ogn^{I\!I}_\Delta \to \Delta$ have irreducible $5$--dimensional  fibers, namely, for $\ogc \in \Delta$,
$\ogn^I_{\ogc}$ and $\ogn^{I\!I}_{\ogc}$.
 %whereas the morphism  $(\ov{\ogn^{I\!I}_{\Delta}})  \to \Delta$ has irreducible general fiber, and  for every $\ogc \in \Delta$, the fiber  $(\ov{\ogn^{I\!I}_{\Delta}})_{\ogc}$ is contained in the union $\ogn^I_{\ogc}\cup \ov{\ogn^{I\!I}_{\ogc}}$.
 }
\end{enumerate}}
\end{enumerate}

\end{pr}
 \begin{proof} The proof uses results from the forthcoming \S \ref{secnonred}:
by  Proposition \ref{rank 2} the varieties  $\ogmm^{I}_{\ogc}$ and $\ogn^{I}_{\ogc}$ are irreducible of dimension $5$
and $\ogmm^{I}_{\Delta}$ and $\ogn^{I}_{\Delta}$ are irreducible of dimension $7$.
By Proposition \ref{even case} the variety $\ogmm^{I\!I}_{\ogc}=\ogmm_{\ogc}\setminus \ogmm^{I}_{\ogc}$
is irreducible of dimension $5$ and, by Proposition \ref{odd case}, the same holds for 
$\ogn^{I\!I}_{\ogc}=\ogn_{\ogc}\setminus \ogn^{I}_{\ogc}$. By Remark \ref{piattezza}, the varieties $\ogmm_{\Delta}$
and $\ogn_{\Delta}$ are flat over $\Delta$. Since the fibers of
$\ogmm^{I\!I}_{\Delta}=\ogmm_{\Delta}\setminus \ogmm^{I}_{\Delta}$ and 
$\ogn^{I\!I}_{\Delta}=\ogn_{\Delta}\setminus \ogn^{I}_{\Delta}$ are irreducible of dimension $5$,
the closure $\ov{\ogmm^{I\!I}_{\Delta}}$ of $\ogmm^{I\!I}_{\Delta}$ in $\ogmm^{I\!I}_{\Delta}$
and the closure $\ov{\ogn^{I\!I}_{\Delta}}$ of $\ogn^{I\!I}_{\Delta}$ in $\ogn^{I\!I}_{\Delta}$
are irreducible of dimension $7$.
\end{proof}

\begin{rmk} The varieties $\ogmm^{I\!I}_{\Delta}$, 
$\ogmm^{I\!I}_{\ogc}$ (and their respective closures) are actually reduced but we don't need this result.
 %{\cdg $\ogn^{I\!I}_{\Delta}$ and  $\ogn^{I\!I}_{\ogc}$ NON SONO RIDOTTE quindi le ho tolte!!!!!}
\end{rmk}

%%%%%%%%%%%%%%%%%%%%%%%%%
%%%%%%%%%%%%%%%%%%%%%%%%%
%%%%%%%%%%%%%%%%%%%%%%%%%
%%%%%%%%%%%%%%%%%%%%%%%%%
%%%%%%%%%%%%%%%%%%%%%%%%%
%%%%%%%%%%%%%%%%%%%%%%%%%
%%%%%%%%%%%%%%%%%%%%%%%%%
%%%%%%%%%%%%%%%%%%%%%%%%%
%%%%%%%%%%%%%%%%%%%%%%%%%
%%%%%%%%%%%%%%%%%%%%%%%%%

\subsection{The fibers over  non-reduced curves} \label{secnonred}

In this section we study sheaves the non-reduced curves of $|2H|$ with the aim of proving Propositions \ref{rank 2}, \ref{even case}, and \ref{odd case}, which were used in the proof of Proposition \ref{propriass}. First we deal with the simpler case of the sheaves of type I (\S \ref{irr comp I}), then in \S \ref{prelim to type 2}-\ref{irr comp II} we carry out the more involved  study of the sheaves of type II.

%%%%%%%%%%%%%%%%%%%%%%%%%
%%%%%%%%%%%%%%%%%%%%%%%%%
%%%%%%%%%%%%%%%%%%%%%%%%%
%%%%%%%%%%%%%%%%%%%%%%%%%
%%%%%%%%%%%%%%%%%%%%%%%%%
%%%%%%%%%%%%%%%%%%%%%%%%%
%%%%%%%%%%%%%%%%%%%%%%%%%
%%%%%%%%%%%%%%%%%%%%%%%%%
\subsubsection{Irreducible components over  non-reduced curves: sheaves of type I}\la{irr comp I}$\;$

\begin{pr} \label{rank 2}  \label{comparingstab}
{\rm 
({\bf
Irreducibility of  $\ogmm^I_{\ogc}$, $\ogn^I_{\ogc}$, $\ogmm_\Delta^I$ and $\ogn_{\Delta}^I$})
}
Let $\ogc=2\ogcred \in \Delta$.
$\;$
\begin{enumerate}
\item The locus $\ogmm^I_\Delta$ is an irreducible component
of $\ogmm_\Delta$ and the morphism $\ogmm^I_\Delta \to \Delta$ is projective with irreducible $5$--dimensional fibers $\ogmm^I_{\ogc}$.
\item The locus $\ogn^I_\Delta$ is an irreducible component
of $\ogn_\Delta$ and the morphism $\ogn^I_\Delta \to \Delta$ is projective with irreducible $5$--dimensional fibers $\ogn^I_{\ogc}$.
\end{enumerate}
%\item The locus $\ogmm_{\ogc}^{I}$  parametrizing sheaves of type $I$ is irreducible projective of dimension $5$ and is thus an irreducible component of  $\ogmm_{\ogc}$;  similarly, with $\ogn$ replacing $\ogmm$.
%\item The locus $\ogmm^I_\Delta$ is an irreducible component
%of $\ogmm_\Delta$; similarly with $\ogn$ replacing $\ogmm$.
%\item For a smooth curve $\ogcred \in B'$, the assignment   $\m{F}  \mapsto \det \m{F}$ realizes $\ogn_{\ogc}^{I}$
%(respectively, $\ogmm_{\ogc}^{I}$) as a smooth fibration over ${\rm Jac}^{d'} (\ogcred)$, locally trivial in the \'etale topology, with fibers a transversal complete intersections of two $4$-dimensional quadrics (respectively, $\mathbb P^3$).

\end{pr}

\begin{proof} 
  The sheaves of type $I$ are $\mc O_\ogcred$-modules,  so by  Definition \ref{defstab} the loci $\ogmm^I_{\ogc}$  and $\ogn^I_{\ogc}$ parametrizing semistable sheaves of type $I$ on  $\ogc$, with Euler characteristic $\chi$, are naturally identified with the moduli spaces of  torsion free slope semistable sheaves on the reduced underlying curve $\ogcred$, with rank $2$ and degree $d':=\chi - 2\chi(\mc O_\ogcred)$.
The latter space is irreducible and projective by  \cite{Seshadri, re82}. By  \ci[Thm 4.3.7]{HL}, $\ogmm_\Delta^I$ and $\ogn^I_\Delta$ are projective over $\Delta$.
%By  \cite{re82} (see also \cite[Prop 9]{Seshadri} for the case of nodal curves), the moduli space of semistable rank $2$ vector bundles on an irreducible curve with locally planar singularities is irreducible and projective. 
The statement about the dimension of the fibers follows from Remark \ref{piattezza}. This proves the first assertion. The second one follows, since the morphisms $\ogmm^I_\Delta, \ogn^I_\Delta
\to \Delta$ are proper, surjective and with irreducible fibers. 
%Finally, part (3) is \cite[Thm 2 and Thm 4]{Narasimhan-Ramanan}.
%%Since the morphisms $\ogm/\Delta$ and $\ogn/\Delta$ have all irreducible components of all fibers of dimension $5$ (cf. Remark \ref{piattezza})
%%
%%
%%In view of Remark \ref{comparingstab}, for the first statement we only have to observe that the moduli space of semistable rank $2$ vector bundles on an irreducible curve with locally planar singularities is irreducible and projective: this is \cite[Prop 9]{Seshadri} for nodal curves and  \cite{re82} for arbitrary planar singularities curves. Since the morphisms $\ogm/\Delta$ and $\ogn/\Delta$ have all irreducible components of all fibers of dimension $5$ (cf. Remark \ref{piattezza}), the first part of the second assertion  concerning $\ogc$ follows from (1).
%%
%%
%%The remaining part of the second assertion follows from the fact that the morphisms $\ogmm^I_\Delta, \ogn^I_\Delta
%%\to \Delta$ are proper, surjective and with irreducible fibers.
%%The third assertion is \cite[Thm 2 and Thm 4]{Narasimhan-Ramanan}.
\end{proof}

As a side remark, we mention that it is a classical result of \cite{Narasimhan-Ramanan}, that for a smooth curve $\ogcred \in B'$, the assignment   $\m{F}  \mapsto \det \m{F}$ realizes $\ogn_{\ogc}^{I}$ (respectively, $\ogmm_{\ogc}^{I}$) as a smooth fibration over ${\rm Jac}^{d'} (\ogcred)$, locally trivial in the \'etale topology, with fibers a transversal complete intersections of two $4$-dimensional quadrics (respectively, $\mathbb P^3$).

\begin{rmk} \label{strictlyssempty}
The strictly semistable locus on $\ogn$   is empty.
By Theorem  \ref{thmpoly} (1)   the   locus of strictly semistable sheaves in $\ogmm_\Delta$ is contained in $\ogmm_\Delta^I$.  
%If $\mc F$ is strictly semistable  and lies in $\ogmm_{\ogc}$, with  $\ogc=2\ogcred\in\Delta$, then by Remark \ref{rmkpoly} $\mc F$ is $S$-equivalent to $F_ 1 \oplus F_2$ where $F_1$ and $F_2$ are pure rank 1 sheaves on the reduced underlying curve $\ogcred$. It follows that  the  locus of strictly semistable sheaves in $\ogmm_\Delta$ is contained in the locus  $\ogmm_\Delta^I$ of sheaves of type $I$.
\end{rmk}

\subsubsection{Preliminary results towards the analysis of type II}\la{prelim to type 2}$\;$

In this subsection we establish some notation, recall some facts, and prove a series of preliminary results towards  Propositions \ref{even case} and \ref{odd case} (which describes the fibers of $\ogmm_\Delta^{I\!I}, \ogn_\Delta^{I\!I}  \to \Delta$).

For $\ogcred \in \ogbp$, let $\ogc  = 2\ogcred \in \Delta$ be the corresponding double curve.
Denote by $\mc I \subset \mc O_{\ogc}$ the ideal sheaf of the reduced curve $\ogcred \subset \ogc$. Then:
\beq\la{fascio I}
\mbox{
${\mc I}^2=0,\;$
$\mc I$ is an $\mc O_\ogcred$-module,$\;$  $\mc I \cong \mc O_\ogcred(-\ogcred)\cong \omega_\ogcred^{-1}, \;$
$\deg \mc I=-2.$
}
\eeq
%In particular, it is a line bundle on the integral projective curve $\ogcred$, and  we have:\beq \label{deg I}\deg \mc I=-2.\eeq

%\begin{cor}
%$\ogn_{\Delta}^{I}$ and $\ogmm_{\Delta}^{I}$ are irreducible and their complements in $\ogn_{\Delta}$ and $\ogmm_{\Delta}$ parametrize sheaves of type $I\!I$.
%\end{cor}

Since the genus $2$ K3 surface $S$ is general in moduli, any curve $\ogcred \in |H|$ has only nodes and cusps (cf. Remark \ref{nodeandcusps}) and, moreover, it has at most two singular points.  For later use we record the following standard facts on rank $1$ torsion free sheaves on 
integral  curves with nodes or cusps.
Let $\Gamma$ be an integral projective curve with only nodes and cusps and let $E$ be a rank one torsion free sheaf on $\Gamma$. By \cite[Lem 1.4]{D'Souza}, we can write $E=n_* L$, where $n: \wh \Gamma \to \Gamma$ is a partial normalization and where $L$ is a line bundle on $\wh \Gamma$. Both $\wh \Gamma$ and $L$ are uniquely determined by $E$. Now let $\mc L \in Pic(\Gamma)$ be a line bundle. Then $ E \otimes \mc L \cong E$ if and only if $\mc L \in \ke \, [n^*:Pic(\Gamma) \to Pic(\wh \Gamma)]$ (cf. \cite[Lem. 2.1]{Beauville}).

\begin{lm} \label{localbehavior} {\rm {\bf (Nodes and cusps)} }
Let $\Gamma$ be as above.
 \begin{enumerate}

%\item  Let $E$ be a rank one torsion free sheaf on $\Gamma$.  Then $E=n_* L$, where $n: \wh \Gamma \to \Gamma$ is a partial normalization and where $L$ is a line bundle on $\wh \Gamma$. Moreover, such $\wh \Gamma$ and $L$ are uniquely determined by $E$.
%
%\item Let $E$ and $\wh \Gamma$ be as  in (1) and let $\mc L \in Pic(\Gamma)$ be a line bundle. Then $ E \otimes \mc L \cong E$ if and only if $\mc L \in \ke \, \{n^*:Pic(\Gamma) \to Pic(\wh \Gamma)\}$.
%
%\item Let $p \in \Gamma$ be a singular point and let  $\wh \Gamma$ be the normalization of $\Gamma$ at $p$. Then $(n_*\mc O_{\wh \Gamma})^\vee:=\Shhom_{\mc O_{\Gamma}}( n_* \mc O_{\wh \Gamma}, \mc O_{\Gamma})\cong m_p$, where $m_p \subset \mc O_{\Gamma}$ is the ideal of $p$. Moreover, $(m_p)^\vee \cong n_* \mc O_{\wh \Gamma}$.

\item Let $q \in \Gamma$ be any point. Then  $\Ext^1_{\mc O_{\Gamma}}(\C_q, \mc O_{\Gamma}) =\Ext^1_{\mc O_{\Gamma,q}}(\C_q, \mc O_{\Gamma,q})=\C$. The unique (up to isomorphism) extension is:
\[
0 \to \mc O_{\Gamma} \to (m_q)^\vee  \to \C_q \to 0,
\]
obtained by dualizing the standard exact sequence $0 \to m_q \to  \mc O_{\Gamma}  \to \C_q \to 0$.
%(Notice that $(m_q)^\vee  \cong \mc O_{\Gamma}(q)$ if $q$ is smooth and $(m_q)^\vee \cong n_*\mc O_{\wh \Gamma}$ otherwise). 
\item Let $q \in \Gamma$ be a singular point, and let $\wh \Gamma$ be the normalization of $\Gamma$ at $q$. Then 
$\Ext^1_{\mc O_{\Gamma}}(\C_q,  m_q) =\Ext^1_{\mc O_{\Gamma,q}}(\C_q,  m_q)=2$. The  corresponding $1$-dimensional family of  isomorphism classes of extensions is given by considering the natural extension
\[
0 \to m_q  \to \mc O_{ \Gamma} \to \C_q \to 0
\]
and tensoring it with the elements of the one dimensional group $\ke[n^*:Pic(\Gamma) \to Pic(\wh \Gamma)]$ (which fixes $m_q$).
\end{enumerate}
\end{lm}
\begin{proof} %$(1)$ \cite[Lem 1.4]{D'Souza} $(2)$ \cite[Lem. 2.1]{Beauville}. $(3)$  \cite[Lemma 3]{cookdjvu}. 
$(1)$ Since $\Gamma$ is a Gorenstein curve,  have  $\dim \Ext^1_{\mc O_{\Gamma,q}}(\C_q, \mc O_{\Gamma,q})=1$. $(2)$ 
The first statement follows from \cite[Lem. 2.5.5]{Cook-thesis}. 
The second statement follows from the considerations before the Lemma, with $E=m_q$ and $\wh \Gamma$ the normalization of $\Gamma$ at $q$.  
\end{proof}

\begin{lm}(\cite[Lem. 2.5.9]{Cook-thesis}) \label{cook} Let  $\Gamma$ be an integral curve, let $p \in \Gamma$   be a nodal or cuspidal point,  and let $F_p$ and $K_p$ be rank $1$ torsion free $\mc O_{\Gamma,p}$-modules. Then
\[
 \Ext^1_{\Gamma,p}(F_p, K_p) =\left\{ \begin{array}{cc}
0 & \text{if $F_p$, or $ K_p$ is free,} \\
\C^2 & \text{if neither $F_p$, nor $K_p$ is free.}
\end{array}\right.
\]
\end{lm}

\begin{lm} \label{lemmino} 
 Let $ n: \wh \Gamma \to \Gamma$ be a partial normalization of a reduced curve and let $L$ and $L'$ be torsion free sheaves on $\wh \Gamma$.  Then $\Hom_\Gamma( n_* L,  n_* L')=\Hom_{\wh \Gamma}(L, L')$.
 \end{lm}
 \begin{proof}
 It is clear that a morphism $\varphi: L \to L'$ induces by pushforward a morphism $ n_* \varphi:  n_*L \to  n_* L'$ of torsion free sheaves. Conversely, if $\psi:  n_* L \to  n_* L'$ is a morphism, by pullback we get a morphism $   n^*  n_* L\to  n^* n_* L' \to L'$. Since $L'$ is torsion free, this morphism necessarily factors via the surjection $ n^*  n_* L \to L$ which thus induces a morphism $L \to L'$. Since $ n$ is a generic isomorphism, these two assignments are inverses of each other. 
\end{proof}

\begin{lm} \label{reflexive}
Let $E$ be a pure dimension $1$ sheaf on the K3 surface  $\ks$ and let $\Gamma \subseteq \ks $ be the curve    Fitting support of $E$. Then $\Shext^i_S(E, \mc O_S(-\Gamma))=0$ for $i \neq 1$ and there is a functorial isomorphism $\Shext^1_S(E, \mc O_S(-\Gamma)) \cong E^\vee$ where $E^\vee := \Shhom_\Gamma(F, \mc O_\Gamma).$
Moreover, $E$ is reflexive, i.e. $(E^\vee)^\vee \cong E$.
\end{lm}
\begin{proof}
This is \cite[Prop 1.1.10]{HL} (or can easily be seen using Remark \ref{fitting}).
\end{proof}

%After the following lemma and proposition, we shall consider separately the two cases $\ogmm_\Delta$ and $\ogn_\Delta$, which correspond to the case of sheaves of even and odd Euler characteristic.

\subsubsection{Irreducible components over  non-reduced curves: sheaves of type II}\la{irr comp II}$\;$

We now come back  to $\ogmm_{\ogc}, \ogn_{\ogc}$, $\ogmm_\Delta$ and $\ogn_\Delta$ and turn to sheaves of type $I\!I$. We  follow closely \cite[\S 3.1]{Mozgovoy}, which deals with the case of sheaves with even Euler charcteristics (those parametrized by $\ogmm$).

\begin{lm} \label{possiblesurjections}
 Let $\mc F$ be a sheaf of type $I\!I$ supported on a curve $\ogc=2\ogcred \in \Delta$, let $\mc F_{|\ogcred}$ be its restriction to the reduced curve $\ogcred$, and let $F:=\mc F_{|\ogcred} \slash T$ be the quotient of $\mc F_{|\ogcred}$ by its torsion subsheaf $T:=Tors(\mc F_{|\ogcred})$. If $E$ is  a pure dimension $1$ sheaf on $S$ and $\alpha:\mc F \to E$ is a surjection which is not an isomorphism, then $E \cong F$ and $\alpha$ is a non zero multiple of the  natural morphism $\mc F \to F$. As a consequence, $\mc F$ is (semi)stable if an only if
 \[
 \frac{\chi(\mc F)}{2} \underset{(=)}{<}\chi(F).
 \]
 \end{lm}
 \begin{proof}
Let $\alpha: \mc F \to E$ be a surjection onto a pure dimension $1$ sheaf $E$. Then the Fitting support of $E$ is either $\ogc$ or $\ogcred$. In the first case,  $\ke \, \alpha$ is supported on points,  hence it has to be zero by the purity of $\mc F$. In particular, $\alpha$ is an isomorphism. It follows that we can consider only the case when the Fitting support of $E$ is $\ogcred$ and hence $E$ is an $\mc O_\ogcred$-module of rank one. Restricting $\alpha$ to $\ogcred$ we get a surjection $F \to E$, where $F:=\mc F_{|\ogcred} \slash Tors(\mc F_{|\ogcred})$, which has to be an isomorphism since $F$ and $E$ are torsion free sheaves of rank $1$ on $\ogcred$. Using Definition \ref{defstab}, we conclude the proof.
 \end{proof}

\begin{pr} \label{restriction of F to C}  Let $\ogc=2\ogcred \in \Delta \subseteq \ogb$ and let $\mc F $ be a sheaf parametrized by a point in $ \ogn_{\ogc}^{I\!I}$ or in $ \ogmm_{\ogc}^{I\!I}$, i.e., a stable sheaf of type $I\!I$ (cf. Remark \ref{strictlyssempty}). Then $\mc F$ sits in the following exact sequence:
\[
0 \to F \otimes_{\ogcred} \mc I \to \mc F \to \mc F_{|\ogcred} \to 0,
\]
where $F:=\mc F_{|\ogcred} \slash T$ and  $T:=Tors(\mc F_{|\ogcred})$ is  the torsion subsheaf of $\mc F_{|\ogcred}$. Furthermore,   either $T=0$, which occurs if $\chi(\mc F)$  is even (i.e. $\mc F  \in \ogmm_{\ogc}^{I\!I}$), or  $T=\C_p$, for some point $p \in \ogcred$, which occurs if $\chi(\mc F)$  is odd (i.e. $\mc F  \in \ogn_{\ogc}^{I\!I}$). 
\end{pr}

\begin{proof}
In the even case, this is \cite[Lem. 3.1.6]{Mozgovoy}, whose proof we modify below to deal with the
case of arbitrary $\chi$.
Consider the exact sequence:
\[
\mc F \otimes_{\ogc} \mc I \to \mc F \stackrel{r}{\to}  \mc F_{|\ogcred} \to 0.
\]
Since  $\mc I$ is a nilpotent square zero ideal,  $\mc F \otimes_{\mc{O}_{\ogc}} \mc I$ is an $\mc O_\ogcred$-module, and $\mc F \otimes_{{\ogc}} \mc I\cong \mc F_{|\ogcred} \otimes_{\mc O_\ogcred} \mc I$. Since $\mc F$ is not of type $I$, the restriction map $r$ is not an isomorphism and  $\ker(r)$ is a non--zero pure $\mc O_\ogcred$-module.
This applies to the upcoming $K$ as well. We see that the sheaves $\ker(r)$ and of $\mc F_{|\ogcred}$ have first Chern class equal to $\ogcred$, so they are $\mc O_\ogcred$-modules of rank $1$.  Let
$
T:=Tors(\mc F_{|\ogcred})
$
be the torsion subsheaf of  $\mc F_{|\ogcred}$. Then $F:=\mc F_{|\ogcred} \slash T$ is torsion free of rank $1$.  The natural surjective morphism $\mc F_{|\ogcred} \otimes_{\mc O_\ogcred} \mc I \to \ker(r)$ factors via the quotient $\mc F_{|\ogcred} \otimes_{\mc O_\ogcred} \mc I \to F \otimes_{\mc O_\ogcred} \mc I$, determining an isomorphism $F \otimes_{\mc O_\ogcred} \mc I \to \ker(r)$ of rank $1$ torsion free sheaves on $\ogcred$.  We can summarize this discussion in the following 
commutative diagram of short exact sequences (the zeros are omitted from the vertical ones):
\beq \label{diagram}
\xymatrix{
& & &  T \ar@{^{(}->}[d]&\\
0 \ar[r] & F \otimes \mc I \ar[r] \ar@{^{(}->}[d] & \mc F \ar[r] \ar@{=}[d]  & \mc F_{|\ogcred} \ar@{->>}[d]  \ar[r] & 0 \\
0 \ar[r] & K   \ar@{->>}[d]\ar[r] & \mc F \ar[r] & F \ar[r] & 0 \\
 &T & & &\\
}
\eeq
%The fact that $\coke[F \otimes \mc I \to K] \cong T$ follows from the snake lemma.
Using this diagram, together with (\ref{fascio I}), we get:
\beq \label{chimcFechiF}
\chi(\mc F)=2 \chi (F)-2+ \chi(T),
\eeq
showing that $\chi(\mc F) \equiv \chi(T)$ modulo $2$. By using Lemma \ref{possiblesurjections}  to check stability, we see that:
\beq \label{stabilityF}
\frac{\chi(\mc F)}{2} <\chi(F),
\eeq
so  that $ 0 \le \chi(T)  <2.$
\end{proof}

\begin{cor} \label{odd not loc free} Let $\mc F$ be a sheaf with Fitting support $\ogc=2\ogcred$ and odd Euler characteristic. Then $\mc F$ is not (the push forward of) a line bundle on $\ogc$.
\end{cor}
\begin{proof} If $\mc F$ is of type $I$, then it is clear that it cannot be a locally free $\mc O_{\ogc}$--module.
If $\mc F$ is a  sheaf of type $I\!I$ , then by (\ref{chimcFechiF}), $ T \neq 0$ so $\mc F_{|\ogcred}$ is not locally free.
\end{proof}

\begin{pr}[{\bf The case of $\chi$ even: the structure of $\ogmm^{I\!I}_{\ogc}$ and $\ogmm^{I\!I}_\Delta$}] \label{even case}
Let $\ogc=2\ogcred \in \Delta \subseteq \ogb$. 
The stable sheaves of type $I\!I$ on $\ogc$ with even Euler characteristic $\chi=2k$ are locally free.  
Restricting a sheaf to the underlying reduced curve defines a surjective morphism
\[
\ogmm_{\ogc}^{I\!I} \to \rm{Pic}^{k+2}(\ogcred)
\]
which is a Zariski-locally trivial $\mathbb{C}^3$-bundle whose fibers are isomorphic to $H^1(\ogcred, \mc I)$. In particular, $\ogmm_{\ogc}^{I\!I} $ is smooth and irreducible. The locally closed subvariety $M^{I\!I}_\Delta$ of $\ogm$ can be identified with  degree-$(2k+4)$ component $\rm{Pic}^ {2k+4} (\mc \ogc_\Delta \slash \Delta)$ of the relative Picard scheme of the family $\mc \ogc_\Delta \slash \Delta$ of non-reduced curves. In particular, $\ogmm^{I\!I}_{\Delta}$ is smooth and irreducible.
\end{pr}

\begin{proof} The first step is to use  \cite[Lemma 3.3.3]{Mozgovoy}. We include the details to set up the notation for the case  of $N^{I\!I}_{\ogc}$, i.e., when  $\chi$ is odd.
Let $\mc E$ be any sheaf with Fitting support equal to $\ogc$. Suppose $\mc E$ sits in the following exact sequence
of $\m{O}_{\ogc}$-modules:
\[
0 \to E_2 \to \mc E \to E_1 \to 0
\]
where $E_1$ and $E_2$ are torsion free sheaves of rank one on $\ogcred$. The spectral sequence for the change of coefficients in the Ext groups \cite[Cor.3.2.2]{Mozgovoy} gives:
\beq \label{extcoeff}
0 \to \Ext_{\ogcred}^1(E_1, E_2) \to \Ext^1_{\ogc}(E_1, E_2) \stackrel{\delta}{\to} \Hom_\ogcred(E_1 \otimes \mc I , E_2) \to \Ext^2_{\ogcred}(E_1, E_2).
\eeq
Let $\varepsilon \in \Ext^1_{\ogc}(E_1, E_2) $ be the class of this extension. By \cite[Lem 3.2.2]{Mozgovoy}, $E_1=\mc E\otimes_{\mc O_{\ogc}} \mc O_\ogcred$  (equivalently, $E_2=\mc I \mc E$ ) if and only if $\delta(\varepsilon): E_1\otimes \mc I  \to E_2$ is surjective. Since $E_1 \otimes \m{I}$  is torsion free, this is the case if and only if $\delta(\varepsilon)$ is an isomorphism, which means that $ E_2\cong E_1\otimes \mc I $. Moreover, by \cite[Prop 3.2.7]{Mozgovoy}  $\delta$ is surjective if and only if $E_1$ is locally free.

Now we consider the case of  a stable sheave $\mc F$ with Fitting support equal to $\ogc$ and even Euler characteristic. Then by diagram (\ref{diagram}) $\mc F$ is an extension of $F$ by $F \otimes_{\mc O} \mc I$. So we set  $E_1=F$ and $E_2=F \otimes_{\mc O} \mc I$ and the exact sequence  (\ref{extcoeff}) becomes:
\[
0 \to \Ext^1_{\ogcred}(F , F \otimes \mc I) \to \Ext^1_{\ogc}(F , F \otimes \mc I)  \stackrel{\delta}{\to} \Hom_\ogcred(F\otimes \mc I, F \otimes \mc I)= \mathbb C.
\]
We conclude that if $\mc F$ is a sheaf in $M^{I\!I}_{\Delta}$, then $\delta$ is non zero, so it is surjective and hence $F$ is locally free. By using Lemma \ref{loc free naka}, $\mc F$ is locally free. Conversely, if $F$ is locally free, then $\Ext^2_{\ogcred}(F , F \otimes \mc I)=H^2(\ogcred, F^\vee \otimes F \otimes \mc I)=0$, so $\delta$ is surjective and the isomorphism classes of $\mc O_{\ogc}$-modules realized as extensions of $F$ by $F \otimes_{\mc O_\ogcred} \mc I$ are parametrized by the affine space $\mathbb P(\Ext^1_{\ogc}(F, F \otimes \mc I)) \setminus \mathbb P(\Ext^1_{\ogcred}(F, F \otimes \mc I))$. Since $F$ is locally free, $\Ext^1_{\ogcred}(F , F \otimes \mc I)\cong H^1(\ogcred, \mc I) =\mathbb C^3$.  If we set $\chi(\mc F)=2k$, by (\ref{diagram}) we get $\chi(F)=k+1$ and $\deg (F)=k+2$.

 Associating to each $\mc F$ its restriction $F$ to the reduced underlying curve defines the morphism $\ogmm_{\ogc}^{I\!I} \to \rm{Pic}^{k+2}(\ogcred)$. By the discussion above, the fibers are identified with the affine space $\mathbb P(\Ext^1_{\ogc}(F, F \otimes \mc I)) \setminus \mathbb P(\Ext^1_{\ogcred}(F, F \otimes \mc I)) \cong  H^1(\ogcred, \mc I)$. Another way of seeing that the fibers of the restriction morphism are isomorphic to $H^1(\ogcred, \mc I)$ is to consider the short exact sequence $0 \to \mc I \to \mc O_{\ogc}^\times \to \mc O_\ogcred^\times \to 0$ and the corresponding sequence of first cohomology groups.

The last statement follows from the fact, proved above, that $\mc F$ is a locally free $\mc O_C$-module.
\end{proof}

\begin{cor} \label{stabrk2} Let $\mc{\ogcred} \to |H|=\Delta$ be the family of genus two curves.
The kernel of the restriction morphism ${\rm Pic}^0_{\m{\ogc}_\Delta /\Delta}  \to {\rm Pic}^0_{\mc{\ogcred}/\Delta}$ is an affine group scheme of relative dimension $3$, which acts trivially on the points of $M^{I}_{\Delta}$ and $N^I_{\Delta}$.
\end{cor}

\begin{rmk} \label{rmkextensions}
The change of coefficients exact sequence (\ref{extcoeff}) for the {\rm Ext}-groups with respect to the closed embedding $\ogcred\subseteq \ogc$,  remains valid for the closed embedding $\ogc \subset S$, thus giving an inclusion $\Ext^1_{\ogc}(E_1, E_2) \subset \Ext^1_S(E_1, E_2)$. We claim that this is in fact an isomorphism. Indeed, any extension as $\mc O_S$-modules of $E_1$ by $E_2$  has Fitting support equal to $\ogc$ and is automatically an $\mc O_{\ogc}$-module.  Moreover, any such extension is an extension of $\mc O_{\ogc}$-modules. 

\end{rmk}

%\begin{rmk} The complementary $\mathbb P^2$-bundle with fiber $\mathbb P (\Ext^1_{\ogcred}(F, F \otimes  \mc I))$ over a line bundle $F \in Jac(\ogcred)$ parametrizes sheaves of type $I$ and therefore admits a map to $\ogmm_\Delta^{I}$ and therefore, via the determinant morphism, to $Jac(\ogcred)$. This map factors via the multiplication by $2$ map....

We now deal with sheaves $\mc F$ of type $I\!I$ with odd Euler characteristic $\chi$,
i.e. with $ \mc F \in \ogn^{I\!I}_\Delta$. The goal is to prove Proposition \ref{odd case}, the ``$\chi$ odd" analogue of Proposition 
\ref{even case}.
%  see Proposition \ref{odd case}, and then to summarize the analysis carried out in this section in the main Proposition \ref{propriass}.  This will be followed by Proposition\ref{compmtildedelta} which deals with $\ogm_{\ogc}$ and $\ogm_{\Delta}$ {\cdg ATTENZIONE WRONG REFERENCE}.

 Consider a  double curve  $\ogc \in \Delta$. In analogy with the case of even Euler characteristic treated earlier, in order to describe the open subvariety $\ogn_{\ogc}^{I\!I} \subset \ogn_{\ogc}$  we wish to realize every $\mc F \in \ogn_{\ogc}^{I\!I}$ as an extension of $\mc O_{\ogc}$-modules by two rank $1$ torsion free $\mc O_\ogcred$-modules.  By looking at (\ref{diagram}) and (\ref{chimcFechiF}), we see that any $\mc F \in \ogn_{\ogc}^{I\!I}$ fits into a short exact sequence:
\beq \label{KandF}
0 \to K \to \mc F \to F \to 0.
\eeq
with
\beq \label{degFdegK} 
d:=\deg (F)=\frac{\chi+3}{2}, \qquad  \deg(K)=d-1.
\eeq
As a consequence of Proposition \ref{restriction of F to C}, for each $\mc F \in \ogn_{\ogc}^{I\!I}$,
we find a triple:
\beq\la{trripple}
(p, F, K) \in \ogcred \times \bar Jac^d(\ogcred) \times \bar Jac^{d-1}(\ogcred).
\eeq
In the proof of Proposition \ref{odd case} we will describe explicitly the set of triples occurring as above from a stable sheaf $\mc F$ of type $II$ and odd Euler characteristic.

We start with the following two lemmata,  relating the regularity properties of $F$ with those of $K$, and viceversa. These will only be relevant if the underlying reduced curve $C'$ is singular, since  otherwise $F$ and $K$ are clearly locally free.
%We use them to describe the set  triples $(p,F,K)$ (\ref{trripple}).

\begin{lm} \label{Flocallyfree}
Let $\mc F \in \ogn_{\ogc}^{I\!I}$ be a stable sheaf of type $I\!I$ with odd Euler characteristic supported on a curve $\ogc=2\ogcred$, and let $(p, F, K)$ be the triple (\ref{trripple}). If $F_p$ is  free, then $F$ is locally free on $\ogcred$
and $K$ is locally free on $\ogcred \setminus p$.
\end{lm}
\begin{proof}
From the first vertical sequence in (\ref{diagram}) it follows that $K$ and $F$ are locally isomorphic away from $p$, so it is enough to prove the statement involving $F$.
Suppose $F$ is locally free at $p$. Then $\Ext^1_\ogcred(F, T)=0$ so $\mc F_{|\ogcred} \cong F \oplus \C_p$. Using (\ref{diagram}) we get the following commutative diagram, with horizontal and vertical short exact sequences:
\[
\xymatrix{
0 \ar[r] & F \otimes \mc I   \ar@{=}[d] \ar[r] & \mc E \ar@{^{(}->}[d] \ar[r] & F \ar[r] \ar@{^{(}->}[d]& 0 \\
0 \ar[r] & F \otimes \mc I   \ar[r] & \mc F \ar[r] \ar@{->>}[d] & F \oplus \mathbb C_p \ar@{->>}[d] \ar[r] & 0 \\
& & \mathbb C_p \ar@{=}[r]  & \mathbb C_p  & \\
}
\]
where  $\mc E:=\ker[\mc F \to \mathbb C_p]$. Notice that $\mc E$, which is a pure sheaf of type $I\!I$, has even Euler characteristic. By Lemma \ref{possiblesurjections} it follows that  $F \cong \mc E_{|\ogcred} \slash Tors (\mc E_{|\ogcred})$. Using (\ref{stabilityF}) we deduce that $\chi(\mc E)=\chi(\mc F)-1<  2\chi(F)$ and hence, by Lemma \ref{possiblesurjections}, $\mc E$ is stable. We may thus apply the arguments of the proof of Proposition \ref{even case} to $\mc E$ and deduce that $F$ is locally free.
\end{proof}

By dualizing, we can reverse the role of $F$ and $K$.

\begin{lm} \label{Klocfree} Let $\mc F$ and $(p, F,K)$ be as in Lemma \ref{Flocallyfree}. 
If $K_p$ is free, then $K$ is locally free  on $\ogcred$, and  $F$ is locally free on $\ogcred \setminus p$.
\end{lm}
\begin{proof}
%By Lemma \ref{localbehavior}, 
The sheaf $K$ is locally free at a point if and only if its dual $K^\vee:=\Shhom_\ogcred(K, \mc O_\ogcred)$ is locally free at that point. It is therefore enough to prove the statement for $K^\vee$.
Consider the short exact sequence (\ref{KandF}). It can be viewed both as sequence of pure $\mc O_{\ogc}$ and of $\mc O_S$-modules (cf. Remark \ref{rmkextensions}). Applying $\Shhom_S( -, \mc O_{S}(-\ogcred))$ and using Lemma \ref{reflexive}  we get:
\[
0 \to \Shext^1_S(F, \mc O_{S}(-\ogcred)) \to \Shext^1_S(\mc F, \mc O_{S}(-\ogcred)) \to \Shext^1_S(K, \mc O_{S}(-\ogcred)) \to 0.
\]
where:
\beq \label{FdualKdual}
\Shext^1_S(F, \mc O_{S}(-\ogcred))=\Shhom_\ogcred(F, \mc O_\ogcred)=:F^\vee,\,\, \Shext^1_S(K, \mc O_{S}(-\ogcred)=\Shhom_\ogcred(K, \mc O_\ogcred):=K^\vee,
\eeq
and:
\beq \label{Fprimo}
\mc F':=\Shext^1_S(\mc F, \mc O_{S}(-\ogcred)) \cong \Shhom_{\ogc}(\mc F, \mc O_{\ogc}) \otimes_{\mc O_{S}}  \mc O_{S}(-\ogcred).
\eeq
Notice that $\mc F'$ is a pure sheaf of type $I\!I$, with odd Euler characteristic. By Lemma  \ref{possiblesurjections}, up to multiplication by a non-zero scalar, the morphism $\mc F' \to K^\vee$ is precisely the morphism $\mc F' \to \mc F'_{|\ogcred} \slash Tors(\mc F'_{|\ogcred})$. Using (\ref{degFdegK}) we find that:
\[
 \chi(K^\vee)=-d, \quad \chi(F ^\vee)=-d-1, \quad
\chi(\mc F')=\chi(F ^\vee)+\chi(K^\vee)=-2d-1,
\]
so that we have:
\[
\frac{\chi(\mc F')}{2} <\chi(K^\vee),
\]
and by Lemma \ref{possiblesurjections},  $\mc F'$ is stable. By Lemma \ref{tripleandduality} below, the triple associated with $\mc F'$, is precisely
\[
(p, K^\vee, F^\vee).
\]
By applying Lemma \ref{Flocallyfree} to this triple, we see that if $K^\vee$ is locally free at $p$, then it is locally free everywhere.
\end{proof}

We remark that in the proof of the lemma above we showed that the sheaf $\mc F'$ defined in (\ref{Fprimo}) is a stable sheaf of type $I\! I$ on $\ogc$ with odd Euler characteristics. The following lemma describes how the triple (\ref{trripple}) changes when passing from $\mc F$ to $\mc F'$.

\begin{lm} \label{tripleandduality}
Let $\mc F$ be as in Lemma \ref{Flocallyfree}.  Let $\mc F'=\Shext^1_S(\mc F, \mc O_{S}(-\ogcred))$ be as
in (\ref{Fprimo}). Then the support of
 $T':=Tors(\mc F'_{|\ogcred})$  is the same as the support  $\{p\}$ of $T=Tors(\mc F_{|\ogcred})$. 
In particular,  if $(p, F, K) $ is the triple (\ref{trripple})  associated with  $F$, then $(p, K^\vee, F^\vee)$ is the triple (\ref{trripple})  associated with a sheaf  $\mc F'$  in $N^{I\!I}_{\ogc}$ with $\chi (\mc F')= -2d(F)-1$.
\end{lm}
\begin{proof} 
We have already observed in the course of proving Lemma \ref{Klocfree} that $K^\vee= \mc F'_{|\ogcred} \slash T'$ and that $F^\vee =\ker[\mc F' \to K^\vee]$. Moreover, since $\mc F'$ is stable of odd Euler characteristic, by Proposition \ref{restriction of F to C}, $T'$ is a skyscraper sheaf supported at one point of $\ogcred$. We need to show that $Supp (T') =\{p\}$.  To this aim, consider the short exact sequences $0 \to F \otimes \mc I \stackrel{a}{\to}  K \to T \to 0 $ and $ 0 \to K^\vee \otimes \mc I \stackrel{a'}{\to}  F^\vee \to T' \to 0$ obtained by considering the left most column of (\ref{diagram}) for the sheaves $\mc F$ and $\mc F'$. We will show that $T $ and $ T'$ are supported on the same point by showing that the morphisms $a: F \otimes \mc I \to K $ and  $a'\otimes id_{\omega_\ogcred}: K^\vee \to  F^\vee \otimes \omega_\ogcred$ are dual to each other.
Consider the short exact sequence:
\beq \label{j}
0 \to \mc O_S(-\ogcred) \stackrel{j}{\to} \mc O_S \to \mc O_\ogcred \to 0,
\eeq 
and its dual:
\beq \label{jdual}
0 \to \mc O_S \stackrel{j^\vee}{\to} \mc O_S(\ogcred) \to \mc O_\ogcred(\ogcred)\cong \omega_\ogcred \to 0.
\eeq
Tensoring (\ref{j}) by $\mc F$ we  a morphism $\mc F \otimes_{\mc O_S} \mc O_S(-\ogcred) \to \mc F$ which by (\ref{diagram}) factors as follows:
\beq \label{dadualizzare}
\xymatrix{
\mc F \otimes_{\mc O_S} \mc O_S(-\ogcred) \ar@{->>}[d]^-{c} \ar[rr]^-{id_{\mc F} \otimes j}  & & \mc F\\
 F \otimes {\mc I} \ar@{^{(}->}[rr]^-{a} & & K \ar@{^{(}->}[u]^-{b} }
\eeq
Similarly, we can tensor  (\ref{j}) by $\mc F'$ and get a morphism   $\mc F' \otimes_{\mc O_S} \mc O_S(-\ogcred) \to \mc F'$ which factors as:
\beq \label{dadualizzare2}
\xymatrix{
\mc F' \otimes_{\mc O_S} \mc O_S(-\ogcred) \ar@{->>}[d]^-{c'} \ar[rr]^-{id_{\mc F'} \otimes j}  & & \mc F'\\
 K^\vee \otimes {\mc I} \ar@{^{(}->}[rr]^-{a'} & & F^\vee \ar@{^{(}->}[u]^-{b'} }
\eeq
We can also dualize  (\ref{dadualizzare}) by first applying $\Shext^1( -, \mc O_S(-\ogcred))$ and then tensoring the resulting diagram by the line bundle $\mc O_{S}(-\ogcred)$. Using (\ref{FdualKdual}) and (\ref{Fprimo}) we find:
\beq \label{dadualizzare3}
\xymatrix{
\mc F'\otimes_{\mc O_S} \mc O_S(-\ogcred)  \ar@{->>}[d]_{b^\vee \otimes id_{ \mc O_S(-\ogcred)}} \ar[rr]^-{(id_{\mc F} \otimes j)^\vee \otimes id_{ \mc O_S(-\ogcred) }}  & & \mc F' \\
 K^\vee \otimes \mc I \ar@{^{(}->}[rr]_{a^\vee\otimes id_{ \mc O_S(-\ogcred)} }& & F^\vee \ar@{^{(}->}[u]_{c^\vee \otimes id_{ \mc O_S(-\ogcred)}}  }
\eeq
Using (\ref{j}) - (\ref{jdual}) we see that $(id_{\mc F} \otimes j)^\vee \otimes id_{ \mc O_S(-\ogcred) }=id_{\mc F'} \otimes j^\vee  \otimes id_{ \mc O_S(-\ogcred)}=id_{\mc F'} \otimes j$. It follows that diagrams (\ref{dadualizzare2}) - (\ref{dadualizzare3}) are two ways of factoring the same morphism into the product of three morphisms. By the first part of Lemma \ref{possiblesurjections}, the surjective morphism $b^\vee \otimes id_{ \mc O_S(-\ogcred)}: \mc F'\otimes_{\mc O_S} \mc O_S(-\ogcred)  \to  K^\vee \otimes \mc I$ coincides (up to a scalar) with the morphism $c'$ of diagram (\ref{dadualizzare2}).
Since this morphism is surjective, and the two morphism $b'$ and $c^\vee \otimes id_{ \mc O_S(-\ogcred)}$ are injective, we may conclude that $a^\vee\otimes id_{ \mc O_S(-\ogcred)}=a'$ and we are done.
\end{proof}

We are now in the position to prove the ``$\chi$ odd" analogue of Proposition \ref{even case}.

\begin{pr} [{\bf The case of $\chi$ odd: the structure of $\ogn^{I\!I}_{\ogc}$ and $\ogn^{I\!I}_{\Delta}$}] \label{odd case} Let $\ogc=2\ogcred \in \Delta \subseteq \ogb$. The fiber $\ogn^{I\!I}_{\ogc}$ is irreducible. More precisely, the reduced underlying variety $(\ogn_{\ogc}^{I\!I})_{red}$ is an affine bundle of rank $2$ over a $3$-dimensional irreducible locally closed subvariety $Z \subset \ogcred \times \bar Jac^d(\ogcred) \times \bar Jac^{d-1}(\ogcred)$ parametrizing triples $(p, F, K) $ that arise from $\mc F \in \ogn_{\ogc}^{I\!I}$ as in (\ref{trripple}). The fiber of this bundle  over a point $(p, F, K)$ is given by:
\beq \label{c3bundle}
\mathbb P\Ext^1_{\ogc}(F, K) \setminus \mathbb P \Ext^1_\ogcred(F, K),
\eeq
which is isomorphic to $\mathbb C^2$. Here, the inclusion $\Ext^1_\ogcred(F, K) \subset \Ext^1_{\ogc}(F, K)$ inducing (\ref{c3bundle}) is given as in (\ref{extcoeff}) by the change of coefficients exact sequence.
\end{pr}

\begin{proof} By (\ref{diagram}), every $\mc F \in \ogn^{I\!I}_{\ogc}$ is obtained as an extension of $F$ by $K$, so $\mc F$ determines an element in $\Ext^1_{\ogc}(F, K)$.  Conversely,  the extensions in $\Ext^1_{\ogc}(F, K) \setminus \Ext^1_{\ogcred}(F, K)$ (cf. (\ref{extcoeff})) correspond to stable sheaves of type $I\!I$ on $\ogc$ with odd Euler characteristics. We now show that $\ogn^{I\!I}_{\ogc}$, with its reduced induced structure, is a fibration in affine spaces $\mathbb P\Ext^1_{\ogc}(F, K) \setminus \mathbb P \Ext^1_\ogcred(F, K)$ over an irreducible variety $Z$ parametrizing the set of possible $F$ and $K$.

To see this, we start by computing $\Ext^1_{\ogc}(F, K)$ and $  \Ext^1_\ogcred(F, K)$.
By Remark \ref{rmkextensions},  $\Ext^1_{\ogc}(F, K) = \Ext^1_S(F, K)$. The computations of both of these ext groups depend on the singularities of $\ogcred$, as well as on the local behavior of $F$ and $K$ at  $p \in \ogcred$, or at any other singular point of $\ogcred$.  In Lemmata \ref{Flocallyfree} and \ref{Klocfree}, we  ruled out some possibilities for the local behavior of the triples arising from a sheaf $\mc F \in \ogn^{I\!I}_{\ogc}$. In the proof of this proposition we will rule out one further case.
We  say that $F$ (or $K$) is singular at a point $p \in \ogcred$ if $F$ (respectively, $K$) is not locally free at $p$. We denote by $Sing(F) \subset \ogcred$ the locus of points where $F$ is singular, and similarly for $K$.
Throughout, we will use the following fundamental short exact sequence from (\ref{diagram})
\beq \label{fundes}
0 \to F \otimes \mc I \to K \to \mathbb C_p \to 0,
\eeq
relating $F$, $K$, and $p$, where $\mc I \cong \omega_\ogcred^\vee$. Notice that (\ref{fundes}) implies that $F$ and $K$ are locally isomorphic away from $p$.
Using Lemma \ref{localbehavior} we have the following relations between $F$ and $K$ in terms of their local behavior at $p$. Clearly, if $F$ and $K$ are both locally free, then $p$ has to be a smooth point of $\ogcred$ and we find: 
\beq \label{FKlf}
K \cong F \otimes \mc I \otimes \mc O_\ogcred(p).
\eeq
If $F$ is locally free, but $K_p$ is not free, then $K$ is uniquely determined:
\beq \label{Flf}
K \cong F \otimes \mc I \otimes m_p^\vee.
\eeq
If $K$ is locally free, but $F_p$ is not free we find:
\beq \label{Klf}
F \cong K \otimes \mc I^\vee \otimes m_p.
\eeq
If this is the case, then $K$ is not uniquely determined by $F$ and $p$. Indeed, by $(2)$ Lemma \ref{localbehavior},  there is whole $\mathbb P^1=\mathbb P \Ext^1(\mathbb C_p, F \otimes \mc I ) $ of possible (not mutually isomorphic) $K$ fitting in a  short exact sequence as in (\ref{fundes}). The locally free $K$ correspond to open subsets $\mathbb C^* \subset \mathbb P^1$, if $p$ is a node, or $\mathbb C \subset \mathbb P^1$, if $p$ is a cusp.

Finally, if neither $F_p$ nor $K_p$ are free, then we can write
\beq \label{neitherlf}
K=n_* G, \,\, F=n_*L, \,\, \text{ with } G=L \otimes n^* \mc I \otimes \mc O_{\wh \ogcred}(\wh p),
\eeq
and where $n: \wh \ogcred \to \ogcred$ is the normalization of $\ogcred$ at $Sing (F)=Sing(K)$, where $L$ and $G$ are locally free on $\wh \ogcred$, and  $\wh p \in n^{-1}(p) \subset \wh \ogcred$. In the course of the proof we will show that this last case does not occur if $F$, $K$, and $p$ arise from a stable sheaf of type $I\!I$.

\begin{lm} \label{ExtS} Let $F$, $K$ and $p$ be as in (\ref{trripple}). Then
\[
\Ext^1_S(F,K)=\left\{ \begin{array}{cc}
\C^3 & \text{ if $F$ or $K$ are locally free on $\ogcred \setminus p$} \\
\C^4 & \text{ otherwise }
\end{array}\right.
\]
\end{lm}
\begin{proof}
The Euler pairing (cf. \cite[6.1.5]{HL}) and Serre duality on $\ks$  yield:
\beq \label{eulerpairing}
\begin{aligned}
\chi(F,K):&=\dim \Hom_S(F,K)-\dim\Ext^1_S(F,K)+\dim\Hom_S(K,F)\\
& =-v(F) \cdot v(K)=-\ogcred^2=-2.
\end{aligned}
\eeq
By (\ref{degFdegK}), $\deg F=\deg K+1$, so $\Hom_S(F,K)=\Hom_\ogcred(F,K)=0$.
%It follows that in order to compute the dimension of $\Ext^1_S(F,K)$, we only need to compute $\Hom_S(K,F)=\Hom_\ogcred(K,F)$.
To prove the Lemma it is thus enough to show that $\dim \Hom_\ogcred(K,F)=1$ if $F$ or $K$ are locally free away from $p$ and $\dim \Hom_\ogcred(K,F)=2$ otherwise. This is  done using the expressions of equations (\ref{FKlf}), (\ref{Flf}), (\ref{Klf}), and (\ref{neitherlf}).  For example, if $F$ is locally free,  then by (\ref{FKlf}) and (\ref{Flf}), we have
\[
\Hom_\ogcred(K,F)=\Hom_\ogcred(F \otimes \mc I \otimes m_p^\vee, F)=H^0(\ogcred,   \mc I^\vee \otimes m_p)=\C.
\]
A similar computation (using (\ref{Klf})) holds if $K$ is locally free, but $F_p$ is not free.
If neither $F_p$ nor $K_p$ are free, then using $(\ref{neitherlf})$ and Lemma \ref{lemmino} we have
\[
\Hom_\ogcred(K,F)=\Hom_{\wh \ogcred}(L \otimes n^* \mc I \otimes \mc O_{\wh \ogcred}(\wh p), L)=H^0(\wh \ogcred,   n^* \mc I^\vee \otimes \mc O_{\wh \ogcred}(-\wh p)).
\]
It follows that if $\wh \ogcred$ is a curve of arithmetic genus $1$ (i.e., $F$ and $K$ are locally free on $\ogcred \setminus p$), then $H^0(\wh \ogcred,   n^* \mc I^\vee \otimes \mc O_{\wh \ogcred}(-\wh p))=1$, while if $\wh \ogcred$ is a smooth rational curve, then $H^0(\wh \ogcred,   n^* \mc I^\vee \otimes \mc O_{\wh \ogcred}(-\wh p))=2$. 
\end{proof}

\begin{lm} \label{ExtC} Let $F$, $K$ and $p$ as in (\ref{trripple}). Then
\[
\Ext^1_\ogcred(F,K)=\left\{ \begin{array}{cc}
\C^2 & \text{ if $F$ or  $K$ is locally free} \\
\C^3  & \text{ if $Sing(F) = Sing (K)=\{p\}$} \\
\C^4 & \text{ otherwise }
\end{array}\right.
\]
\end{lm}
\begin{proof}
We use the local to global spectral sequence, which yields the long exact sequence
\beq \label{localtoglobal}
0 \to H^1(\ogcred, \Shhom_\ogcred(F,K)) \to \Ext^1_\ogcred(F,K) \to H^0(\ogcred, \Shext^1_\ogcred(F,K)) \to H^2(\ogcred, \Shhom_\ogcred(F,K))=0.
\eeq
We thus  have to compute $H^1(\ogcred, \Shhom(F, K))$ and $H^0(\ogcred, \Shext^1_\ogcred(F, K))$. Using (\ref{FKlf}), (\ref{Flf}), (\ref{Klf}), and (\ref{neitherlf}) we see that
\[
\Shhom(F, K)=\left\{ \begin{array}{cc}
 \mc I \otimes m_p^\vee & \text{ if $F_p$ or $K_p$ is free}\\
 n_*(n^* \mc I \otimes \mc O_{\wh \ogcred}(\wh p)) & \text{ otherwise.}
 \end{array}\right.
\]
This immediately implies that $H^1(\ogcred, \Shhom(F, K))=H^1(\ogcred, \mc I \otimes m_p^\vee)=\C^2$ if $F_p$ or $K_p$ is free; that $H^1(\ogcred, \Shhom(F, K))=H^1(\wh \ogcred, n^* \mc I \otimes \mc O_{\wh \ogcred}(\wh p))=\C$ if $Sing(F)=Sing(K)=\{p\}$ so $\wh \ogcred$ is of arithmetic genus $1$; and finally that  $H^1(\ogcred, \Shhom(F, K))=H^1(\wh \ogcred, n^* \mc I \otimes \mc O_{\wh \ogcred}(\wh p))=0$ if $\# Sing(F)=\# Sing(K)=2$ so that $\wh \ogcred$ is a smooth rational curve. To compute $H^0(\ogcred, \Shext^1_\ogcred(F, K))$ we use Lemma \ref{cook} to find:
\beq \label{exisuipunti}
\Shext^1_\ogcred(F,K)=\bigoplus_{x \in Sing(F)} \Ext^1_{\mc O_{\ogcred,x}}(F_x,K_x)\simeq \bigoplus_{x \in Sing(F)} \C^2. 
\eeq
Putting everything together yields the proof of the Lemma.
\end{proof}

As a consequence of these two lemmas we see that the inclusion $\Ext^1_\ogcred(F,K) \subset \Ext^1_S(F,K)$
of Remark \ref{rmkextensions} is proper, i.e., there is a sheaf of type $I\!I$ on $\ogcred$ that is an extension of $K$ by  $F$ if and only if $F$ and $K$ are locally free $\ogcred \setminus p$ and at most one is not locally free at $p$. When this is the case, then
\beq  \label{dimext}
\dim \Ext^1_\ogcred(F,K)=2, \quad \dim \Ext^1_S(F,K)=3.
\eeq
Now let
\[
Z \subset  \ogcred \times \bar Jac^d(\ogcred) \times \bar Jac^{d-1}(\ogcred)
\] 
be the set of such triples. The discussion above shows that $Z$ is contained in the open set $U\subset  \ogcred \times \bar Jac^d(\ogcred) \times \bar Jac^{d-1}(\ogcred)$ consisting of points such that $F$ and $K$ are locally free $\ogcred \setminus p$ and at most one is not locally free at $p$. More precisely, we claim that $Z$ is the intersection of $U$ with the graph $\Gamma_\psi$ of the rational map
\[
\begin{aligned}
\ogcred \times \bar Jac^d(\ogcred) &\stackrel{\psi}{ \dashrightarrow}  \bar Jac^{d-1}(\ogcred) \\
(p, F) & \longmapsto  F \otimes \mc I \otimes m_p^\vee.
\end{aligned}
\]
Indeed, the fiber of the projection $\Gamma \to \ogcred \times \bar Jac^d(\ogcred)$ consists of: the sheaf  $F \otimes \mc I \otimes m_p^\vee$  over the locus where $F$ is locally free at $p$; the isomorphism classes of extensions $\mathbb P^1=\mathbb P \Ext^1_\ogcred(\mathbb C_p, F\otimes \mc I)$ over the points where $p \in Sing (F)$ (cf. Lemma \ref{localbehavior}). It follows from (\ref{FKlf}), (\ref{Flf}), (\ref{Klf}), and from the remarks shortly thereafter, that $Z=\Gamma \cap U$. In particular, $Z$ is an irreducible variety of dimension $3$.

We are left with showing that the reduced fiber $(\ogn_{\ogc}^{I\!I})_{red}$ is a $2$--dimensional affine bundle over $Z$, with fiber $ \mathbb P\Ext^1_S(F,K) \setminus \mathbb P \Ext^1_\ogcred(F,K)$  over a point $(p, F, K)$. To see this, first note that if $\mc F_W$ is flat family of sheaves in $\ogn_{\ogc}^{I\!I}$, parametrized by a reduced scheme $W$, then we can preform the construction of Proposition \ref{restriction of F to C} in families to get $W$-flat families $F_W$ and $K_W$, of rank $1$ torsion free sheaves of degree $d$ and $d-1$, respectively, as well as a flat family of degree one skyscraper sheaves $T_W$. By the universal property of compactified Jacobians, these families induce a regular morphism $W \to \ogcred \times \bar Jac^d(\ogcred) \times \bar Jac^{d-1}(\ogcred)$ whose image is contained in $Z$. Now choose as $W$ the preimage of $\ogn_{\ogc}^{I\!I}$ in the appropriate Quot scheme, taken with its reduced induced structure. Then $W$ is such that its quotient by the group action on the Quot scheme is exactly $(\ogn_{\ogc}^{I\!I})_{red}$ (cf. \cite[(3) pg 29]{Mumford-GIT}). The equivariant morphism $W \to Z $ factors via the quotient morphism $W \to (\ogn_{\ogc}^{I\!I})_{red}$ and induces a surjective morphism $(\ogn_{\ogc}^{I\!I})_{red} \to Z$. By (\ref{dimext}), the fibers of this morphism are affine $2$--dimensional spaces.

This ends the proof of Proposition \ref{odd case}
\end{proof}

\begin{rmk} \label{NDelta non reduced} By \cite[Lem. 4.12]{Chen-Kass}, the varieties 
$\ogn_{\ogc}^{I\!I}$ and $\ogn^{I\!I}_{\Delta}$ are non-reduced.
\end{rmk}

\subsection{The top degree direct image sheaves $R^{10}$ and the local systems $\ms{L}$}\la{tsr1}$\;$

The purpose of this section is to prove Proposition \ref{rtop} which describes the
restriction of the sheaves $R^{10} \pogm_* \rat $ and $R^{10}\pogn_* \rat$ to the locally closed subvarieties  
$\ogb\setminus \Sigma$,
$\Sigma\setminus\Delta$ and $\Delta$ of $\ogb$.
The proof of the  proposition uses:
 the analysis
of the irreducible components of the restrictions of 
$\ogm, \ogmm$ and $\ogn$  to the loci $\ogb \setminus \Sigma, \Sigma \setminus \Delta$ and $\Delta$
and of the corresponding fibers (Proposition 
\ref{propriass});
some basic properties of the trace morphism and of the direct image sheaf in top degree, as summarized in Fact \ref{fatti da usare};  the general topological Lemma \ref{rtoplemma}.

\begin{fact}\la{fatti da usare}
{\rm 
({\bf Trace morphism and direct image in top degree)
}}
In this section, we need the following three sets of facts:
\ben
\item
The basic properties of  the trace morphism ${\rm Tr}_f: R^{2d} f_! \rat_X (d) \to \rat_T$, 
for a flat morphism  $f: X\to T$ of relative dimension $d$. See \ci{sga43}, especially  Exp. XVIII, Th\'eor\`eme 2.9
(functoriality and compatibility properties),
and Remarque 2.10.1 (the trace morphism is an isomorphism if and only if all the fibers of $f$ have a unique  irreducible component of dimension $d$).

\item
If a flat morphism $f: X\to T$ of relative dimension  $d$  has reduced fibers, then the sheaf
$R^{2d} f_! \rat_X (d)$ is  the $\rat_T$-linearization of the sheaf of sets of irreducible components
of the fibers of $f$; see \ci[Proposition 7.3.2]{ngofl 2008}.

\item 
Let $f: X\to T$ be a projective morphism of pure relative dimension $d$ and with irreducible fibers.
Let $L^d: \rat_X \to \rat_X[2d]$ be the morphism induced by the $d$-th power of the
first Chern class of an $f$-ample line bundle  $L$ on $X$.
In view of the irreducibility of the fibers, by  pushing forward, we get an isomorphism $L^d: \rat_T \stackrel{\sim}\to  R^{2d}f_* \rat_X$.

Moreover,  if $U\subset X$ is a Zariski  open subset intersecting every fiber of $f$ and $f_{U}:U\rightarrow T$ is the restriction of $f$ then $R^{2d}f_{U!} \rat_X \simeq R^{2d}f_* \rat_X\simeq \rat_{T}$. 
In fact, by denoting by $f_{X\setminus U}: X\setminus U\rightarrow T$ the restriction of $f$, the  statement follows from the exact sequence:
$$
%\la{trace}
\xymatrix{
 R^{2d-1} f_{X\setminus U*} \rat_{X\setminus U} =0  \ar[r] &
R^{2d} f_{U!} \rat_{U} \ar[r]^-{R^{2d} j_!}_-\simeq&  
%\ar[dr]^-{\rm Tr}_-\simeq &  
R^{2d} f_{*} \rat_{X} \ar[r]&   
%\ar[d]_-{\rm Tr} &
0= R^{2d} f_{X\setminus U*} \rat_{X\setminus U},
%\ar[r]  & 0,
%\\
%&& \rat_\Delta, & \rat_\Delta  \ar[u]_-\simeq^-{L^d}&
}
$$
where the vanishing statements hold since the fibers of $f_{X\setminus U}$ have at most dimension $d-1$.
\een
\end{fact}

\begin{lm}\la{rtoplemma}
Let $X$ be a variety of pure dimension $d+n$, let $T$ be a normal irreducible variety of dimension $n$ and let 
$f:X\rightarrow T$ be a projective surjective morphism. 

Let $\emptyset=X_{0}\subset X_{1}\subset X_{2}\subset\cdots\subset X_{m-1}\subset X_{m}=X$ be a filtration of $X$ by closed subvarieties such that, for $0<i\le m$:
\begin{enumerate} 
%\item{$X_{i}\setminus X_{i-1}$ is irreducible,} 
\item{Every fiber of the restriction 
$f^{0}_{i}:X_{i}\setminus X_{i-1}\rightarrow T$ of $f$ has  dimension $d$,} 
\item{$R^{2d} f^{0}_{i!}\rat_{X_{i}\setminus X_{i-1}}\simeq \rat_{T}$.}  
\end{enumerate}
 
Then the direct image sheaf in top degree $R^{2d}f_{*}\rat_{X}$ is the trivial  local system of rank $m.$
\end{lm}

\begin{proof} Without loss of generality, we may assume that the $X_i$   are reduced.
 Let $f_{i}:X_{i}\rightarrow T$ be the restriction of $f$. 
 
 CLAIM: the sheaves $R^{2d} f_{i!}\rat_{X_{i}}$ are  local systems of rank $i$.

We now prove the CLAIM by induction on $i$.
Let $\iota: X_{i-1} \to X_{i} \leftarrow X_{i}\setminus X_{i-1}:j$ be the closed and open complementary  embeddings. 
By applying  $R f_{i!}$ to the short exact sequence $0\to j_!  \rat_{X_{i}\setminus X_{i-1}} \to \rat_{X_{i}} \to i_*  \rat_{X_{i-1}}\to 0$,  
we get the  following long exact sequence:
\beq\la{ggoo}
%\la{trace}
\xymatrix{
R^{2d-1} f_{i-1*} \rat_{X_{i-1}} \ar[r] &
R^{2d} f^{0}_{i!} \rat_{X_{i}\setminus X_{i-1}} \ar[r]^-{R^{2d} j_!}&  
%\ar[dr]^-{\rm Tr}_-\simeq &  
R^{2d} f_{i*} \rat_{X_{i}} \ar[r]&   
%\ar[d]_-{\rm Tr} &
R^{2d} f_{i-1*} \rat_{X_{i-1}} \ar[r]  &
0.
%\\
%&& \rat_\Delta, & \rat_\Delta  \ar[u]_-\simeq^-{L^d}&
}
\eeq
By (2), we have that  $R^{2d} f^{0}_{i!} \rat_{X_{i}\setminus X_{i-1}}\simeq \rat_{T}$. 
By the induction hypothesis,  we have that 
$R^{2d} f_{i-1*} \rat_{X_{i-1}}\simeq \rat_{T}^{i-1}$. It follows that, 
if $R^{2d} j_!$ is an injective morphism of constructible sheaves, then   $R^{2d} f_{i*} \rat_{X_{i}}$
is a rank $i$ local system.  In this case, the CLAIM would follow.

\iffalse
{\cm Since $R^{2d} f_{i-1*} \rat_{X_{i-1}}$ is a local system it suffices to check the desired  injectivity
on a non-empty  open subset of the irreducible base $T$. 
Shrinking $T$ if necessary, we may assume that $X_{i-1}$  and $X_i$ are flat over $T$: 
in this case it makes sense to consider the trace morphisms
${\rm Tr}_{f^{0}_{i}}: R^{2d} f^{0}_{i!} \rat_{X_{i}\setminus X_{i-1}}\rightarrow \rat_{T}$
and ${\rm Tr}_{f_{i}}: R^{2d} f_{i*} \rat_{X_{i}\setminus X_{i-1}}\rightarrow \rat_{T}$.
By  \ci{sga43},  Exp. XVIII, Th\'eor\`eme 2.9
(basic functoriality and compatibility properties), we have 
${\rm Tr}_{f_{i}}={\rm Tr}_{f_{i}}\circ R^{2d} j_!$. The assumption (2)
$R^{2d} f^{0}_{i!} \rat_{X_{i}\setminus X_{i-1}}\simeq \rat_{T}$ implies,  by  base change,  that every fiber of $f^{0}_{i!}$ has a unique irreducible component of top dimension $d$. By \ci[Exp. XVIII, Remarque 2.10.1 ]{sga43},  the trace morphism ${\rm Tr}_{f_{i}}$  is an isomorphism. We conclude, in particular,  that the morphism $R^{2d} j_!$ must be injective.  The desired conclusion follows.}
\fi

Since %$R^{2d} f_{i-1*} \rat_{X_{i-1}}$ {\cre Attenzione qua dovremmo mettere 
$R^{2d} f^{0}_{i!} \rat_{X_{i}\setminus X_{i-1}}$
%???} 
is a local system it suffices to check the desired  injectivity
on a non-empty  open subset of the irreducible base $T$. 
Shrinking $T$ if necessary, we may assume that $X_{i-1}$  and $X_i$ are flat over $T$: 
in this case it makes sense to consider the trace morphisms
${\rm Tr}_{f^{0}_{i}}: R^{2d} f^{0}_{i!} \rat_{X_{i}\setminus X_{i-1}}\rightarrow \rat_{T}$
and ${\rm Tr}_{f_{i}}: R^{2d} f_{i*} \rat_{X_{i}}\rightarrow \rat_{T}$.
By  \ci{sga43},  Exp. XVIII, Th\'eor\`eme 2.9
(basic functoriality and compatibility properties), we have 
${\rm Tr}_{f_{i}^0}={\rm Tr}_{f_{i}}\circ R^{2d} j_!$. Moreover 
$R^{2d} f^{0}_{i!} \rat_{X_{i}\setminus X_{i-1}}\simeq \rat_{T}$ implies,  by  base change,  that every fiber of $f^{0}_{i!}$ has a unique irreducible component of top dimension $d$. By \ci[Exp. XVIII, Remarque 2.10.1 ]{sga43},  the trace morphism ${\rm Tr}_{f_{i}^0}$  is an isomorphism. 
We conclude, in particular,  that the morphism $R^{2d} j_!$ must be injective.  The CLAIM is proved.

The conclusion follows from the fact that the morphism ${\rm Tr}_{f^0_i}^{-1} \circ {\rm Tr}_{f_i}$ splits the short exact sequence stemming from (\ref{ggoo}).
\iffalse
In order to prove that  the local system  $R^{2d} f_{i*} \rat_{X_{i}}$ is constant of rank $i$, we proceed as follows.
Since $T$ is normal and irreducible,
the fundamental group of every non-empty Zariski open subset of $T$ surjects on the fundamental group of $T$.
Therefore,  it suffices to show that the restriction $R^{2d} f_{i*} \rat_{X_{i}}$ to a non-empty Zariski open subset is constant. By  proper base change, we are free to prove the statement  with  $T$ replaced  by a non-empty Zariski open subset.

Since $R^{2d} f^{0}_{i!} \rat_{X_{i}\setminus X_{i-1}}\simeq \rat_{T}$ and $f^{0}_{i}$
has  fibers  of dimension $d$, the fibers of the morphism $f^{0}_{i}$ have a unique irreducible component
of dimension $d$. 
Shrinking $T$ if necessary, we may assume that, for  every $0<i\le m$, the closure $Y_{i}$ of $X_{i}\setminus X_{i-1}$
is flat with irreducible fibers over $T$. The natural morphism $g: \coprod_{i=1}^{m}{Y_i}\rightarrow X$ over $T$ induces a bijection on the $d$-dimensional irreducible components of the fibers. Let $g_{i}: Y_{i}\rightarrow T$ be the restriction of $g$. We thus obtain an isomorphism:
 $$\oplus_{i=1}^{m} R^{2d} (g_{i}\circ f)_{*} \rat_{Y_i} =  R^{2d} (g\circ f)_{*} \rat_{\coprod_{i=1}^{m}{Y_i}}\stackrel{\simeq}\rightarrow R^{2d} f_{i*} \rat_{X_{i}}.$$ 
By \ci[Exp. XVIII, Remarque 2.10.1]{sga43},  we deduce that  $R^{2d} (g_{i}\circ f)_{*} \rat_{Y_i}\simeq \rat_{T}^{i}$
and  therefore that $R^{2d} f_{i*} \rat_{X_{i}}\simeq \rat_{T}^{i}$.
\fi
\end{proof}

\begin{pr}\la{rtop} {\rm ({\bf  The direct image sheaves in top degree})}
Let  $R^{10}_{\ogmm}:= R^{10} \pogmm_* \rat  , $, $R^{10}_{\ogm}:= R^{10}\pogm_* \rat$,
and $R^{10}_{\ogn}:= R^{10} \pogn_* \rat, $
Then we have canonical  isomorphisms  of constructible sheaves:
%polarizable variations of pure Hodge structures of weight $10$:
%%%\[
%%%\begin{tabular}{ccc}
%%%${R^{10}_{\ogmm}}_{|\ogb \setminus \Delta} \simeq \rat_{\ogb \setminus \Sigma}$ , &
%%%%{R^{10}_{\ogmm}}_{|\Sigma \setminus \Delta} \simeq \rat_{\Sigma \setminus \Delta}, 
%%%&
%%%${R^{10}_{\ogmm}}_{| \Delta} \simeq \rat^{\oplus 2}_{\Delta}$; 
%%%\\
%%%${R^{10}_{\ogm}}_{|\ogb \setminus \Sigma} \simeq \rat_{\ogb \setminus \Sigma}$, &
%%%${R^{10}_{\ogm}}_{|\Sigma \setminus \Delta} \simeq \rat^{\oplus 2}_{\Sigma \setminus \Delta}$, &
%%%${R^{10}_{\ogm}}_{| \Delta} \simeq \rat^{\oplus 4}_{\Delta}$; 
%%%\\
%%%${R^{10}_{\ogn}}_{|\ogb \setminus \Sigma} \simeq \rat_{\ogb \setminus \Sigma}$ & 
%%%${R^{10}_{\ogn}}_{|\Sigma \setminus \Delta} \simeq \rat_{\Sigma \setminus \Delta} \oplus \ms{L}_{\Sigma \setminus \Delta}$ , &
%%%${R^{10}_{\ogn}}_{| \Delta} \simeq \rat^{\oplus 2}_{\Delta}$,
%%%\end{tabular}
%%%\]

\beq\la{isortop}
\xymatrix{
{R^{10}_{\ogmm}}_{|\ogb \setminus \Delta} \simeq \rat_{\ogb \setminus \Sigma} , &
%{R^{10}_{\ogmm}}_{|\Sigma \setminus \Delta} \simeq \rat_{\Sigma \setminus \Delta}, 
&
{R^{10}_{\ogmm}}_{| \Delta} \simeq \rat^{\oplus 2}_{\Delta}; 
\\
{R^{10}_{\ogm}}_{|\ogb \setminus \Sigma} \simeq \rat_{\ogb \setminus \Sigma}, &
{R^{10}_{\ogm}}_{|\Sigma \setminus \Delta} \simeq \rat^{\oplus 2}_{\Sigma \setminus \Delta}, &
{R^{10}_{\ogm}}_{| \Delta} \simeq \rat^{\oplus 4}_{\Delta}; 
\\
{R^{10}_{\ogn}}_{|\ogb \setminus \Sigma} \simeq \rat_{\ogb \setminus \Sigma} & 
{R^{10}_{\ogn}}_{|\Sigma \setminus \Delta} \simeq \rat_{\Sigma \setminus \Delta} \oplus \ms{L}_{\Sigma \setminus \Delta} , &
{R^{10}_{\ogn}}_{| \Delta} \simeq \rat^{\oplus 2}_{\Delta},
}
\eeq
where $\ms{L}$ is the rank one local system on $\Sigma \setminus \Delta$ corresponding to the 
$\zed/2\zed$-module of character $-1$ via the isomorphism $\pi_1 (\Sigma \setminus \Delta)
\simeq \zed/2\zed$.
\end{pr}

\begin{proof}
For ease of notation: 
 if $W$ is a variety,  $\varphi:W\rightarrow \ogb$ is a morphism 
and  $Z\subset \ogb$ is a  subvariety, then we set $W_{Z}:=\varphi^{-1}(Z)$ and 
 ${R^{10}_{W}}:= R^{10} \varphi_! \rat_{W}$.

The statements in the first column  of (\ref{isortop}) follow from the fact that the morphisms in question
are flat of relative dimension five with integral fibers over the indicated loci; see  Remark \ref{piattezza} and Proposition \ref{componentisigma}.

{\bf Proof that  ${R^{10}_{\ogm}}_{|\Sigma \setminus \Delta} \simeq \rat^{\oplus 2}_{\Sigma \setminus \Delta}$.} Recall from Lemma \ref{dtb} and proper base change that  
\beq \label{dtmsigm}
(R \bogm_* \rat_{\ogm})_{| M_{\Sigma \setminus \Delta}}  \simeq
\ms{IC}_{\ogmm_{\Sigma \setminus \Delta}} \oplus  \rat_{({\rm Sym}^2 \ogmp)_{\Sigma \setminus \Delta}}[-2] .
\eeq
By \cite{Lehn-Sorger}, the singularities of $M_{\Sigma \setminus \Delta}$ are of type $A_1$ quotient singularities, 
so  that we have $\ms{IC}_{\ogmm_{\Sigma \setminus \Delta}}  \cong \rat_{ \ogmm_{\Sigma \setminus \Delta}} $ (cf. \cite{Gottsche Sorgel}). Pushing forward (\ref{dtmsigm}) via the proper morphism $\pogmm: \ogmm_{\Sigma \setminus \Delta} \to \Sigma \setminus \Delta$ we get:
\[
R \pogm_* \rat_{\ogm_{\Sigma \setminus \Delta}}=Rm_* Rb_* \rat_{\ogm_{\Sigma \setminus \Delta}} =Rm_* \rat_{\ogmm_{\Sigma \setminus \Delta}} \oplus Rp^2_*\rat_{({\rm Sym}^2 \ogmp)_{\Sigma \setminus \Delta}}[-2] .
\]
and thus:
\[
R^{10} \pogm_* \rat_{\ogm_{\Sigma \setminus \Delta}}=R^{10}m_* \rat_{\ogmm_{\Sigma \setminus \Delta}} \oplus R^8p^2_*\rat_{{\rm Sym}^2 \ogmp_{\Sigma \setminus \Delta}}.
\]
By using  Fact \ref{fatti da usare}.(2) and  the irreducibility of  the five dimensional $\ogmm_b$ (cf. Proposition \ref{componentisigma}) and of the four dimensional $\ogmp_b$ (cf. Remark \ref{exCompecc}),  we see that $R^{10}m_* \rat_{\ogmm_{\Sigma \setminus \Delta}}=\rat_{\Sigma \setminus \Delta}(-5)$ and $R^8p^2_*\rat_{{\rm Sym}^2 \ogmp_{\Sigma \setminus \Delta}} \cong \rat_{\Sigma \setminus \Delta}(-4)$ and we are done.

{\bf Proof that $R^{10}_{\ogn_{\Sigma \setminus \Delta}} \simeq \rat_{\Sigma \setminus \Delta} \oplus \ms{L}_{\Sigma \setminus \Delta}$.} 
The variety $\ogn_{\Sigma \setminus \Delta}$ is irreducible (cf. \ci[Proof of Prop 1.3]{ra08}), the morphism 
$\ogn_{\Sigma \setminus \Delta} \to\Sigma \setminus \Delta$ is flat,  with each of its fibers are reduced 
 and with two
 irreducible components (cf. Proposition 
 \ref{componentisigma}). The sheaf $R^{10}_{\ogn_{\Sigma \setminus \Delta}}$ is the $\rat$-linearization of the sheaf  
 of sets $\frak{I}$ of irreducible components of $\ogn_{\Sigma \setminus \Delta}/(\Sigma \setminus \Delta)$.
 Let $V \subseteq \ogn_{\Sigma \setminus \Delta}$ be the locus parameterizing line bundles. By Proposition 
 \ref{componentisigma} and its proof we have that: the morphism $V\to \Sigma \setminus \Delta$ is
 surjective and smooth;  $V$ is dense in each fiber $\ogn_b$; 
  $V_b$  has two connected components.  
  %In particular, locally over   $\Sigma \setminus \Delta$ in the classical topology $V$ has two connected components.
  It follows that  the sheaf of sets
 $\frak I$ is locally constant, with stalks of cardinality two.
 We now examine  the monodromy  of $\frak I$ and  relate it to that of the broken curves.
 Let $b \in \Sigma \setminus \Delta.$ The corresponding curve $C_b = C_{1,b}+C_{2,b} \in |2H|$ has two irreducible components. The  two connected components of $V_b$ parameterize  line bundles of bi-degree  $(2,3)$ and $(3,2)$
 (cf. Lemma \ref{fibratilineari}).  Since $\Sigma \setminus \Delta$ is nonsingular and connected,
 the monodromy can be detected on a Zariski dense open subset $ U \subseteq\Sigma \setminus \Delta$.  We take
 $U:= {\rm Sym}^2\ogbp^o,$ where $\ogbp^o \subseteq \ogbp = |H|$ is the locus of nonsingular curves. It is clear that
 looping in $U$ around the diagonal exchanges the two components of $C_b$, the same way in which the two lines of the corresponding broken conic are. It follows that the bidegrees are swapped as well. The desired conclusion follows.

{\bf Proof of the remaining isomorphisms.}
We now focus on the third column, of (\ref{isortop}), i.e.  on the situation  over $\Delta$.
By proper base change we need to show that:
(a) ${R^{10}_{\ogmm_{\Delta}}}\simeq \rat^{2}_{\Delta}$;
 (b) ${R^{10}_{\ogm_{\Delta}}}\simeq \rat^{4}_{\Delta}$; 
(c) ${R^{10}_{\ogn_{\Delta}}}\simeq \rat^{2}_{\Delta}$.

We are going to use Lemma \ref{rtoplemma}
with filtrations of length two in cases (a) and (c) and of length four in case (b).

We first deal with cases (a) and (c).
In case $(a)$ and (c), we choose the first space $X_1$ of the filtration to be the  locus of rank two  semistable sheaves (cf. Proposition \ref{comparingstab}), i.e. 
in case (a), we set
 $X_{1}:=\ogmm^{I}_{\Delta}$, in case (c),  we  set $X_{1}:=\ogn^{I}_{\Delta}$.
 By Proposition \ref{rank 2}, $\ogmm^{I}_{\Delta}$ and $\ogn^{I}_{\Delta}$ are projective over $\Delta$ with
 $5$-dimensional irreducible fibers.
By Fact \ref{fatti da usare}.(3), we have that ${R^{10}_{\ogmm^{I}_{\Delta}}}\simeq \rat_{\Delta}\simeq {R^{10}_{\ogn^{I}_{\Delta}}}$. 
Note that with our choice of $X_1$, we have that $X_2\setminus X_1$ is,  $\ogmm_\Delta^{I\!I}$ in case $(a)$,
$\ogn_\Delta^{I\!I}$ in case $(c)$. By Proposition \ref{propriass} (the irreducibility of the fibers of $\ogmm^{I\!I}_{C'}$ and $\ogn^{I\!I}_{C'}$), and by Fact \ref{fatti da usare}.(1), we are thus in the position of applying Lemma \ref{rtoplemma} and reach the desired conclusion in cases (a) and (c).

In order to prove (b) via an application of  Lemma \ref{rtoplemma}, 
we  consider the following increasing  filtration of 
%the reduced variety $\ogm_{\Delta}^{red}$ associated with 
$\ogm_{\Delta}$ into closed subvarieties:
\beq\la{filtrcloz}
X_1:=E^I_\Delta \subseteq X_2:= E_\Delta \subseteq 
X_3:= \bogm^{-1} (\ogmm_{\Delta}^{I})  \subseteq X_4 := \ogm_\Delta.
\eeq
Recall Proposition \ref{Compecc} and the subvarieties $E_\Delta= E^I_\Delta \coprod E^{I\!I}_\Delta$ defined in (\ref{defEdelta}).
By using that $\bogm$ is an isomorphism away from the center ${\rm Sym}^2 \ogmp$ of the blow up, the successive differences are seen to be as follows:
\beq\la{ghth}
X_2 \setminus X_1 = E^{I\!I}_\Delta, 
\quad X_3 \setminus X_2 = \ogmm^I_\Delta \setminus
{\rm Sym}^2_\Delta \ogmp, \quad
X_4 \setminus X_3= \ogmm^{I\!I}_\Delta 
\eeq
As shown in Proposition \ref{Compecc}, the projective variety  $X_1=E_{\Delta}^{I}$ 
has  irreducible fibers of dimension $5$ over $\Delta$: by  Fact \ref{fatti da usare}.(3) we get 
$R^{10}_{E_{\Delta}^{I}}\simeq \rat_{\Delta}.$     

By Proposition \ref{Compecc}, the variety 
$X_2 \setminus X_1= E_{\Delta}^{I\!I}$ is flat surjective over $\Delta$ with irreducible fibers.
By Fact \ref{fatti da usare}.(1),  we get that $R^{10}_{E_{\Delta}^{I\!I}\setminus E_{\Delta}^{I}}\simeq \rat_{\Delta}.$  

By Remark \ref{exCompecc}, the fibers of  $\ogmm^I_\Delta$ and of $X_3\setminus X_2$ over the points of $\Delta$ are non empty and irreducible of dimension five.
By Fact \ref{fatti da usare}.(3), we conclude that $R^{10}_{\ogmm^I_\Delta \setminus
({\rm Sym}^2 \ogmp)_\Delta} \simeq \rat_\Delta.$

The morphism  $\ogmm_\Delta \to \Delta$ is flat, so that its restriction to the open subset $M^{I\!I}_\Delta \subset \ogmm_\Delta $  is also flat.  By Proposition \ref{propriass},  the fibers of $M^{I\!I}_\Delta \to \Delta$ are irreducible of dimension
five. By Fact \ref{fatti da usare}.(1), we have that $R^{10}_{M^{I\!I}_\Delta}\simeq \rat_\Delta$.
We conclude that (b) holds in view of Lemma \ref{rtoplemma}. 
\end{proof}

\begin{rmk}\la{onl}
{\rm ({\bf Properties of $\ms{L}$ in (\ref{isortop})})}
We record the following elementary properties of $\ms{L}$.
Let $T$ be an irreducible normal variety of positive dimension, let $\ms{L}_T$ be the rank one local system on ${\rm Sym}^2 T \setminus \Delta_T$ corresponding to the 
$\zed/2\zed$ module of character $-1$ via the epimorphism $\pi_1 ({\rm Sym}^2 T \setminus \Delta_T)
\twoheadrightarrow  \zed/2\zed$. Let $j: U \to {\rm Sym}^2 T \setminus  \Delta_T$ be any open immersion with $U \neq \emptyset$,
and let $i:{\rm Sym}^2 T \setminus \Delta_T \to {\rm Sym}^2 T$ be the evident open immersion.
Let $u:T\to T'$ be a non constant  morphism of irreducible normal varieties. Let 
$u^{(2)}: {\rm Sym}^2 T \to {\rm Sym}^2 T'$ and 
${u^{(2)}}': {\rm Sym}^2 T  \setminus {u^{(2)}}^{-1}(\Delta_{T'}) \to {\rm Sym}^2 T' \setminus  \Delta_{T'}$ be the resulting evident morphisms.
Then: 
\ben
\item
 $i_* \ms{L} = i_! \ms{L}=Ri_! \ms{L}=Ri_* \ms{L} $; 
 \item $j_*j^*\ms{L} = \ms{L}$;
 \item ${\ms{L}_T}_{|  {\rm Sym}^2 T  \setminus {u^{(2)}}^{-1}(\Delta_T)}
= {{u^{(2)}}'}^* \ms{L}_{T'}$. 
 \een
As to {\rm (1)}: the first $=$  is because there are no local invariants around the diagonal; the second $=$
is because $i$ is an open immersion; the third equality follows from Verdier duality and  the self-duality of  $\ms{L}$. As to {\rm (2)}: in view of the normality assumption, it is valid for any local system, due to the surjection of fundamental groups
$\pi_1 (U,u) \to \pi_1({\rm Sym}^2 T \setminus  \Delta_T).$
As to {\rm (3)}: this is because the natural surjection
of fundamental groups above for $T$ factors through the one for $T'$.

Finally, for every integer $\frak d \in \zed$, when $T$ is irreducible nonsingular, the local system $\ms{L}$ underlies 
a unique structure of polarizable pure Hodge-Tate variation of Hodge structures of type $(\frak d, \frak d)$; see \S\ref{gennot}.
\end{rmk}

\section{The Ng\^o Strings of several Lagrangian fibrations}\la{nonon}$\;$
{  
The aim of this section is to determine the Ng\^o strings (cf. Definition \ref{defngostrbis}) of the two Lagrangian fibrations $\ogm, \ogn \to \ogb$ as well as  of several other auxiliary fibrations.
In  \S\ref{nzg} we introduce the Ng\^o strings that will appear in the Decomposition Theorem for  $\ogm \to \ogb$ and $ \ogn \to \ogb$. Then in  \S\ref{dtb't}  and \S\ref{dtbnon} we determine the Decomposition Theorem for  the two universal families of curves  of  the linear systems $|H|$ and $|2H|$ on the K3 surface. To prove  Theorems A, B, and B$'$, we need to compute the cohomology of the Ng\^o strings appearing the Decomposition Theorem for $\ogm, \ogn \to \ogb$. We do by realizing them as direct summands of the Decomposition theorem for other Lagrangian fibrations whose cohomology is known.  This is done in this in  \S\ref{kt}, \S\ref{kt1}, and \S\ref{kt2}. Finally, in   \S\ref{dtmn}  we prove that the Ng\^o strings (\ref{sics}) are precisely those appearing in Decomposition Theorem for the two Lagrangian fibrations. This builds on  \S\ref{dtb't}  and \S\ref{dtbnon}.
}

All decomposition-type results that follow take place in the categories $D^b MHM_{alg} (-)$ and have
an evident counterpart (turn off the Tate shifts) in the corresponding categories $D^b(-,\rat)$ via the functor
$\rm rat$. We only state the stronger result in $D^b MHM_{alg} (-).$

\subsection{Some relevant strings}\la{nzg}$\;$

Let  $u: \ms{C} \to \ms{B}$  be a family of nonsingular projective, not necessarily connected, curves
of genus $g$. Let $v: \ms{P} \to \ms{B}$ be the relative ${\rm Pic}^0_{\ms{C}/\ms{B}}$.
We  consider the higher direct image sheaves $\Lambda^\bullet_\ms{B}:= R^\bullet v_* \rat_\ms{P}$, with $\bullet \in [0,2g]$;
these are polarizable variations of pure Hodge structures of weight $\bullet$.
By taking into account Tate twists, we have  natural isomorphisms:
\beq\la{isowedge}
\xymatrix{
R^1u_* \rat_\ms{C} = R^1v_* \rat_{\ms{P}}; &  \Lambda^\bullet_{\ms{B}}=
\wedge^\bullet R^1u_* \rat_\ms{C}, & \Lambda^{g-\bullet}_{\ms{B}} =  \Lambda^{g+\bullet}_{\ms{B}} (\bullet).
}
\eeq

We come back to our families of curves $\m{C'}\to \ogbp$ and $\m{C}\to \ogb$ over the linear systems $\ogbp= |H|$ and $\ogb=|2H|$.
% onour   $K3$ surface  $\ks$.
%; see \S\ref{fissnot}.
By restricting these two families 
over the  respective sets of regular values $\ogbp^o$ and $\ogb^o$, we obtain families  of type $u$  as above,  of genus $2$ and $5$ respectively, and with connected fibers.
By first restricting the family of curves $\m{C} \to \ogb$ over $\Sigma^o:={\rm Sym}^2 {\ogbp}^o \setminus \Delta_{\ogbp}$, and then  by normalizing the total space of the
resulting family, we obtain a family of type $u$ as above,  with connected total space, but with disconnected fibers given by the disjoint union of two curves of genus $2$.
We thus have (cf. (\ref{bu1})):
\beq\la{norm}
\xymatrix{
\m{C}_{\ogb^o} \to \ogb^o, &   (\m{C}'_{\ogbp^o} \to \ogbp^o)=(\m{C}'_{\Delta^o} \to \Delta^o),  &
\widehat{\m{C}_{\Sigma^o}} \to \Sigma^o,
\\
\{\Lambda^\bullet_{\ogb}:= \Lambda^\bullet_{\ogb^o}\}_{\bullet =0}^{10}, 
&   \{\Lambda^\bullet_{\Delta}:= \Lambda^\bullet_{\ogbp^o}\}_{\bullet =0}^4,  &
\{\Lambda^\bullet_{\Sigma}:= \Lambda^\bullet_{\Sigma^o}\}_{\bullet =0}^8,
}
\eeq
where we have dropped some decorations. We also have the following epimorphisms of groups schemes:
\beq\la{epigs}
\xymatrix{
\p_{|\ogb^o} \ar@{->>}[r]^-= &{\rm Pic}^0_{\m{C}_{\ogb^o}/\ogb^o} &
\ppr_{|\ogbp^o} \ar@{->>}[r]^-= & {\rm Pic}^0_{{\m{C'}_{\!\!\Delta^o}}/\Delta^o}, &
\p_{|\Sigma^o} \ar@{->>}[r]& {\rm Pic}^0_{\widehat{\m{C}_{\Sigma^o}}/\Sigma^o},
}
\eeq
that fiber-by-fiber realize the Chevalley  devissages into affine and Abelian parts of the fibers
of the corresponding group schemes $\p,$ $\ppr$, and $\p$ again, respectively

Recall that $\ms{L}$  in \S\ref{tsr1} is a local system on $\Sigma \setminus \Delta$, and that this latter  contains $\Sigma^o$.
Let $\ms{L}^o: =\ms{L}_{|\Sigma^o}$. Then we have:
 \beq\la{ldf}
 {i_{\Sigma^o}}_* \ms{L}^o = {i_{\Sigma \setminus \Delta}}_*
\ms{L} = {i_{\Sigma \setminus \Delta}}_!
\ms{L},
\eeq
where the last equality, is due to the fact that the local monodromy of $\ms{L}$ around the diagonal is given by multiplication by $-1$.

We introduce the following complexes, viewable in $D^b MHM_{alg} (-)$ or in $D^b(-, \rat)$, and we simply call them strings --we reserve the term Ng\^o strings for those strings that actually appear in the Decomposition Theorems \ref{nstabis} and \ref{defngostrbis}-: 
\beq\la{sics}
\begin{aligned}
\ms{S}_\ogb := \bigoplus_{b=0}^{10} \ms{IC}_\ogb (\Lambda^{b}_{\ogb})[-b],  \quad \quad \quad &
\ms{S}_\Delta := \bigoplus_{b=0}^4 \ms{IC}_\Delta (\Lambda^{b}_{\Delta})[-b], \\
\ms{S}^+_\Sigma := \bigoplus_{b=0}^8 \ms{IC}_\Sigma (\Lambda^{b}_{\Sigma})[-b],  \quad \quad \quad  &
\ms{S}^-_\Sigma := \bigoplus_{b=0}^8 \ms{IC}_\Sigma (\Lambda^{b}_{\Sigma} \otimes \ms{L}^o)[-b],
\end{aligned}
\eeq
where $\ms{L}$ is described in Proposition \ref{rtop}.
They are all subject to the Relative Hard Lefschetz
theorem, which here boils down to: 
\beq\la{plds}
\begin{aligned}
\Lambda_{\ogb}^{5-\bullet}\simeq  \Lambda_{\ogb}^{5+\bullet} (\bullet),  \quad \quad \quad &
\Lambda_{\Delta}^{2-\bullet}\simeq  \Lambda_{\Delta}^{2+\bullet} (\bullet), \\
\Lambda_{\Sigma}^{4-\bullet}\simeq  \Lambda_{\Sigma}^{4+\bullet} (\bullet),  \quad \quad \quad &
\Lambda_{\Sigma}^{4-\bullet} \otimes \ms{L}^o \simeq  \Lambda_{\Sigma}^{4+\bullet} \otimes \ms{L}^o
(\bullet)
\end{aligned}
\eeq

\begin{rmk}\la{szfs}
In the sequel of the paper, we show that the first two and last two summand in each of the  strings
in  (\ref{sics})
are, when viewed in $D^b (\ogb, \rat)$,  shifted sheaves; see  (\ref{z2}), (\ref{z5b}) and (\ref{z6B}).
\end{rmk}

\begin{rmk}\la{rovo}
The local systems of type $\Lambda$ are endowed  with the
structure of polarizable variations of pure Hodge structures they acquire as direct image sheaves. 
We have the following relations between rational  graded  polarizable pure  Hodge structures:
\ben
\item
$H^*(\ogn)= H^*\left( \ks^{[5]}\right)$; see Lemma \ref{lemma birazionali}.

\item $H^*(\ogmp)= H^*\left( \ks^{[2]} \right)=H^*(\ms{S}_\Delta)$; see Lemmata \ref{lemma birazionali} and  \ref{m0}.

\item  $H^*\left({\rm Sym}^2{\ogmp}\right)=H^*(\ms{S}_\Sigma^+)$; see Proposition \ref{dtforsym}.

\item\la{rasta}  $H^*\left({\ogmp}^2\right)= H^*(\ms{S}_\Delta)^{\otimes 2} =
H^*(\ms{S}_\Sigma^+) \oplus H^*(\ms{S}_\Sigma^-)$; see Lemma \ref{m11} and Proposition \ref{dtforsym}.

\item By the forthcoming  Proposition \ref{dtwe}, and by  Lemma \ref{dtb}, we have:
\beq\la{ee4455}
H^*(\ogm) \supseteq  I\!H^*(\ogmm) \supseteq H^*(\ms{S}_B) \subseteq H^*(\ogn).
\eeq

%$H^*(\ogm), I\!H^*(\ogmm),$ and $H^*(\ms{S}_B)$. 
\een
%By the forthcoming  Proposition \ref{dtwe}, and by  Lemma \ref{dtb}, we have:
%\beq\la{ee4455}
%H^*(\ogm) \supseteq  I\!H^*(\ogmm) \supseteq H^*(\ms{S}_B) \subseteq H^*(\ogn).
%\eeq
\end{rmk}

\subsection{The Decomposition Theorem  for the genus two universal curve $\m{C}'\to \ogbp$}\la{dtb't}$\;$

The complete linear system $\ogbp = |H| \simeq \pn{2}$ of our genus two curve $C$ on $K3$ is two dimensional 
and  we have the universal curve  morphism:
\beq\la{gammaprimo}
\gamma': {\m{C}'}^3 \to \ogbp^2 = |C|.
\eeq
The total space ${\m{C}'}$ is nonsingular and three dimensional. The fibers of $\gamma'$ are all integral curves.
Let ${\ogbp}^o \subseteq \ogbp$ be the open subset of regular values of $\gamma'$, and let $i_{\ogbp^o}: {\ogbp}^o
\to \ogbp$ be the resulting open embedding. Let $R^1_\m{C'}$ be  local system restriction of $R^1\gamma'_* \rat$ to ${\ogbp}^o$; 
it coincides with $\Lambda^1_{\Delta}:=\Lambda^1_{{\ogbp}}:= \Lambda^1_{{\ogbp}^o}$ in (\ref{norm}).

\begin{lm}\la{dtb''} {\rm ({\bf  Decomposition Theorem for $R\gamma'_* \rat_{\m{C'}}$})}
The Decomposition Theorem   for the direct image complex  $R\gamma'_* \rat$  takes the following form:
\beq\la{dtb'z}
R\gamma'_* \rat_{\m{C}'} \simeq  \rat_{\ogbp}\oplus i_* \Lambda^1_{\Delta} [-1] \oplus \rat_{\ogbp}[-2](-1).
\eeq

\end{lm}
\begin{proof} For reasons of bookkeeping, we find it more natural to prove the equivalent statement:
\beq\la{dtb'}
R\gamma'_* \rat_{\m{C}'} [3] \simeq  \rat_{\ogbp}[2][1] \oplus i_* \Lambda^1_{\Delta} [2][0] \oplus \rat_{\ogbp}[2][-1](-1),
\eeq
where the summands appear with increasing perverse degree $-1,0,1$.

The following is standard (cf. \ci[Theorem 1.5.3 and Fact 4.5.4]{depcmi}: the following summands appear in the Decomposition Theorem for $R\gamma'_* \rat_{\m{C}'} [3]$:
\[
\rat_{\ogbp}[2][1] \oplus IC_{\ogbp} (\Lambda^1_{\Delta})[0] \oplus \rat_{\ogbp}[2][-1](-1).
\]
Since the fibers of $\gamma'$ are irreducible, we have that $R^2\gamma'_* \rat = \rat_{\ogbp}$, hence:
\beq\la{z1}
IC_{\ogbp} (\Lambda^1_{{\ogbp}})[0] = {i_{{\ogbp}^{o}}}_* \Lambda^1_{{\ogbp}} [2].
\eeq
If in (\ref{dtb'}) there were other summands than these three above, then, by the Relative Hard Lefschetz symmetries, they would need to be in perverse degree zero:
else, they would contribute nontrivially to $R^j\gamma'_*\rat =0$ for some $j \geq 3$, a contradiction
(the fibers of $\gamma$ are curves).

If there were other summands in perverse degree zero, then they would be of the form $I^0 \oplus I^1$, where
the exponent refers to the dimension of the support of the intersection complex. An $I^0$ would contribute non trivially to $R^3\gamma'_* \rat=0$, a contradiction. An $I^1$ would be a sheaf placed in cohomological degree $-1$, and it would thus contribute non trivially to $R^2\gamma'_* \rat$. Since we know that this latter is ${\rat}_{\ogbp}$
and is already accounted for by the summand  $\rat_{\ogbp}[2][-1](-1)$ on the rhs of (\ref{dtb'}), we  get a  contradiction.
\end{proof}

\subsection{The Decomposition Theorem for the genus five universal curve $\m{C}\to \ogb$}\la{dtbnon}$\;$

The complete  linear system $|2C|$ is five dimensional so that we have the associated universal family of curves:
\[
\gamma: {\m{C}}^6 \to \ogb^5 = |2C|.
\]
The total space ${\m{C}}$ is nonsingular and six dimensional. The fibers of $\gamma$ are integral curves away from the integral divisor 
$\Sigma:= {\rm Sym}^2 \ogbp \subseteq \ogb$. Let $\Delta \subseteq \Sigma$ be the diagonal.
The fibers of $\gamma$ over $\Sigma\setminus \Delta$ are reduced and are the union of two irreducible components.
The two irreducible components of a fiber over the locus $\Sigma^o := {\rm Sym}^2 \ogbp^o \setminus \Delta$
are exchanged by the monodromy action around the diagonal.
 Let $\ms{L}$ be the evident rank one local system on $\Sigma \setminus \Delta$ with monodromy $-1$ around $\Delta$ (cf. \S\ref{tsr1}). We view it as a polarizable variation of pure Hodge structures of Hodge-Tate type, and of weight zero on $\Sigma \setminus \Delta$. We have:
 \beq\la{kvr4}
 (R^2\gamma_* \rat_{\m{C}})_{|\Sigma \setminus \Delta}
 =\left(  \rat_{\Sigma \setminus \Delta} \oplus \ms{L}\right) (-1).
 \eeq
 The fibers of $\gamma$ over $\Delta$ are  irreducible non-reduced:  they are   two times a curve in $|C|$ (recall that on our $K3$
surface $\ks$  all curves in $|C|$ are integral).

Let ${\ogb}^o \subseteq \ogb$ be the open subset of regular values of $\gamma$.  Let $R^1_\m{C}$ be the 
local system restriction of $R^1 \gamma_* \rat$ to ${\ogb}^o$; it coincides with  $\Lambda^1_{{\ogb}}$ in (\ref{norm}). Let $i_{\ogb^o}: \ogb^o \to \ogb$ be the evident open immersion.

\begin{lm}\la{dtbo}
{\rm ({\bf  Decomposition Theorem for $R\gamma_* \rat_{\m{C}}$})}
The Decomposition Theorem  for $R\gamma_* \rat$  takes the following form: 
\beq\la{dtb'2o}
R\gamma_* \rat_{\m{C}} \simeq  \rat_{\ogb}\oplus \left({i_{B^0}}_* \Lambda^1_{{\ogb}} [-1] \oplus {i_{\Sigma
\setminus \Delta}}_* \ms{L} [-2](-1) \right) \oplus \rat_{\ogb}[-2] (-1).
\eeq
\end{lm}
\begin{proof}
We prove the equivalent:
\beq\la{dtb'2}
R\gamma_* \rat_{\m{C}} [6] \simeq  \rat_{\ogb}[5][1] \oplus \left({i_{B^0}}_* \Lambda^1_{{\ogb}} [5] \oplus {i_{\Sigma
\setminus \Delta}}_* \ms{L} [4] (-1) \right) [0] \oplus \rat_{\ogb}[5][-1](-1),
\eeq
where the summands are grouped by perversity $(-1,0,1)$.

By looking at the regular part of $\gamma$, we deduce that
the following summands appear in the Decomposition Theorem for $R\gamma_* \rat_{\m{C}} [6]$:
\[
 \rat_{\ogb}[5][1] \oplus IC_{\ogb} (\Lambda^1_{{\ogb}})[0] \oplus \rat_{\ogb}[5][-1] (-1).
\]
By the Relative Hard Lefschetz symmetries, any additional summand   would need to be in perverse degree zero:
else, they would contribute nontrivially to $R^j\gamma_*\rat =0$ for some $j \geq 3$, a contradiction (the fibers of $\gamma$ are curves).

The direct summand contribution of $IC_{\ogb} (\Lambda^1_{\ogb})[0]$ to $R^2 \gamma_* \rat$ is the sheaf 
$\m{H}^{-4} (IC_{\ogb} (\Lambda^1_{\ogb^o}))$. By the condition of support for intersection complexes,
this sheaf is supported in dimension $\leq 3$.  Recall that the curves over the  points of the four dimensional
$\Sigma^o$ have two irreducible components. It follows that there must be a contribution of $\Sigma$, to the l.h.s. of (\ref{dtb'2}), of the form
$IC_{\Sigma} (\frak{L})[0]$,
where $\frak{L}$ is some rank one local system on some open dense subset $\Sigma'$ of $\Sigma$.
Since it contributes to $R^2 \gamma_* \rat$,  as a polarizable variation of pure Hodge structures, the local systems $\frak L$ is necessarily of pure Hodge-Tate type
with weight $2$.  The intersection  complex $IC_{\Sigma} (\frak{L})$ has non zero cohomology sheaves in degrees 
contained in the interval $[-4,-1]$, where $i_{\Sigma'}: \Sigma \to \ogb$ is the evident locally closed embedding. The sheaf  $\m{H}^{-4} (IC_{\Sigma} (\frak{L}))$ contributes to $R^2 \gamma_*\rat$.
The remaining cohomology  sheaves in degrees in the interval $[-3,-1]$  contribute to  $R^j \gamma_* \rat $ for $j \geq 
3$, and are thus zero.  It follows that $IC_\Sigma (\frak{L})$ is the  sheaf  ${i_{\Sigma'}}_*\frak{L}$ placed in degree $-4$. The local system $\frak{L}$ agrees with $\ms{L}$ where it is defined. 
Since the two local systems are both of rank one and of  Hodge-Tate type, it  follows that
$\frak L = \ms{L}(-1)$, where they are both defined. It follows that   ${i_{\Sigma'}}_*\frak{L}= {i_{\Sigma \setminus \Delta}}_*\ms{L}(-1)$, so that $IC_\Sigma (\frak{L}) = {i_{\Sigma
\setminus \Delta}}_* \ms{L} [4](-1)$.

Note that the stalks of $R^2\gamma_* \rat$ on $\Delta$ are one dimensional. Note also that 
${i_{\Sigma \setminus \Delta}}_*L= {i_{\Sigma \setminus \Delta}}_!L$ has zero stalks on $\Delta$ (no local invariants
near the diagonal).

By  inspecting the cohomology sheaves of $IC_\ogb (\Lambda^1_{\ogb})$, we deduce that this perverse sheaf is a sheaf
placed in degree $-5$, hence of the predicated form:
\beq\la{z2}
IC_\ogb (\Lambda^1_{\ogb}) = {i_{B^o}}_*  \Lambda^1_{\ogb^o} [4]
\qquad
\mbox{(equivalently, $\ms{IC}_\ogb (\Lambda^1_{\ogb}) = {i_{\ogb^o}}_*  \Lambda^1_{\ogb}$)}.
\eeq
\end{proof}

\subsection{The Ng\^o strings  for  the Lagrangian fibration $\pogmp:  \ogmp \to \ogbp$}\la{kt}$\;$

Recall that $\ogmp$ is a four-dimensional moduli space which admits  the Lagrangian fibration 
$\pogmp:\ogmp\to \ogbp$ (\ref{notm}), that we have the local systems of type $\Lambda$ (\ref{notm}), the strings (\ref{sics}) and the isomorphisms (\ref{plds}).

\begin{lm}\la{m0} {\rm ({\bf Ng\^o strings  for $\pogmp:  \ogmp \to \ogbp$})}
The Decomposition Theorem  for $\pogmp:  \ogmp \to \ogbp$ has the following form:
\beq\la{z4b}
R {\pogmp}_* \rat_{\ogmp} \simeq \ms{S}_{\Delta}= \rat_{\ogbp} \bigoplus
\bigoplus_{b=1}^3 \ms{IC}_{\ogbp} (\Lambda^{b}_{\Delta}) [-b] \bigoplus \rat_{\ogbp} [-4](-2).
\eeq
where, moreover, the following intersection cohomology complex is a sheaf:
\beq\la{z5b}
\ms{IC}_{\ogbp} (\Lambda^{2\pm 1}_{\Delta}) = {i_{{\ogbp}^o}}_* R^1_{\m{C'}} \left(-\frac{1}{2} \mp
\frac{1}{2}\right).
\eeq
\end{lm}
\begin{proof}
By Proposition \ref{tadaw}, the morphism $\pogmp:  \ogmp \to \ogbp$ is part of a $\delta$-regular weak Abelian fibration.
 We can thus apply Theorem \ref{nstabis}, so that,  in view of the fact that the fibers of $\pogmp$ are irreducible
 (they are the compactified Jacobians of the integral locally planar curves in $\ogbp$ \cite{re80}),
we have that the Decomposition Theorem has the desired  form (\ref{z4b}). 

Since the general fiber of the group scheme ${\rm Pic}^0_{\m{C'}/\ogbp}$ is the Jacobian of a nonsingular curve of genus
two, by (\ref{plds})   $\Lambda^1_{\Delta}  \simeq \Lambda^3_{\Delta}(1)$, so that
in order to prove  (\ref{z5b})  it is enough to prove that:
\[
\ms{IC}_{\ogbp} (\Lambda^{1}_{\Delta}) = {i_{{\ogbp}^o}}_* \left((R^1 \gamma'_* \rat_{\m{C'}})_{| {B'}^o}\right).
\]
which follows from 
(\ref{z1}) and the fact, already observed at the beginning of \S\ref{dtb't},  that $\Lambda^{1}_{\Delta} = (R^1 \gamma'_* \rat_{\m{C'}})_{| {B'}^o}$.
% ($H^1$ of the Jacobian of a nonsingular curve equals $H^1$ of the curve).
\end{proof}

\subsection{The Ng\^o strings for  the Lagrangian fibration $\pogmp^2: \ogmp \times \ogmp \to \ogbp\times \ogbp$}\la{kt1}$\;$

Let $\ogmp$ be as in the  beginning of \S\ref{kt}. Given a product with the projections onto the factors, the box product is the ordinary tensor product of the pull-backs via the projections.
We consider the Lagrangian fibration $\pogmp^2=\pogmp \times \pogmp: \ogmp \times \ogmp \to \ogbp\times \ogbp$. Recall (\ref{notm})  that for simplicity, we are denoting $\Lambda^\bullet_{\ogbp^o}$ simply by $\Lambda^\bullet_{\Delta}$. We denote by $\Lambda^\bullet_{\Delta \times \Delta}$ the analogous local systems on $\ogbp^o \times \ogbp^o$ associated with $\ppr \times \ppr \to \ogbp \times \ogbp$.

\begin{lm}\la{m11}
The Decomposition Theorem for $\pogmp^2$ has the following form:
\beq\la{sabB}
\begin{aligned}
R\pogmp^2_* \rat_{\ogmp\times \ogmp}  \simeq & \ms{S}_{\Delta} \boxtimes \ms{S}_\Delta 
\\
= &\bigoplus_{b=0}^8 \bigoplus_{b'+b''=b} \left(   \ms{IC}_{\ogbp}(\Lambda^{b'}_{\Delta}) \boxtimes
\ms{IC}_{\ogbp}(\Lambda^{b''}_{\Delta})    \right)     [-b] \\
= & \bigoplus_{b=0}^8  \ms{IC}_{\ogbp \times \ogbp}(\Lambda^{b}_{\Delta \times \Delta}) [-b].
\end{aligned}
\eeq
Moreover, the following  summands are  sheaves: \beq\la{sab0B}
\ms{IC}_{\ogbp \times \ogbp}(\Lambda^{4\pm 4}_{\Delta \times \Delta}) (2\pm 2)
= \rat_{\ogbp \times \ogbp},
\eeq
\beq\la{tassB}
\ms{IC}_{\ogbp \times \ogbp} (\Lambda^{4\pm 3}_{\Delta \times \Delta}) 
= \bigoplus_{j=1}^2 {\rm pr}_j^* \ms{IC}_{\ogbp}(\Lambda^1_{\Delta})  (3\pm 3)
=  \bigoplus_{j=1}^2 {\rm pr}_j^*  \, {i_{{\ogbp}^o}}_* \Lambda^1_{\ogbp^o} (3\pm 3).
\eeq
\end{lm}
\begin{proof} For reasons of bookkeeping, we prove the theorem for the shifted $R\pogmp^2_* \rat_{\ogmp\times \ogmp} [8]$.

The projections ${\rm pr}_j: \ogbp \times \ogbp$ are smooth of relative dimension $2$, so that the functors 
${\rm pr}_j^*[2]$
preserve pure Hodge complexes and intersection complexes (the coefficients are the pulled-back coefficients).
The desired conclusions (\ref{sabB}) follows by: taking 
Lemma \ref{m0}; applying the K\"unneth formula keeping track of the perversities;
organizing the summands.  The identities (\ref{sab0B}) follow from the
isomorphisms $\Lambda^{0}_{\Delta \times \Delta} \simeq\Lambda^{8}_{\Delta \times \Delta}
(4) \simeq \rat_{{\ogbp}^o \times {\ogbp}^o}$ (use (\ref{plds}) and the irreducibility of the fibers) and the fact that $\ogbp \times \ogbp$ is nonsingular.
The identities  (\ref{tassB}) follow from the natural   isomorphism (\ref{plds})
$\Lambda^1_{\Delta \times \Delta} \simeq \Lambda^7_{\Delta \times \Delta} (3)$ (the general fiber
of the group scheme is the product of the Jacobians of two nonsingular projective curves of genus two), as well as
(\ref{z1}), with the shift changed from $[2]$ to $[4]$ to take into account
that we are pulling back along  smooth morphisms, the projections, of relative dimension $2$. 
\end{proof}

\subsection{The Ng\^o strings  for the Lagrangian fibration
$\pogmp^{(2)}: {\rm Sym}^2 \ogmp   \to  {\rm Sym}^2\ogbp$}\la{kt2}$\;$

The morphism $\ogmp^2: \ogmp^2 \to \ogbp^2$ is equivariant for the action of the symmetric group in two letters, so that it induces the natural  morphism $\pogmp^{(2)}: {\rm Sym}^2 \ogmp   \to  {\rm Sym}^2\ogbp$
on the two-fold symmetric products. 
We have the commutative diagram:
\beq\la{symsym}
\xymatrix{
\ogmp \times \ogmp \ar[r]^-{\pogmp^2} \ar[d]^-a  \ar[rd]^-c& \ogbp \times \ogbp \ar[d]^-b \\
{\rm Sym}^2 \ogmp \ar[r]^-{\pogmp^{(2)}}  &{\rm Sym}^2 \ogbp.
}
\eeq

Recall that we have the the local systems
$\Lambda^\bullet_{\Sigma}$ (\ref{norm}) and $\ms{L}^o$  (\ref{ldf})  defined on $\Sigma^o={\rm Sym}^2 \ogbp^o \setminus \Delta_{\ogbp}$
and the strings  $\ms{S}^\pm_\Sigma$  
 defined in (\ref{sics}).
\begin{pr}\la{dtforsym}

The Decomposition Theorems for the morphisms $c$ and $\pogmp^{(2)}$ take the following form:
\beq\la{idbB}
Rc_* \rat_{\ogmp \times \ogmp}  = \ms{S}^+_\Sigma \oplus \ms{S}^-_\Sigma;
\qquad
R\pogmp^{(2)}_* \rat_{{\rm Sym}^2 \ogmp}  \simeq \ms{S}^+_\Sigma. 
\eeq
Moreover, 
we have that the following terms  are sheaves: 
\beq\la{z6B}
\begin{aligned}
&\ms{IC}_\Sigma (\Lambda^{4\pm 4}_{\Sigma})=\rat_\Sigma   (2\pm2),  \quad 
& \ms{IC}_\Sigma (\Lambda^{4\pm 4}_{\Sigma} \otimes \ms{L}^o) &={i_{\Sigma^o}}_* \ms{L}_{\ogbp} 
(2\pm 2) = {i_{\Sigma^o}}_! \ms{L}_{\ogbp} (2\pm 2),
\\
& \ms{IC}_\Sigma \left(\Lambda^{4\pm 3}_{\Sigma}\right)
= {i_{\Sigma^o}}_* \Lambda^{4\pm 3}_{\Sigma} (3\pm 3),     \quad 
& \ms{IC}_\Sigma (\Lambda^{4\pm 3}_{\Sigma} \otimes \ms{L}^o) &= 
{i_{\Sigma^o}}_* \left(\Lambda^{4\pm 3}_{\Sigma} \otimes \ms{L}^o \right) (3 \pm 3).
\end{aligned}
\eeq
\end{pr}
\begin{proof} 
Denote by $\ms{L}_{\ogmp}$ the rank one local system on ${\rm Sym}^2 \ogmp \setminus \Delta_{\ogbp}$
with monodromy $-1$ around the diagonal; it is pure of Hodge-Tate type, with weight zero.
We have $Ra_* \rat_{\ogmp \times \ogmp} = a_* \rat_{\ogmp \times \ogmp} =
\rat_{{\rm Sym}^2 \ogmp} \oplus ({i_{{\rm Sym}^2 \ogmp \setminus \Delta_{\ogmp}}})_* \ms{L}_{\ogmp}$.
It follows that: 
\beq\la{rosso}Rc_* \rat_{\ogmp \times \ogmp}  = R\pogmp^{(2)}_* \rat_{{\rm Sym}^2 \ogmp} \oplus 
R\pogmp^{(2)}_*{i_{{\rm Sym}^2 \ogmp \setminus \Delta_{\ogmp}}}_* \ms{L}_{\ogmp}.
\eeq

By the functoriality of the direct image, the fact that $b$ is finite and (\ref{sabB}), we have $Rc_* \rat_{\ogmp \times \ogmp}  \simeq Rb_*   R\pogmp^{2}_*
\rat_{\ogmp \times \ogmp}= b_* (\ms{S}_\Delta \boxtimes \ms{S}_\Delta)$. Note
that $\ms{S}_{\Delta}$, being the direct sum  of shifted intersection complexes supported exactly on $\ogbp$,
is determined, via the intermediate extension functor, by its restriction to any dense open subset
of $\ogbp$. The analogous fact remains true for $\ms{S}_\Delta \boxtimes \ms{S}_\Delta$
on $\ogbp \times \ogbp$. It follows that the analogous fact
remains true for $b_*(\ms{S}_\Delta \boxtimes \ms{S}_\Delta) \simeq 
R\pogmp^{(2)}_* \rat_{{\rm Sym}^2 \ogmp} \oplus 
R\pogmp^{(2)}_*{i_{{\rm Sym}^2 \ogmp \setminus \Delta_{\ogmp}}}_* \ms{L}_{\ogmp}$
on ${\rm Sym}^2 \ogbp$: 
in fact,  the morphism $b$ being finite and surjective, the  direct image $b_*$ sends intersection complexes
supported exactly  on $\ogbp^2$ to ones supported exactly  on ${\rm Sym}^2 \ogbp$.

It is thus enough to verify the desired assertions (\ref{idbB}) concerning (\ref{rosso})
on any Zariski-dense open subset
$U$ of ${\rm Sym}^2 \ogbp$. On any such open subset $U$,  (\ref{rosso})
takes the form:
\beq\la{lart} 
\left(R\pogmp^{(2)}_* \rat_{{\rm Sym}^2 \ogmp}\right)_{| U} \oplus 
\left( R\pogmp^{(2)}_*{i_{{\rm Sym}^2 \ogmp \setminus \Delta_{\ogmp}}}_* \ms{L}_{\ogmp}\right)_{| U}.
\eeq

We set $U:= \Sigma^o:= {\rm Sym}^2 \ogbp^o \setminus \Delta_{\ogbp^o}$. 
In view of Remark \ref{onl}.(3),  recalling that $\ms{L}$ is defined on ${\rm Sym}^2 \ogbp \setminus 
\Delta_{\ogbp}$,  and that $\ms{L}_{\ogbp} ={\pogmp^{(2)}}^*\ms{L}$, and by using the projection formula,
 we have that (\ref{lart}) reads as follows:
\beq\la{lartb} 
\left(R\pogmp^{(2)}_* \rat_{{\rm Sym}^2 \ogmp}\right)_{| \Sigma^o} \oplus 
\left( 
\left(R\pogmp^{(2)}_* \rat_{{\rm Sym}^2 \ogmp}\right)_{| \Sigma^o}
\otimes
\ms{L}_{|\Sigma^o}
\right).
\eeq

 The proper morphism 
$\pogmp^{(2)}$ is smooth over $\Sigma^o$, so that, by Deligne's theorem \ci[Theorem 1.5.3]{depcmi}, i.e. the Decomposition Theorem for proper smooth morphisms, we have that:
\beq\la{betoee}
\left(R\pogmp^{(2)}_* \rat_{{\rm Sym}^2 \ogmp}\right)_{| \Sigma^o} \simeq
\bigoplus_{\bullet =0}^8 
\left(R^\bullet\pogmp^{(2)}_* \rat_{{\rm Sym}^2 \ogmp}\right)_{| \Sigma^o} [-\bullet]
\eeq
Now, the rhs of (\ref{betoee}) coincides with ${\ms{S}_{\Sigma}^+}_{|U}$:
this follows from the fact that the morphism $(\pogmp^{(2)})_{|\Sigma^o}$ is a torsor
for the group scheme (\ref{epigs}) $({\rm Pic}^0_{\widehat{\m{C}_{\Sigma^o}}/\Sigma^o})$.
This can also be seen to follow directly  from the Ng\^o String Theorem \ref{nstabis}. 

We  have thus proved
  (\ref{idbB}).

All the summands in $\ms{S}^+_\Sigma$ and $\ms{S}^-_\Sigma$ have been shown to be direct summands
of terms of the form $b_* \ms{IC}_{\ogbp^o \times \ogbp^o} (\Lambda^\bullet_{\Delta \times \Delta})$.
Since the direct image functor $b_*$ sends sheaves to sheaves, in view of (\ref{sabB}) and (\ref{tassB}),  the  intersection complexes 
in (\ref{z6B}) are shifted sheaves, and thus have the predicated form. 
\end{proof}

\subsection{Ng\^o strings  for the Lagrangian fibrations $\ogm, \ogn  \to \ogb $}\la{dtmn}$\;$

This section is devoted to proving Proposition \ref{dtwe}.   While it falls falls short of establishing the exact shape of the Decomposition Theorem
for the complexes $R\pogm_* \rat_{\ogm}$ and $R\pogn_* \rat_{\ogn}$,
the fact that  this shortcoming is measured by the same integer $\e$  in both
expressions (\ref{dtmnt}),  is key to proving our main Theorem A in \S\ref{pmt}.

Before embarking in the proof of Proposition \ref{dtwe}, we list some facts we need.

\begin{fact}\la{1e2}$\;$

\ben
\item
By the very definition (\ref{sics}),  the cohomology sheaf in degree $10$  of the string $\ms{S}_\ogb$
 takes the form:
$\m{H}^{10} (\ms{S}_{\ogb})=
\rat_B \oplus \oplus_{i=1}^4 \m{H}^i(\ms{IC}_B (\Lambda^{10-i}_{\ogb^o}))
$, where the direct sum ends at $4$ by the conditions of support for intersection complexes. These same conditions
tell us that $\dim {\rm support}\,
\m{H}^i(\ms{IC}_B (\Lambda^{10-i}_{\ogb^o})) \leq 4-i$, for $1 \leq i \leq 4$.
The Relative Hard Lefschetz  isomorphism $\Lambda^1_{\ogb^o}=\Lambda^9_{\ogb^o}(4)$, combined with Lemma \ref{dtbo} (cf. (\ref{z2})),
implies that $\m{H}^1(\ms{IC}_B (\Lambda^{9}_{\ogb^o}))=0$. In summary,  we have:
\beq\la{kio}
\m{H}^{10} (\ms{S}_{\ogb})= \rat_B 
\oplus  
{^{\leq 2}\m{H}^2(\ms{IC}_B (\Lambda^{8}_{\ogb^o}))}
\oplus  
{^{\leq 1}\m{H}^3(\ms{IC}_B (\Lambda^{7}_{\ogb^o}))}
\oplus  
{^{\leq 0}\m{H}^4(\ms{IC}_B (\Lambda^{6}_{\ogb^o}))},
\eeq
where the exponents on the left are the upper bounds on the dimensions of the respective supports.

\item
In the same vein, but by using Proposition \ref{dtforsym} in place of Lemma \ref{dtbo},
we have that the cohomology sheaf in degree $8$  of the string $\ms{S}^+_\Sigma$
-i.e. in degree $0$ for $S^+_{\Sigma}$- takes the form:
\beq\la{kiob}
\m{H}^{8}(\ms{S}^+_{\Sigma}) = \rat_\Sigma 
\oplus  
{^{\leq 1}\m{H}^2(\ms{IC}_\Sigma (\Lambda^{6}_{\Sigma^o}))}
\oplus  
{^{\leq 0}\m{H}^3(\ms{IC}_B (\Lambda^{5}_{\Sigma^o}))}.
\eeq
where the exponents on the left are the upper bounds on the dimensions of the respective supports.
This is improved in  (\ref{nip}).
\item
Similarly, the cohomology sheaf in degree $8$  of the string $\ms{S}^-_\Sigma$ takes the form:
\beq\la{kiob-}
\m{H}^{8}(\ms{S}^-_{\Sigma}) = {i_{\Sigma \setminus \Delta}}_* \ms{L} 
\oplus  
{^{\leq 1}\m{H}^2(\ms{IC}_\Sigma (\Lambda^{6}_{\Sigma^o} \otimes \ms{L}^o))}
\oplus  
{^{\leq 0}\m{H}^3(\ms{IC}_B (\Lambda^{5}_{\Sigma^o}\otimes \ms{L}^o))},
\eeq
where the exponents on the left are the upper bounds on the dimensions of the respective supports.
This is  improved to (\ref{kiob-p}).

\item
In the same vein, but by using Lemma  \ref{m0} in place of Lemma \ref{dtbo}, coupled with the fact,
due to the conditions of support, that
$\m{H}^2 \ms{IC}(\frak L) =0$ for any system of coefficients $\frak L$ on a surface,
we have that the cohomology sheaf in degree $4$  of the string $\ms{S}_\Delta$ reads as follows:
\beq\la{kioboo}
\m{H}^{4}(\ms{S}_{\Delta}) = \rat_\Delta 
\eeq

\item
The following is clear: if $T$ is an irreducible variety and $F\simeq \rat_T^{\oplus r}$ is a constant sheaf 
on $T$, then $F$ does not admit a direct summand supported on any  proper  subvariety.

\item The local system $\ms{L}$ on $\Sigma \setminus \Delta$ of Proposition  \ref{rtop}
satisfies the following properties   Fact \ref{onl} (cf. the proof of Proposition 
\ref{dtbo}): (i) ${i_{\Sigma \setminus \Delta}}_* \ms{L} = 
{i_{\Sigma \setminus \Delta}}_! \ms{L}$ (because there are no local monodromy invariants around the diagonal);
(II) for every Zariski dense open subset $j: U \subseteq  \Sigma \setminus  \Delta$, we have that
$\ms{L} = j_* j^*\ms{L}$ 
 (true for any local system, due to the normality of the varieties involved);
 (IIi) ${i_{U}}_* (\ms{L}_{| U}) =  {i_{\Sigma \setminus \Delta}}_* \ms{L}$ on $\ogb$ (follows immediately from (II)
 by functoriality).
\een

\end{fact}

\begin{pr}\la{dtwe}
The Decomposition Theorem for $R\pogm_* \rat_{\ogm},  R\pogn_* \rat_{\ogn}$ in $D^b MHM_{alg} (\ogb)$ takes  the following form: 
\beq\la{dtmnt}
\begin{aligned}
R\pogm_* \rat_{\ogm} \simeq & \ms{S}_B \bigoplus \ms{S}^+_\Sigma [-2] (-1)
 \bigoplus \ms{S}_\Delta^{\oplus (1+\e_{\ogm})}[-6] (-3) , \\
R\pogn_* \rat_{\ogn} \simeq & \ms{S}_B \bigoplus \ms{S}^-_\Sigma [-2]  (-1) \bigoplus 
\ms{S}_\Delta^{\oplus \e_{\ogn}}[-6] (-3),
\end{aligned}
\eeq
where $\e_{\ogm}=\e_{\ogn}=0$, or $1$.
\end{pr}
\begin{proof}
By virtue of Proposition \ref{tadaw} about  $\delta$-regularity, we can apply Theorem \ref{nstabis} on Ng\^o strings.

We first deal with $R\pogm_* \rat_{\ogm}$. 

By Corollary \ref{dtbr}, the subvarieties $B,\Sigma$ and $\Delta$
are supports; in particular, each one of them must contribute the associated Ng\^o string as a direct summand of 
 $R\pogm_* \rat_{\ogm}[10]$.
 
By Proposition \ref{rtop}, the local systems of type  $L$ appearing in Theorem \ref{nstabis} applied to each of these three supports
are constant of some strictly positive ranks $r_{\ogm, \ogb}, r_{\ogm,\Sigma}, r_{\ogm,\Delta}$, and the same is true after the push-forward (\ref{dirsumdbis}).

It follows that the  direct sum of the Ng\^o strings associated with these three supports takes the following form:
\beq\la{am}
\ms{S}_\ogb^{\oplus r_{\ogm, \ogb}} \oplus {\ms{S}_\Sigma^+}^{\oplus r_{\ogm, \Sigma}}
[-2] (-1)  \oplus \ms{S}_\Delta^{\oplus r_{\ogm,\Delta}} [-6] (-3) ,
\eeq 
so that, according to Fact \ref{1e2}(1,2,4),  the combined contribution to the direct image
 $R^{10} \pogm_* \rat_{\ogm}$ stemming from the local systems of type $L$ for these three Ng\^o strings
 takes the form: 
\beq\la{am1}
\left( \rat_{\ogb}^{\oplus r_{\ogm, \ogb}} \oplus  \rat_\Sigma^{\oplus r_{\ogm, \Sigma}} \oplus 
\rat_\Delta^{\oplus r_{\ogm, \Delta}}\right)  (-5).
\eeq
Proposition \ref{rtop} implies that $r_{\ogm, \ogb}=r_{\ogm, \Sigma}=1$ and that $1 \leq r_{\ogm, \Delta} \leq 2$,
Moreover: 

\ben
\item
$r_{\ogm ,\Delta} =2$ if an only if the combined contribution of the two  Ng\^o strings $\ms{S}_\ogb$ and $\ms{S}_\Sigma^+$ to ${R^{10}_{\ogm}}_{|  \Delta }$,
i.e. the direct sum of the l.h.s. of  (\ref{kio}) and (\ref{kiob}) restricted to $\Delta$,   is
exactly $\rat_\Delta^{\oplus 2}$;

\item
$r_{\ogm ,\Delta} =1$   if an only if  the above contribution is exactly $\rat_\Delta^{\oplus 3}$.
\een

We define $\e_{\ogm}$ so that $1+\e_{\ogm} =r_{\ogm,\Delta}$:
\beq\la{rt}
\e_{\ogm}  := r_{\ogm,\Delta} -1.
\eeq
The l.h.s. of (\ref{dtmnt}) has been established.

Before studying  $R\pogn_* \rat_{\ogn}$,  we note in passing that the above analysis, Fact \ref{1e2}.(1,2,5) and Proposition \ref{rtop} imply the following improvements of (\ref{kio}) and  of  (\ref{kiob}):
\beq\la{nip}
\m{H}^{10} (\ms{S}_{\ogb})= \rat_B
\oplus  
\m{H}^2(\ms{IC}_B (\Lambda^{8}_{\ogb^o})),
\quad
\m{H}^2(\ms{IC}_B (\Lambda^{8}_{\ogb^o})) \simeq \rat_{\Delta}^{\oplus \e^{2,8}_{\ogb}},
\quad 
\m{H}^{8}(\ms{S}^+_{\Sigma}) = \rat_\Sigma,
\eeq
where $\e^{2,8}_{\ogb} = 1- \e_{\ogm}$. Note that 
$r_{\ogm,\Delta}=1,2$, and $r_{\ogm,\Delta}=2$ iff $\e_{\ogm}=1$ iff $\e^{2,8}_{\ogb}=0$.

We now study $R\pogn_* \rat_{\ogn}$.  

By virtue of Proposition \ref{tadaw} about  $\delta$-regularity, we can apply
Theorem \ref{nstabis} on Ng\^o strings.
By Proposition \ref{rtop}, we have that ${R^{10}_{\ogn}}_{|\ogb \setminus \Sigma} \simeq \rat_{\ogb \setminus \Sigma} (-5)$, which, jointly with  Corollary \ref{nstabis} on the shape of Ng\^o strings and the associated local system contributing to $R^{10}_{|\ogb \setminus \Sigma}$, implies that $(R\pogn_* \rat_{\ogn})_{|\ogb \setminus \Sigma} \simeq (\ms{S}_\ogb)_{|\ogb \setminus \Sigma}$. It follows that the Ng\^o string for $R\pogn_* \rat_{\ogn}$ associated with $\ogb$ is $\ms{S}_\ogb$. Moreover, any additional support must be contained in $\Sigma$. 

As shown earlier, the contribution of $\ms{S}_{\ogb}$ to $R^{10}_{\ogn}$ is
$
\rat_B 
\oplus \m{H}^2(\ms{IC}_B (\Lambda^{8}_{\ogb^o})),
$
where the second summand is supported on $\Delta$.  In particular, this contribution restricted to $\Sigma \setminus 
\Delta$ is $\rat_{\Sigma \setminus \Delta}$,   which does not yield the whole 
${R^{10}_{\ogn}}_{| \Sigma \setminus \Delta} = (\rat_{| \Sigma \setminus \Delta}  \oplus \ms{L})(-5)$, as it is prescribed by Proposition 
\ref{rtop}.

It follows  that $\Sigma$ is a support and that the local system of type $L$,
defined on a suitable  Zariski dense open subset of $\Sigma$,
prescribed by Theorem \ref{nstabis} is the restriction of $\ms{L}$ to this open subset.
It follows, also by virtue of Fact \ref{1e2}.(6),  that the the associated Ng\^o string   is $\ms{S}^-_{\Sigma}[-2](-1)$.

It also follows that the only other possible Ng\^o strings are supported at closed subvarieties of $\Delta$.
Since ${R^{10}_{\ogn}}_{|\Delta} \simeq \rat_{\Delta}^{\oplus 2}(-5)$,  Corollary \ref{nstabis}
and Fact \ref{1e2}.(6) imply that the only remaining possible support is $\Delta$.

In particular, the only possible additional Ng\^o string is $S_{\Delta}^{\oplus \e_{N} }$, 
with $\e_{N}=0,1$. 

By combining  (\ref{kiob-}), the first equality in (\ref{nip}),  and the sentence following (\ref{nip}), we see that 
$\e_{N}=1$ 
iff $\e^{2,8}_B=0$ iff $\e_{\ogm}=1$.
In particular, $\e_{N}=\e_{\ogm}$ and the r.h.s. of (\ref{dtmnt}) follows, with the same value of $\e$ on both sides.

The proposition is proved.
\end{proof}

\begin{rmk}\la{imrq}
In analogy with (\ref{nip}), we note that  the analysis carried out so-far implies that (\ref{kiob-}) can be improved to: 
\beq\la{kiob-p}
\m{H}^{8}(\ms{S}^-_{\Sigma}) = {i_{\Sigma \setminus \Delta}}_* \ms{L}. 
\eeq
\end{rmk}

\section{Proofs of Theorems A, B and B\texorpdfstring{$'$}{'}}\la{pfs ab}

We are now ready to combine the results  proved so far with Theorem  \ref{nstabis}, the refined version of Ng\^o's Support Theorem. We prove Theorem A, then Theorem B and B'. The statement of Theorem B implies the statement of Theorem A. We chose however to first prove Theorem A and then Theorem B, as the extra layer given by the Hodge structures may make the topological arguments yielding Theorem A less clear. We then prove Theorem B, which implies Theorem B' thanks to a density argument using the period mapping.

\subsection{Proof of the main Theorem A}\la{pmt}$\;$

\begin{fact}\la{fattoka}
Let $\m{A}$ be a semisimple Abelian category where every object has finite length and the isomorphism
classes of simple objects form a set $\frak{S}$. Recall that by definition the zero object is not simple. Every object $a \in \m{A}$ is isomorphic
to a unique finite direct sum of simple objects with multiplicities:
 \beq\la{fdswm}
 a \simeq \bigoplus_{\frak s \in \frak S} \frak s^{\oplus n_{\frak s} (a)}.
 \eeq
 If we have an identity $[a] = [b]-[c]$ in the Grothendieck group $K(\m{A})$, with $a,b, c \in \m{A}$,
 then 
 \beq\la{gkokg-}
 n_{\frak s}(a) = n_{\frak s}(b) - n_{\frak s}(c).
 \eeq
Let  $\phi: {\rm Obj}(\m{A}) \to \frak M$  be  an assignment into a  commutative group which  is additive in exact sequences, then, if $[a]=[b]-[c]$ are as above, then 
\beq\la{gkokg}
\phi (a) =  \phi (b)- \phi (c).
 \eeq
\end{fact}

In the remainder of this section, we let $\rat$GPPHS  be the category of rational graded polarizable
pure Hodge structures, we let  $b_{*}$ denotes the graded dimension of the corresponding graded vector spaces (Betti numbers), and we let $h^{\bullet \star}$ denote the bi-graded dimension of the corresponding bi-graded vector spaces (Hodge numbers).  If we set $\m{A}=\rat$GPPHS and $\phi=b_*, h^{\bullet \star}$, then 
we  are in the situation of Fact \ref{fattoka}, with  $\frak M = \zed^{\nat}, ({\zed^2})^\nat$, respectively.

\begin{pr}\la{okkay} {\rm  ({\bf Cut and paste of polarizable graded pure Hodge structures})}
\beq\la{okayb}
b_*(\ogm)=b_{*}(\ogn)  +
b_{*}\left( ({\rm Sym}^2 \ogmp)^{\oplus 2} [-2]\right) -
b_{*}\left( ({\ogmp}^2)[-2]\right) + b_{*}\left( (\ogmp)[-6]\right);
\eeq
\beq\la{okayh}
h^{\bullet \star}(\ogm) = h^{\bullet \star}(\ogn)  +
h^{\bullet \star}\left( ({\rm Sym}^2 \ogmp)^{\oplus 2} [-2](-1)\right) -
h^{\bullet \star}\left( ({\ogmp}^2)[-2](-1)\right) + h^{\bullet \star}\left( (\ogmp)[-6](-3)\right).
\eeq
\end{pr}
\begin{proof}
According to Proposition \ref{dtwe},  to Remarks \ref{rovo} and \ref{nondipende},  we have isomorphisms of 
finite dimensional rational graded vector spaces, even of rational polarizable graded pure Hodge structures:
\beq\la{lasting}
\xymatrix{
H^*(\ogm) &\simeq& H^*(\ms{S}_B) \oplus 
H^*(\ms{S}_\Sigma^+)[-2](-1) \oplus H^*(\ms{S}_\Delta)^{\oplus 1 + \e}[-6](-3),\\
H^*(\ogn) &\simeq& H^*(\ms{S}_B) \oplus 
H^*(\ms{S}_\Sigma^-)[-2](-1) \oplus H^*(\ms{S}_\Delta)^{\oplus  \e}[-6](-3),
\\
H^* ({\ogmp}^2) &\simeq & H^*(\ms{S}_\Sigma^+) \oplus  H^*(\ms{S}_\Sigma^-).
}
\eeq
By working  with the Abelian category  of finite dimensional graded vector spaces, or even of rational polarizable graded pure Hodge structures, we obtain the identities in the corresponding Grothendieck groups:
(note the important fact that
the $\e$'s cancel out) 
\beq\la{k1122}
H^*(\ogm)  = H^*(\ogn)  + \left(2 H^*({\rm Sym}^2 \ogmp) - H^* \left({\ogmp}^2\right)\right) [-2](-1) 
+ H^* \left(\ogmp\right)[-6](-3).
\eeq

It remains to apply the identity (\ref{gkokg}).
\end{proof}

\medskip

{\bf Proof of Theorem A.} It is enough to  compute the Betti and Hodge number of $\ogm$.
%The Betti and Hodge numbers of a manifold in the deformation class $OG10$ can be computed on any 
%element in the class. We do so by choosing $\ogm$ as a representative.

The Betti and Hodge numbers of all the varieties appearing on the r.h.s. of (\ref{okayb}) and  
(\ref{okayh}), namely, of $\ogn, \ogmp, \ogmp^2$, and of ${\rm Sym}^2 \ogmp$ are known thanks to G\"ottsche's and Macdonald's formulae \cite{Gottsche}. The odd Betti numbers of these varieties are zero, and their Hodge numbers in even degree  are as follows.
The Hodge numbers of $\ogmp$:
\beq \label{hpq M'}
\begin{array}{ccccc}
 && 1 && \\
    & 1 & 21   & 1 &     \\
1 & 21 & 232 & 21 & 1 \\
\end{array}
\eeq
From the Hodge diamond of $\ogmp$, a direct computation gives the  Hodge numbers of $\ogmp^2$:
\beq \label{hpq M'2}
\begin{array}{ccccccccc}
&&& &1&&&&\\
&&& 2 & 42   & 2 &&&\\
 && 3 & 84 & 907 & 84 & 3 && \\
    & 2 & 84   & 1350 & 9870 & 1350   & 84   & 2 &     \\
1 & 42 & 907 & 9870 & 55596 & 9870 & 907 & 42 & 1. \\
\end{array}
\eeq
and those of $\Sym^2\ogmp$:
\beq \label{hpq sym2M'}
\begin{array}{ccccccccc}
&&& &1&&&&\\
&&& 1 & 21   & 1 &&&\\
 && 2 & 42 & 464 & 42 & 2 && \\
    & 1 & 42   & 675 & 4935 & 675   &42   & 1 &     \\
1 & 21 & 464 & 4935 & 27914 & 4935 & 464 & 21 & 1. \\
\end{array}
\eeq
The Hodge numbers of $\ogn$ can be recovered from the generating series of \cite{Gottsche}\footnote{We used the script found at \url{https://pbelmans.ncag.info/blog/2017/11/18/hodge-numbers-hilb/} to compute theses numbers.}:
\beq \label{hpq N}
\begin{array}{ccccccccccc}
&&& & & 1 &  &&&&\\
&&& &1 & 21 & 1 &&&&\\
&&& 1 & 22   & 254 & 22 & 1 &&&\\
 && 1 & 22 & 276 & 2277 & 276  & 22 &1 && \\
    & 1 & 22 & 276 & 2530 & 16469 & 2530   & 276 & 22 & 1 &     \\
1 & 21 & 254 & 2277 & 16469 & 87560 & 16469 &2277 &254 & 21 & 1 . \\
\end{array}
\eeq
In view of (\ref{k1122}), Theorem A now follows by  direct calculation.

\begin{rmk}\la{shenandyin} 
The proof of Theorem A combines the determination of the N\^go strings (Proposition \ref{dtwe}) in $D^b MHM_{alg} (\ogb)$ with a cancelation occurring in a Grothendieck group (Proposition \ref{okkay}). J. Shen and Q. Yin have informed us that they can obtain the computation of the Hodge numbers as in Theorem A using  the validity of Proposition \ref{dtwe} in $D^b (\ogb, \rat)$ (which can be considered the main technical
result of this paper) and without invoking its validity at the level of mixed Hodge modules, and combine it
with the main result in  \ci{shen-yin}. This approach can then replace the last part of the proof of Theorem A given above.
%The proof of Theorem A combines the determination of the N\^go strings (Proposition \ref{dtwe}), with a cancelation occurring in a Grothendieck group
%(Proposition \ref{okkay}).
%J. Shen and Q. Yin have informed us that one can use  Proposition \ref{dtwe} and  obtain a proof of the Hodge-theoretic statement
%of  Theorem A parallel to the one given above that does not use mixed Hodge modules when dealing with Proposition
%\ref{okkay}.(\ref{okayh}), but that rather makes use of the main result in \ci{shen-yin}, 
%namely that the dimensions of the graded spaces $Gr_i^PH^{i+j}=h^{ij}$ for the total space of a projective hyperkahler Lagrangian fibration 
%{\cre sostituirei namely ... con namely that the perverse Hodge numbers 
%of a projective hyperkahler Lagrangian fibration equal the corresponding Hodge numbers of the total space of the
%projective hyperkahler Lagrangian fibration.}
%. Note that this approach does not seem to be  useful towards Theorem B.
\end{rmk}

\subsection{The main Theorem B$'$}\la{pmtB}$\;$

Using Remarks \ref{caso0m0} and \ref{non general} we can formulate the following, slightly more general, version of Theorem B:

\begin{thmBp} \label{thmbprimo}
Let $(\ks,H)$ be a polarized  $K3$ surface of genus $2$. Let  $w=(0,2H,\chi)$,  $\chi$ odd, and $2v'=(0,2H,2\chi')$ be Mukai vectors (cf. (\ref{2vprimo}) and (\ref{w})) which are positive in the sense of Remark \ref{caso0m0}  .
Let $L$, resp. $L'$, be a polarization which is $w$-generic, resp. $2v'$-generic (cf. Remark \ref{non general}).
Finally, let $\ogn=M_{w, L}(\ks)$ be the moduli space of $L$-stable sheaves on $\ks$ with Mukai vector $w$ and let $\ogm= \widetilde M_{2v', L'}(\ks)$ be a symplectic resolution of the moduli space $M_{2v', L'}(\ks)$ of $L'$-semistable sheaves on $\ks$ with Mukai vector $2v'$. Then, using the notation as in the statement of Theorem B, we have isomorphisms:
\beq\la{motivep}
\begin{aligned}
\ogm = & \,\,  \ks^{(5)} \oplus \left[\ks^{(4)} \langle -1 \rangle \right]^{\oplus 2}
\oplus \mathbb S_{(2,2)} (\ks) \langle -1 \rangle \oplus \left[ \ks^{(3)} \langle -2 \rangle \right]^{\oplus 2} \oplus \\
& \oplus  \left[ \mathbb S_{(2,1)} (\ks) \langle -2 \rangle \right]^{\oplus 2}
\oplus \left[\ks \otimes \ks \right]  \langle -3 \rangle \oplus \left[ \ks^{(2)} \langle -3 \rangle \right]^{\oplus 3} \oplus \left[ \ks \langle -4 \rangle\right]^{\oplus 2};
\end{aligned}
\eeq
\beq\la{motivenp}
\xymatrix{
N=  S^{(5)} \oplus \left[S^{(3)}\otimes S \right] \langle -1 \rangle 
 \oplus \left[S \otimes S^{(2)}\right]^{\oplus 2} \langle -2 \rangle
\oplus \left[S^2\right]^{\oplus 2} \langle -3 \rangle \oplus S\langle -4 \rangle.
}
\eeq
\end{thmBp}

The graded pure polarizable Hodge structures of all the varieties on the r.h.s. of (\ref{k1122}) are known
and can be expressed as follows in terms of the Hodge structure of the underlying K3 surface $\ks$.
Let $V=H^*(S)$ be the rational Hodge structure of $\ks$. Using (\ref{notm}), Lemma \ref{lemma birazionali}, as well as \cite{Gottsche Sorgel} and MacDonald's formula, we can write:
\beq \label{hs M'}
H^*(\ogmp)=V^{(2)}\oplus V[-2](-1)\\
\eeq
\beq \label{hs tutte}
\begin{aligned}
&H^*(\ogmp^2)=V^{(2)} \otimes V^{(2)}\oplus 2V^{(2)} \otimes V [-2](-1)\oplus V \otimes V [-4](-2)\\
&H^*(\Sym^2 \ogmp)=\Sym^2 V^{(2)}\oplus V^{(2)} \otimes V [-2](-1)\oplus V^{(2)}[-4](-2)\\
&H^*(N)=V^{(5)}\oplus V^{(3)}\otimes V[-2](-1)\oplus 2V \otimes V^{(2)}[-4](-2)\oplus 2V \otimes V [-6](-3) \oplus  V[-8](-4).
\end{aligned}
\eeq

For a partition $\lambda$  of an integer $k$, we denote by $\mathbb S_{(\lambda)}(-)$ the corresponding Schur functor. For the notation and basic properties of Schur functors we refer the reader to \cite[Ch. 6]{Fulton-Harris}. Let $W$ be a vector space. By \cite[Thm 6.3]{Fulton-Harris}, the $GL(W)$--representation $W^{\otimes k}$ decomposes into a sum of irreducible subrepresentations, each of which is a Schur module $\mathbb S_{(\lambda)}(W)$. Moreover, if $W$ is a $\mathbb Q$--vector space, so are all of its Schur modules. It is well known that if $W$ is a $\mathbb Q$--vector space endowed with a polarizable Hodge structure, then $W^{\otimes k}$, as well as each of its irreducible representations, inherit compatible rational polarizable Hodge structures.  This gives a rational polarizable Hodge structures on each of the Schur modules $\mathbb S_{(\lambda)}(W)$ which is such that the decomposition of $W^{\otimes k}$ into the direct sum of irreducible representations is also a decomposition into a direct sum of Hodge structures.

In what follows, we need the following Lemma.

\begin{lm} \label{lemma schur} The following isomorphisms of rational polarizable Hodge structures hold:
\[
\begin{aligned}
&V^{(2)} \otimes  V=  \mathbb S_{(2,1)}(V) \oplus V^{(3)}\\
&V^{(3)} \otimes  V=  \mathbb S_{(3,1)}(V) \oplus V^{(4)}\\
&V^{(2)} \otimes V^{(2)}= \mathbb S_{(2,2)}(V) \oplus  \mathbb S_{(3,1)}(V) \oplus  V^{(4)}\\
&\Sym^2(V^{(2)})=\mathbb S_{(2,2)}(V) \oplus V^{(4)}\\
\end{aligned}
\]
\end{lm}
\begin{proof}  By the observation on the compatibility of the Hodge structures that proceeds the lemma, it is enough to prove that these isomorphisms hold as $GL(V)$--subrepresentations $V^{(k)}$, for an appropriate integer $k$.
The first two statements follow then from the formula on page 79 of \cite{Fulton-Harris}, the second from the formula on page 81 of \cite{Fulton-Harris}, while the last statement from \cite[Ex. 6.16]{Fulton-Harris}. 
\end{proof}

Using (\ref{hs tutte}) and Lemma \ref{lemma schur}  we find the following isomorphism of rational polarizable Hodge structures:

%\beq \label{hs M'}
%H^*(\ogmp)=V^{(2)}\oplus V[-2](-1).
%\eeq
\beq \label{hs M'2}
\begin{aligned}
H^*(\ogmp^2)&=\underbrace{\mathbb S_{(2,2)}(V)\oplus \mathbb S_{(3,1)}(V)\oplus \mathbb S_{(4)}(V)}_{V^{(2)} \otimes V^{(2)}}\oplus \underbrace{V^{(3)}[-2](-1)^{\oplus 2}\oplus \mathbb S_{(2,1)}(V)[-2](-1)^{\oplus 2}}_{(V^{(2)} \otimes V [-2](-1))^{\oplus 2}}\oplus V^{\otimes 2}[-4].
\end{aligned}
\eeq
\beq \label{hs sym2M'}
\begin{aligned}
H^*(\Sym^2 \ogmp)&=\underbrace{V^{(4)}\oplus \mathbb S_{(2,2)}(V)}_{\Sym^2 V^{(2)}}\oplus \underbrace{V^{(3)}[-2](-1)\oplus \mathbb S_{(2,1}(V)[-2](-1)}_{V^{(2)} \otimes V [-2](-1)}\oplus V^{(2)}[-4](-2)
\end{aligned}
\eeq
\beq \label{hs N}
\begin{aligned}
H^*(N)&=V^{(5)}\oplus  \underbrace{V^{(4)}\oplus S_{(3,1}(V)[-2](-1)}_{V^{(3)}\otimes V[-2](-1)}\oplus \underbrace{V^{(3)}[-4](-2)\oplus \mathbb S_{(2,1}(V)[-4](-2)}_{(V \otimes V^{(2)}[-4](-2))^{\oplus 2}}\oplus \\
&\oplus V \otimes V [-6](-3)^{\oplus 2} \oplus  V[-8](-4).
\end{aligned}
\eeq

\begin{proof}[Proof of Theorem B$'$] 
We first prove the Theorem in the case when the degree two K3 surface $\ks$ is general in moduli, i.e., $\ks$ has Picard rank $1$ (this is the statement of Theorem B). Then (\ref{k1122}) holds. We thus substitute formulas (\ref{hs M'}),(\ref{hs M'2}) (\ref{hs sym2M'}) and (\ref{hs N}), into equation (\ref{k1122}) in order to express the Hodge structure of $\ogm$ in terms of that of $\ks$. Since the category of rational polarizable Hodge structures is semisimple, in view of Fact \ref{fattoka} we can make cancellations and find:
\[
\begin{aligned}
H^*(\ogm)&=V^{(5)}\oplus V^{(4)}[-2](-1)^{\oplus 2}\oplus \mathbb S_{(2,2)}(V)[-2](-1)\oplus V^{(3)}[-4](-2)^{\oplus 2}\oplus \mathbb S_{(2,1)}(V)[-4](-2)^{\oplus 2}\oplus \\
&\oplus V \otimes V [-6](-3)\oplus V^{(2)}[-6](-3)^{\oplus 3}\oplus V[-8](-4)^{\oplus 2}.
\end{aligned}
\]
This proves the theorem in case the degree two K3 surface $\ks$ has Picard rank $1$.

Before proving the general case we recall that, since birational irreducible holomorphic symplectic manifolds have isomorphic integral Hodge structures, we can prove the theorem for any holomorphic symplectic birational model of $\ogn$ or of $\ogm$. Thanks to Remark \ref{non general}, it is thus sufficient to prove the statement for an arbitrary symplectic resolution of $M_{w, H}(\ks)$ or of $M_{2v', H}(\ks)$. 

Let $(S_0,H_0)$ be a polarized K3 surface of genus $2$.  Choose a one-parameter deformation $(\m{S}, \m{H})$ of 
$(S_0,H_0)$ parameterized by a disk $(D,t_0)$ such that the subset of points $t \in D^*$ such that
$NS(S_t)$ has rank one is dense.
Let $v_t=(0, 2H_t, 2 \chi')$  be the corresponding Mukai vector.
We consider the relative moduli space $\m{M} \to D$, where $M_t = M_{v_t,H_t} (S_t)$.
By \ci[Prop 2.20]{Perego-Rapagnetta}, up to shrinking $(D,t_0)$, we may assume that, for every $t \neq t_0$, the polarization  $H_t$ is $v_t$-generic (cf. Remark \ref{non general}).
By \ci[Prop 2.23 and its proof]{Perego-Rapagnetta}, by blowing up the relative moduli space over the punctured disk $D^*$ along its singular locus results in a resolution $\w{\m{M}}_{D^*} \to \m{M}_{D^*} $ which is point-by-point $t\in D^*$ a symplectic resolution.
By Remark \ref{non general}, $M_{t_o}$ admits a symplectic resolution $\w{M}_{t_0} \to M_{t_0}$. By \ci[Thm 0.6, Cor 4.2]{KLSV}, up to passing to a branched cover of $(D,t_0)$, we may assume that
there is a smooth proper family  of projective manifolds $\w{M}_D \to D$ such that 
the restriction to $D^*$ coincides with  the family $\w{M}_{D^*} \to D^*$ introduced above and such that the central fiber is birational to $\w{M}_{t_0}$.

Consider the two period mappings \cite{Griffiths} associated with the following  two variations of (un-polarized)  graded rational  Hodge structures: the graded cohomology of the fibers of $\w{M}_D \to D$ over $D$; the one on the
r.h.s. of (\ref{motive}) as the K3 surface varies over $D$.
By the first part of the proof, these two holomorphic period mappings coincide on the dense subset of $D$ seen above where the N\'eron-Severi of the K3 has rank one, so that
the two mappings coincide on $D$.
\end{proof}

\section{Appendix: Hodge-theoretic  Ng\^o Strings}\la{vnstbis}$\;$

The goal of this appendix is to state Theorem \ref{nstabis} (MHM Ng\^o strings over $\comp$),
which is a refinement of the Ng\^o Support Theorem \ci[Th\'eor\`eme 7.2.1]{ngofl}
valid in the Hodge-theoretic context of M. Saito's mixed Hodge modules \ci{Saito 1990, Schnell on Saito}.  We do not write out a detailed proof
since, as it is explained below, it can obtained by repeating B.-C. Ng\^o's proof, with only two minor additional observations.

Theorem \ref{nstabis} admits several variants: $(a)$ varieties over  finite fields  and constructible $\oql$-adic sheaves for the \'etale topology;
varieties over the complex numbers and:  $(b')$ constructible  $\oql$-sheaves for the \'etale topology; $(b'')$ constructible $\rat$-sheaves  for the classical topology, with algebraic strata; $(b''')$, i.e. Theorem \ref{nstabis}
per se,  mixed Hodge modules for the classical topology with algebraic strata.
By the standard spreading out techniques of  \ci{bbd}[\S 6.1], version $(a)$ implies versions
$(b',b'')$.

Ng\^o support Theorem \ci[Th\'eor\`eme 7.2.1]{ngofl} is the $(a)$-version of (\ref{dirsumdbis}),
which roughly speaking, asserts that a support contributes a direct summand to the sheaf $R^{2d}f_* \rat_M$.
The paper \ci{ngofl} does not mention explicitly the complete collection of direct summands of $Rf_*{ \oql}_M$ on $S$  associated with a support $Z \subseteq S$. In fact, somewhat surprisingly, it seems a bit vague concerning this point; see  \ci[p.113,  second sentence of last paragraph]{ngofl}.  
 Even though this is not explicitly mentioned in \ci{ngofl}, it is not difficult to make the Ng\^o Support Theorem 
more complete by making precise -building in an essential way on Ng\^o's method, especially
\ci[Proposition 7.4.10]{ngofl}- all the direct summands 
supported on a given support appearing in the Decomposition Theorem for $Rf_* {\oql}_M$.  This leads to
the versions $(a,b',b'')$ of Theorem \ref{nstabis}. 

In this paper, we need Theorem \ref{nstabis}, i.e.  the Hodge-theoretic version  $(b''')$. Again, this is not difficult to achieve, by repeating B.-C. Ng\^o's proof of the support theorem by working in the
categories $D^b MHM_{alg}(-)$. There are only two minor points to consider: 

\begin{fact}\la{2minor}$\,$

\ben
\item
The weight argument \ci[p.120-121]{ngofl} involving the $\oql$-adic Tate module and Frobenius weights,
goes through on the nose  for the Tate module defined as an object in $D^b MHM_{alg}(S)$ (cf. \S \ref{drwaf} and \S\ref{gennot}):  we only need to work generically on a support, and the Tate module, being a direct image, is generically an admissible polarizable variation of mixed Hodge structures,
of weights $-1$ and $0$.

\item
The ``Libert\'e" statement  \ci[Proposition 7.4.10]{ngofl}. B.-C. Ng\^o proves this  by using finite fields. Since it is a topological statement concerning a graded module over a graded algebra  in the category of 
lisse $\oql$ sheaves in a geometric context, we can use the spreading out
techniques mentioned above to carry over the ``Libert\'e" statement  from the algebraic closure of a finite field
and graded semisimple lisse $\oql$-sheaves,
to the situation over the complex numbers and graded semisimple local systems. One then promotes the
``Libert\'e"  statement from the category of graded semisimple local systems, to the category of graded polarizable variations of pure Hodge structures by using the proof of 
\ci[Lemma 7.4.11]{ngofl}, where the graded object $E$ one obtains there is automatically a graded polarizable
variation of pure Hodge structures. 
\een
\end{fact}

Let us introduce the principal players in Theorem \ref{nstabis} by first giving the set-up and then by discussing
some related objects.

\begin{setup}\la{setubis}
Let  $(M,S, P)$ be a $\delta$-regular weak abelian fibration (Definition \ref{waf}) of relative dimension $d:=
\dim (M/S)=\dim (P/S)$, and 
such that $M/S$ is projective, 
 $S$ is irreducible, and  $M$ is a rational homology manifold ($\ms{IC}_M=\rat_M$). 
 \end{setup}
 
 We have assumed $\dim (M/S)=\dim (P/S)$ only to simplify the numerology. In this paper,
 we apply  Theorem \ref{nstabis}  
 to Lagrangian fibrations, so that 
 the numerology gets simplified  even further, since then one  also has $d:=\dim{(M/S)} = \dim {(P/S)} = \dim (S)$.
 
 Recall the notation fixed in \S\ref{gennot}. In  particular, the supports of $Rf_* \rat_M$ are the integral subvarieties of $S$ that  support a non-zero  direct summand of $Rf_* \rat_M$.
 Let $\ms{A}$ be a finite set enumerating the supports  
 of $Rf_* \rat_M$, so that the supports are denoted
 $Z_\alpha \subseteq S$, $\alpha \in \ms{A}$.  Let $g_\alpha: A_\alpha \to V_\alpha$ be the Abelian scheme (\ref{ghj})
 associated with $P_{|Z_\alpha}$, where  $V_\alpha \subseteq Z_\alpha$   is a suitable Zariski open and dense
 subvariety contained in the regular locus of $Z_\alpha$, 
 and we are free to shrink it if it is useful/necessary.
 Let $\delta^{\rm ab}_\alpha$ be the relative dimension of $A_\alpha/V_\alpha$; then $\delta^{\rm ab}_\alpha =
 d- \delta(Z_\alpha)$.
 Let $\Lambda^\bullet_\alpha:= R^\bullet {g_\alpha}_* \rat_{A_\alpha}$, $0 \leq \bullet \leq 2 \delta^{\rm ab}_\alpha$; these are polarizable variations of pure Hodge structures of weight $\bullet$ on $V_\alpha$.
 Recall that  a polarizable variation of pure Hodge structures   $L$  of weight zero and of Hodge-Tate type on a nonsingular subvariety $V \subseteq S$  is determined by its underlying local system, and it gives rise to a pure Hodge module of weight $\dim (V)$   on the closure $Z$ of $V,$ namely $IC_Z (L)$.

 We can now state the main result of this appendix.

\begin{tm}\la{nstabis} {\rm ({\bf MHM Ng\^o Strings Over $\comp$})}
In the Set-up \ref{setubis},
there is an  isomorphism in $D^b MHM_{alg} (S)$ of pure objects of weight zero:
\beq\la{dtngostrbis}
Rf_*\rat_M \simeq 
\bigoplus_{\alpha \in \ms{A}} 
\left\{\bigoplus_{\bullet=0}^{2\delta^{\rm ab}_{\alpha}}
\ms{IC}_{Z_\alpha} \left( \Lambda^\bullet_{\alpha} \otimes L_\alpha  \right) \left[-\bullet\right]  \left[-2\left(d-\delta^{\rm ab}_\alpha   \right)\right] 
\left(\delta^{\rm ab}_\alpha -d\right)
\right\},
\eeq
where  $L_\alpha$ is a polarizable variation of pure Hodge structures of weight zero, and of Hodge-Tate type,
 on   $V_\alpha$ (we may need to shrink $V_\alpha$ to achieve this).

Moreover, we have:
\ben
\item
The identities:
\beq\la{basiceqbis}
%\delta (Z_\alpha) = {\rm codim}_S(Z_\alpha), \qquad 
\delta^{\rm ab}_\alpha = \dim (M/S)  - \dim (S) + \dim (Z_\alpha), \qquad \forall \alpha \in \ms{A}.
\eeq
\item
Canonical isomorphisms:
\beq\la{sdrhl0bis}
\Lambda^{\delta ^{\rm ab}_\alpha -\bullet}_{A_\alpha} =\Lambda^{\delta ^{\rm ab}_\alpha + \bullet}_{A_\alpha} (\bullet), \quad \Lambda^{\delta ^{\rm ab}_\alpha -\bullet}_{A_\alpha} \otimes L_\alpha  = \Lambda^{\delta ^{\rm ab}_\alpha + \bullet}_{A_\alpha} (\bullet) \otimes L_\alpha,
\qquad \forall   \;  0\leq \bullet \leq \delta^{\rm ab}_\alpha.
\eeq
In addition, one can realize other such isomorphisms, with associated Primitive Lefschetz Decompositions, via the 
Relative Hard Lefschetz Theorem  for the projective morphisms $g_\alpha$.

\item
The object $R^{2d}f_* \rat_M $ on $S$ admits  the following pure direct summand,
where $i_{V_\alpha}: V_\alpha \to Z_\alpha$ is the natural open immersion, and the equalities are summand-by-summand:
\beq\la{dirsumdbis}
\bigoplus_\alpha{ i_{V_\alpha}}_* L_\alpha (-d) = \bigoplus_\alpha  \ms{IC}_{Z_\alpha} (L_\alpha (-d)).
\eeq

If the general fiber  of $M/S$ is integral and if we take $Z=S$,  with $S$ normal, then ${i_{V_S}}_*L =\rat_S$.

 If all the fibers of $M/S$ are integral, then the only support is $S$ and, if $S$ is normal, then ${i_{V_S}}_*L=\rat_S$.
 \item
By taking degree $\star$ cohomology in (\ref{dtngostrbis}),  we have an isomorphism in the category of graded polarizable mixed Hodge structures (polarizable
pure Hodge structures if $S$, hence $M$, is complete):
\beq\la{rombo}
H^\star(M, \rat) \simeq \bigoplus_\alpha \left\{\bigoplus_{b=0}^{2 \delta^{\rm ab}_\alpha} 
I\!H^{\star -b -2d +2 \delta^{\rm ab}_\alpha}  (Z_\alpha, \Lambda^b_\alpha \otimes L_\alpha)  ) 
(\delta^{\rm ab}_\alpha - d)\right\}, \qquad \forall \star.
\eeq
\een

\end{tm}

\begin{proof}
Recall that $d:=\dim (M/S)= \delta(Z_\alpha) + \delta_{\alpha}^{\rm ab}$.
The identity (\ref{basiceqbis})  follows by combining the inequality in \ci[Proposition 7.2.2]{ngofl}
(which is stated and proved in the \'etale context over the closure of an algebraically closed field, and, by \ci{bbd}[\S 6.1], remains valid for the classical topology  over $\comp$)
 taken together with  the $\delta$-regularity hypothesis, which is precisely the opposite inequality. This proves (1)
 
 Note that (2) is a statement about polarizable variations of pure Hodge structures (pvphs)
 The first canonical isomorphism in (\ref{sdrhl0bis}) is standard and due to Lieberman (see \ci[\S7.4.4]{ngofl}).
 Clearly,  the second one follows from the first by twisting by any pvphs. The remaining statements in (2), concerning the Hard Lefschetz Theorem, are also standard. This proves (2).
 
 We are left with proving  (\ref{dtngostrbis}) and (3) and (4).
 %Until further notice, we  work in the constructible bounded derived category $D^b(S,\rat)$.
 Recall that we are free to shrink the Zariski-dense open subset $V_\alpha \subseteq Z_\alpha$ if useful/necessary. In what follows, we do so to meet our needs without explicit mention.

The Decomposition Theorem in $D^b MHM_{alg}(S)$ for the projective morphism $f:M\to S$ implies that there is  an isomorphism
\beq\la{bboo11}
Rf_* \rat_M \simeq \bigoplus_{\alpha \in \ms{A}} \left(\bigoplus_{i\in o^-_\alpha}^{o_\alpha^+}
{\ms IC}_{Z_\alpha} (\mathcal K_\alpha^i)[-i],
\right)
\eeq
where  the
${\mathcal K_\alpha^i}$ are  pvphs of weight $i$  on $V_\alpha$, and where $o^-_\alpha \geq 0$ marks the first occurrence of a non-zero direct summand
supported on $Z_\alpha$, and $o^+_\alpha \geq o_\alpha^-$ marks the last.

We fix $\alpha \in \ms{A}$.  

The graded object  
${\mathcal K_\alpha}:= \oplus_{i\geq 0} {\mathcal K^i_\alpha}$ 
lives in degrees in the interval $[o_\alpha^-, o^+_\alpha]$, and has graded weight zero, i.e. weight $i$ in degree $i.$

By \ci{ngofl}[Proposition 7.2.2], which  (by \ci{bbd}[\S6.1], again), remains valid in $D^b(S,\rat)$,  the local system $\mathcal K_\alpha^{o^+_\alpha}$ is a direct summand of $(R^{2d}f_* \rat_M)_{|V_\alpha}$.
It follows that $o^+_\alpha=2d$; in particular,  the last occurrence $o^+_\alpha$ is independent of $\alpha$.

By combining (\ref{bboo11}) with the symmetries stemming from Verdier Duality, we deduce easily that 
$o^-_\alpha= 2d  - 2\delta_\alpha^{\rm ab}$, i.e. the interval $[o_\alpha^-, o_\alpha^+]=
[2d - 2\delta_\alpha^{\rm ab}, 2d]$ has length twice the relative dimension $\delta_\alpha^{\rm ab}$ of the 
abelian scheme $A_{\alpha}/V_\alpha$.

The graded object 
\[\wedge_\alpha = \bigoplus_{i=-2\delta_\alpha^{\rm ab}}^0  \wedge^i_{\alpha}
:= \bigoplus_{i=-2\delta_\alpha^{\rm ab}}^0 R^{2\delta_\alpha^{\rm ab}+i}{g_\alpha}_* 
\rat_{A_{V_\alpha}}(\delta_\alpha^{\rm ab}):
\]  is concentrated in semi-negative degrees in the interval $[-2\delta_\alpha^{\rm ab},0]$;
has weight zero (semi-negative weight $i$ in degree $i$);
 has, via Poincar\'e Duality, as stalks the homology of the fibers of $A_{\alpha}/V_{\alpha}$;  it admits,
 via the Pontryagin product operation, a natural structure of graded algebra $\wedge^i_\alpha \otimes \wedge^j_\alpha \to \wedge^{i+j}_\alpha$.

The graded object of pure weight zero
 \[
 \Lambda_\alpha:=\bigoplus_{i=0}^{2\delta_\alpha^{\rm ab}} \Lambda^i_\alpha:= \bigoplus_{i=0}^{2\delta_\alpha^{\rm ab}} R^i{g_\alpha}_* \rat_{A_\alpha}
 \]  is a free graded $\wedge_\alpha$-module
 generated in degree $2\delta_\alpha^{\rm ab}$ by $\Lambda_\alpha^{2\delta_\alpha^{\rm ab}}
 \simeq \rat_{V_\alpha}(-\delta_\alpha^{\rm ab})$, via the Pontryagin product operation which, again via Poincar\'e Duality, is expressed by operations (recall that $\star \leq 0$ and $\bullet \geq 0$)
 \[\wedge_\alpha^\star \otimes \Lambda_\alpha^\bullet \to \Lambda_\alpha^{\bullet +\star}.\]
 
 By \ci{ngofl}[Proposition 7.4.10], we have that $\mathcal K_\alpha$ is
 a free graded $\wedge_\alpha$-module. In fact,  to be precise, we need to make sure that the action
 $\wedge_\alpha$ on $\mathcal K_\alpha$ (cf. \ci{ngofl}[pp. 120-121]
is defined  and free  in the category of pvphs on $V_\alpha$; for this we use  Fact \ref{2minor}, parts (1) (action) and (2) (freeness).

 Since $\mathcal K_\alpha$ has non trivial entries only in the interval $[2d- 2 \delta_\alpha^{\rm ab}, 2d]$, which
 has length $2\delta^{\rm ab}_\alpha$, we deduce that, as a graded object, we have 
 $\mathcal K_\alpha= (\Lambda_\alpha \otimes L_\alpha)[-2(d-\delta_\alpha^{\rm ab})]$, where $L_\alpha :=
{\mathcal K^{2d}_\alpha} (\delta_\alpha^{\rm ab})$. 

At this point, by taking care of the bookkeeping of weights and Tate shifts, 
(\ref{dtngostrbis}) follows, and $L_\alpha$ is as predicated since 
$L_\alpha (-d)$ 
 is  a pvphs   direct summand
of  the pvphs $(R^{2d}f_* \rat_M)_{|V_\alpha}$ which, having stalks   generated by the fundamental classes of the irreducible components of the fibers,  is pure, of Hodge-Tate type with weights $2d$. 

At this point, (3)  and (4) follows easily: the latter by taking cohomology;  the former by inspecting the contributions
of each support to the top cohomology sheaf $R^{2d}f_* \rat_M$, and recalling that normality implies uni-branch,
and this latter implies that the sheaf-theoretic direct image of the constant sheaf from a Zariski-dense open is also constant.
\end{proof}

\begin{defi}\la{defngostrbis}
We call the  $\alpha$-summands in (\ref{dtngostrbis})  the  Ng\^o strings of $Rf_*\rat_M$.
If we set:
\beq\la{ngozz}
\ms{S}_\alpha:= \bigoplus_{b=0}^{2\delta^{\rm ab}_{\alpha}}
\ms{IC}_{Z_\alpha} \left( \Lambda^b_{\alpha} \otimes L_\alpha  \right) \left[-b\right],
\eeq
then we may re-write
(\ref{dtngostrbis}) as follows:
\beq\la{reorenbis}
Rf_*\rat_M \simeq \bigoplus_\alpha \ms{S}_\alpha [-2(d - \delta^{\rm ab}_\alpha) ] (\delta^{\rm ab}_\alpha -d).
\eeq
We also call the $\ms{S}_\alpha$'s the Ng\^o strings of $Rf_*\rat_M$; these start in cohomological degree zero and are pure of weight zero.

\end{defi}

Each Ng\^o string in  (\ref{dtngostrbis}) is
 pure of weight zero and 
a direct sum of a collection of intersection complexes supported precisely at $Z_\alpha$,
placed in cohomological degree in the interval $[2d - 2\delta^{\rm ab}_\alpha, 2d]$. Up to the twist
by $L_\alpha$, the coefficients of the term indexed by $\bullet$  are the direct image $\Lambda^\bullet_\alpha=R^b {g_\alpha}_* \rat_{A_\alpha}$. These latter direct images
on $V_\alpha$ depend on the group scheme $P/S$, not on $M/S$. 
We record this fact in the following Remark which plays an important role 
 in our proof of our main Theorems A and B, where it is crucial to be able to relate the Decomposition Theorem
for the $\delta$-regular weak abelian fibration structure on   $\ogm$, to the Decomposition Theorems for the  auxiliary $\delta$-regular weak abelian fibrations we consider; see 
Lemmata \ref{m0} and \ref{m11}, and Propositions \ref{dtforsym} and \ref{dtwe}.

\begin{rmk}\la{nondipende}
As an object in $D^b MHM_{alg} (S)$, the Ng\^o string $\ms{S}_\alpha$ associated with the support $Z_\alpha$
depends only on the Abelian group scheme $A_\alpha$ and on the topological local system $L_\alpha$ (because
it can be enriched uniquely to a polarizable variation of Hodge structures of Hodge-Tate type and weight zero).
The same holds for the graded polarizable mixed Hodge structure (pure and polarizable if $Z_\alpha$ is complete)
$H^\star(Z_\alpha, \ms{S}_\alpha)$.
Clearly, the $L_\alpha$'s do depend on $M/S$.
\end{rmk}

\end{document}